\documentclass[ijoc,nonblindrev]{informs3hack2} 

\SingleSpacedXI

\usepackage{subfigure}
 
\usepackage{titlesec,soul}
\usepackage[titletoc,toc,title]{appendix}
 \usepackage{graphicx}
 \usepackage{subfigure}
 \usepackage{multirow}
 
 \usepackage{todonotes,bm}

\usepackage{thmtools} 
\usepackage{thm-restate}
\usepackage{placeins}

\usetikzlibrary{shapes,decorations.pathreplacing, calligraphy, arrows,calc,arrows.meta,fit,positioning}
\tikzset{
	-Latex,auto,node distance =1 cm and 1 cm,semithick,
	state/.style ={ellipse, draw, minimum width = 0.7 cm},
	point/.style = {circle, draw, inner sep=0.04cm,fill,node contents={}},
	bidirected/.style={Latex-Latex,dashed},
	el/.style = {inner sep=2pt, align=left, sloped}
}

\declaretheorem[name=Theorem]{theorem}
\declaretheorem[name=Proposition,sibling=theorem]{proposition}

\declaretheorem[name=Corollary,sibling=theorem]{corollary}

\declaretheorem[name=Remark,sibling=theorem]{remark}
\declaretheorem[name=Definition,sibling=theorem]{definition}

\DeclareMathOperator{\supp}{supp}

\DeclareMathOperator{\Diag}{Diag}

\DeclareMathOperator{\tr}{tr}
\DeclareMathOperator{\ldet}{ldet}

\usepackage[hypertexnames=false,colorlinks=true,breaklinks=true,bookmarks=true,urlcolor=blue,citecolor=blue,linkcolor=blue,bookmarksopen=false,draft=false]{hyperref}

\usepackage{cleveref}
\def\Re{\mathbb{R}}

\def\hat{\widehat}
\def\Z{\mathcal{Z}}

\usepackage{algorithm}
\usepackage{algorithmic}
\usepackage{enumerate}

\usepackage[flushleft]{threeparttable}

\renewcommand*{\qed}{\hfill\ensuremath{\square}}
\renewcommand{\S}{\mathbb{S}}

\begin{document}


\RUNAUTHOR{Li, Fampa, Lee, Qiu, Xie, Yao}

\RUNTITLE{D-optimal Data Fusion: Exact and Approximation Algorithms}

\TITLE{D-optimal Data Fusion:\\ Exact and Approximation Algorithms}

	\ARTICLEAUTHORS{%
	\AUTHOR{Yongchun Li}
	\AFF{School of ISyE, Georgia Institute of Technology, Atlanta, GA, USA, \EMAIL{ycli@gatech.edu}} 
	\AUTHOR{Marcia Fampa}
\AFF{COPPE, Universidade Federal do Rio de Janeiro, Brasil, 
\EMAIL{fampa@cos.ufrj.br}}
	\AUTHOR{Jon Lee}
\AFF{IOE Department, University of Michigan, Ann Arbor, MI, USA, \EMAIL{jonxlee@umich.edu}}
		\AUTHOR{Feng Qiu}
	\AFF{Energy Systems Division, Argonne National Laboratory, Lemont, IL, USA, fqiu@anl.gov}
		\AUTHOR{Weijun Xie}
	\AFF{School of ISyE, Georgia Institute of Technology, Atlanta, GA, USA, \EMAIL{wxie@gatech.edu}
	\AUTHOR{Rui Yao}
	\AFF{Energy Systems Division, Argonne National Laboratory, Lemont, IL, USA, ryao@anl.gov}
} 
}

\ABSTRACT{%
We study the D-optimal Data Fusion (DDF) problem, which aims to select new data points, given an existing Fisher information matrix, so as to maximize the logarithm of the determinant of the overall Fisher information matrix. We show that the DDF problem is NP-hard and has no constant-factor polynomial-time approximation algorithm unless P $=$ NP. Therefore, to solve the DDF problem effectively, we propose two convex integer-programming formulations and investigate their corresponding complementary and Lagrangian-dual problems. We also develop scalable randomized-sampling and local-search algorithms with provable performance guarantees. Leveraging the concavity of the objective functions in the two proposed formulations, we design an exact algorithm, aimed at solving the DDF problem to optimality. We further derive a family of submodular valid inequalities and optimality cuts, which can significantly enhance the algorithm performance. Finally, we test our algorithms using real-world data on the new phasor-measurement-units placement problem for modern power grids, considering the existing conventional sensors. Our numerical study demonstrates the efficiency of our exact algorithm and the scalability and high-quality outputs of our approximation algorithms.
}%



\maketitle

%
\section{Introduction}\label{intro} 
We study the following D-optimal Data Fusion problem
\begin{align}\label{model}
z^*:=\max_{S \subseteq [n], |S|=s} \ldet \bigg(\bm{C} + \sum_{i \in S} \bm{a}_i\bm{a}_i^{\top}\bigg), \tag{DDF}
\end{align}
where ldet denotes the natural logarithm of the determinant, the positive-definite matrix $\bm C \in \S_{++}^d$ is an existing Fisher Information Matrix (FIM),
the columns $\{\bm a_i \in \Re^d\}_{i\in [n]}$ of $\bm{A}\in \Re^{d\times n}$
 represent $n$ candidate $d$-dimensional data points to be selected, and the positive integer $s\in [n]:=\{1,2,\ldots,n\}$ denotes the number of data points to be selected. In \ref{model}, the FIM comprises two parts: $\bm C$ and $\sum_{i \in S} \bm{a}_i\bm{a}_i^{\top}$, corresponding to the information obtained from existing data and new selected points, respectively. Therefore, the goal of \ref{model} is to maximize the information gain by integrating new data points with conventional data from other sources. It is worth remarking that \ref{model} differs from the classic D-optimal design problem in the following aspects: (i) the conventional D-optimal design problem does not have the existing FIM $\bm C$; (ii) the D-optimal design problem typically assumes that $s\ge d$, while in \ref{model}, it is possible that $s< d$. Hence, the existing results for D-optimal design in \cite{singh2020approximation,madan2019combinatorial,nikolov2019proportional,DOPT_PFL} do not directly apply to \ref{model}.

Below, we describe some interesting examples for which the proposed \ref{model} can be applicable. 
\begin{itemize}
	\item 
\textit{Sensor Fusion}: Modern sensor networks often involve multi-type sensors working collectively for a specific monitoring task (see \cite{varshney2012distributed, zhang1995two}). 
When installing new (possibly high-end) sensors, in order to achieve economical and effective operation, it is desirable to maximize the overall information obtained by fusing a small number of new sensors with the existing ones. Using the D-optimality criterion, widely-used in sensor placement, the corresponding optimal sensor-fusion problem is equivalent to \ref{model}. In particular, $\bm C \in \S_{++}^n$ represents the FIM of the existing sensors installed at an $n$-node network. The additional FIM components introduced by new sensors can be shown to be equivalent to adding their corresponding rank-one matrices into the existing FIM (see, e.g., \cite{li2011phasor,kekatos2011convex,yang2013optimal}). 
In this problem, the dimension $d$ of sensor measurements is equal to the number of sensor locations, i.e., $d=n$, and we must have $s\le d=n$. We test a sensor fusion problem in power systems in our numerical study section.

\item \textit{Active Learning}: In the big-data era, there are much more unlabeled data than labeled ones, where the latter are often expensive to acquire. 
Recently, researchers have been proactively working on active learning to select a small subset of unlabeled data points to label, 
with the hope of achieving a desired learning goal (e.g., classification accuracy) using fewer training data and/or achieving smaller labeling costs. Active learning enables the learners to iteratively select the most informative data points by exploiting labeled ones and exploring the unlabeled ones, 
 which is also known as sequential experimental design in statistics.
Thus, given that the FIM of the existing labeled data points is positive-definite, the D-optimal new-data selection problem can be formulated in the form of \ref{model} (see \cite{he2009laplacian,lapsins2021active, martinez2019sequential,yu2006active}). 
At each active learning iteration, the number of selected unlabeled data points is often much smaller than the dimension, especially for high-dimensional data, and can even be set to one for the sake of computational efficiency and stable convergence (see \cite{wang2017active}).

\item \textit{Regularized D-optimal Design}: 
The regularized experimental design that arises from the linear models with a regularization penalty has been recently studied in \cite{tantipongpipat2020lambda, martinez2019sequential}, which admits the same form as our \ref{model}. 
For regularized D-optimal design, its FIM often involves an additional positive-definite matrix $\bm C$ (e.g., an identity matrix in ridge-regularized D-optimal design).
\end{itemize}

\subsection{Maximum-Entropy Sampling Problem (MESP)}

\ref{model} is closely related to the maximum-entropy sampling problem
\begin{align}\label{MESP}
\max_{S \subseteq [n], |S|=s} \ldet \bm{C}_{S,S}
, \tag{MESP}
\end{align}
where $\bm C\in\mathbb{S}_+^n$ has rank at least $s$, and $\bm C_{S,S}$ denotes the principal submatrix of $\bm C$ indexed by $S$.

Since \cite{shewry1987maximum}, \ref{MESP} has seen a variety of applications in statistics and information theory, as well as algorithmic advances (see \cite{FL2022}). Below, we demonstrate that \ref{MESP} is a special case of \ref{model} with $n=d$ when the covariance matrix of \ref{MESP} is positive-definite. This result implies that \ref{model} is NP-hard (see \cite{ko1995exact}) and cannot be approximated by any polynomial-time algorithm within $\Omega(n\log (c))$ with some constant $c>1$ unless $P= NP$ according to \cite{civril2013exponential}. Therefore, to efficiently solve \ref{model}, this paper focuses on developing exact and near-optimal approximation algorithms.

In fact, as we will see below,
\ref{model} is also a special case of \ref{MESP}.

\begin{theorem} For a covariance matrix $\bm C\in\S^n_{++}$, \ref{MESP} is a special case of \ref{model}.
\end{theorem}

\begin{proof}
Let $\bm F$ be the Cholesky factor of the positive semidefinite matrix $(\bm C / \lambda_{\min}(\bm C) - \bm I_n)$. 
Then, $\bm C / \lambda_{\min}(\bm C) = \bm I_n + \bm F^\top \bm F$, and for any $S \subseteq [n]$, $|S|=s$, we can verify that 
\begin{equation}\label{eqmespa}
\begin{array}{ll}
\ldet \bm C_{S,S} & = \ldet\left( (\bm C / \lambda_{\min}(\bm C))_{S,S} \right) - s \log(1 / \lambda_{\min}(\bm C) ) \\[3pt]
 &  = \ldet\left( (\bm I_n + \bm F^\top \bm F)_{S,S} \right) - s \log(1 / \lambda_{\min}(\bm C) ) \\[3pt]
  & = \ldet \left( \bm I_n + \sum_{i=1}^n x_i \bm f_i \bm f_i^\top \right) - s \log(1 / \lambda_{\min}(\bm C) ) ,
\end{array}
\end{equation}
where $\bm f_i\in\Re^n$ is the $i$-th column of $\bm F$ and $\bm x$ is the 0/1 characteristic vector of $S\subseteq[n]$, i.e., $x_j=1$, if $j\in S$; $x_j=0$ if $j\in [n]\setminus S$. The result follows because the first term in \eqref{eqmespa} is the objective function of \ref{MESP} for the covariance matrix $\bm C$, and the last term is the objective function of a particular \ref{model} added to a constant.
\qed
\end{proof}

\begin{theorem} \ref{model} is a special case of MESP.
\end{theorem}

\begin{proof}
Let $\bm C\in\S_{++}^d $ and $\bm a_i\in\Re^d$, for $i\in [n]$, be the input data for \ref{model}. Let $\bm C^{\frac{1}{2}}$ be the square root of $\bm C$ and 
$\bm{B}:=[\bm{b}_1,\cdots,\bm{b}_n]\in \Re^{d\times n}$,
	where $\bm b_i := \bm{C}^{-\frac{1}{2}} \bm a_i$, for $ i \in [n]$. Then, for any $S \subseteq [n]$, $|S|=s$, we can verify that 
\begin{equation}\label{eqmesp}
\begin{array}{ll}
\ldet \bigg(\bm C + \sum_{i \in S} \bm a_i\bm a_i^{\top}\bigg) &= \ldet \bm C + \ldet(\bm I_d + \sum_{i \in S} \bm C^{-\frac{1}{2}}\bm a_i\bm a_i^{\top}\bm C^{-\frac{1}{2}})\\
& = \ldet \bm C + \ldet(\bm I_d + \sum_{i \in S} \bm b_i\bm b_i^{\top})\\[3pt]
& = \ldet \bm C + \ldet ( \bm I_n + \bm B^\top \bm B)_{S,S} .
\end{array}
\end{equation}
The result follows because the first term in \eqref{eqmesp} is the objective function of \ref{model} for the
existing FIM $\bm C$, set $\{ \bm a_i \in \Re^d\}_{i\in [n]}$ of $n$ candidate data points, and $s$ number of data points to be selected; and the last term in \eqref{eqmesp} is the objective function of a particular \ref{MESP} added to a constant.
\qed
\end{proof}


\subsection{Relevant literature}
In this subsection, we survey the relevant literature on exact and approximation algorithm for solving \ref{model} or its variants including the regularized D-optimal design and MESP.

%
%
%
%


\textit{Exact Algorithms:} As a special case of \ref{model}, \cite{he2009laplacian} and \cite{martinez2019sequential} studied the regularized D-optimal design with number of selected data being $s=1$ and derived a closed-form optimal solution by using the formula for the determinant of a rank-one change to a symmetric matrix. A series of research works aimed at solving MESP, another special case of \ref{model}, to optimality by a 
branch-and-bound algorithm have been conducted in \cite{ko1995exact,LeeConstrained,AFLW_IPCO,AFLW_Using,LeeWilliamsILP,HLW,AnstreicherLee_Masked,BurerLee,anstreicher2018maximum,anstreicher2020efficient,Mixing,chen2022computing}. 
\cite{BartonACC,BartonORL} and also \cite{FampaLeeOA} instead considered outer-approximation approaches, and
\cite{li2020best} developed a branch-and-cut algorithm based on (sub)gradient inequalities.
In contrast to previous approaches, we propose a new algorithm by deriving (sub)gradient inequalities and submodular inequalities from two equivalent convex integer-programming formulations of \ref{model},
and exploring probing techniques to obtain optimality cuts. The submodular inequalities and the optimality cuts are numerically demonstrated to significantly strengthen the exact algorithm proposed, which is an enhancement of the LP/NLP branch-and-bound algorithm from \cite{QG92}.


\textit{Approximation Algorithms:} Besides exact algorithms, more scalable yet effective approximation algorithms have also attracted attention, such as greedy, local-search, and randomized-sampling algorithms. Although the objective function of \ref{model} is submodular and non-decreasing, it can be negative. Thus, the existing performance guarantees (e.g., $(1-1/e)$-approximation ratio) for nonnegative non-decreasing submodular maximization cannot be directly applied \cite{nemhauser1978analysis}. On the other hand, it has been shown that the randomized-sampling and local-search algorithms can be successfully applied to the D-optimal design and \ref{MESP} to generate provably near-optimal solutions \cite{li2020best,madan2019combinatorial,singh2020approximation,nikolov2015randomized,nikolov2019proportional}.
Motivated by these recent breakthroughs, we tailor these two approximations algorithms to \ref{model}, and we develop efficient implementations and theoretical guarantees. 
In particular, our approximation bounds of the sampling and local-search algorithms are invariant with 
respect to $s\leftrightarrow n-s$, outperforming the state-of-the-art ones \cite{nikolov2015randomized,li2020best} for \ref{MESP} when $s > n/2$. The detailed comparison can be found in \Cref{sec:approx}. 

\subsection{Summary of the organization and contributions}
\begin{enumerate}[(i)]
\item We derive two convex integer-programming (CIP) formulations of \ref{model}, termed R-DDF and M-DDF. We also study the Lagrangian-dual of their continuous relaxations and establish their optimality gaps.

\item We develop complementary formulations of R-DDF and M-DDF, based on the fact that \ref{model} can be interpreted as excluding $n-s$ less informative data points from a cardinality-$n$ dataset. We show that R-DDF shares the same continuous relaxation value as that of its complement and M-DDF does not.

\item Exploring the two CIPs and the Lagrangian-dual of their continuous relaxations, we establish the theoretical performance guarantees of proposed local-search and randomize-sampling algorithms for solving \ref{model}, as displayed in \Cref{table1}. 
We remark that each CIP provides us a different analysis of approximation bounds and some bounds in \Cref{table1} are invariant with respect to $s$ and $n-s$, due to the complementary formulations.

\item Exploring the two CIPs, we succeed to derive closed-form (sub)gradient and submodular based valid inequalities, which 
are used on an enhancement of the LP/NLP branch-and-bound (B\&B) algorithm proposed in \cite{QG92}, for solving \ref{model}. 
In \cite{MFR20} a numerical comparison is presented between the main algorithms in the literature for convex mixed-integer nonlinear-programming (MINLP), and the LP/NLP B\&B presents an excellent performance in the comparison.

\item We investigate probing techniques to derive optimality cuts to strengthen \ref{model} from both primal and dual perspectives, where our probing schemes are effective and easy-to-implement, based on tight Lagrangian dual bounds and near-optimal approximation algorithms. Our numerical study confirms the effectiveness of the optimality cuts.


\item All the analyses and results can be extended to the regularized D-optimal design problem and \ref{MESP} with a positive-definite sampling covariance matrix.

\end{enumerate}

\begin{table}[ht] 
	\vskip -0.1in
	\centering
	\caption{Approximation bounds for the local-search and sampling algorithms}
	\label{table1}
	\begin{threeparttable}
		\setlength{\tabcolsep}{3pt}\renewcommand{\arraystretch}{1.3}
		\begin{tabular}{l| c| c }
			\hline 
	& \multicolumn{1}{c|}{\ref{model_dopt}} &	\multicolumn{1}{c}{\ref{model_mesp} } \\
			\hline
	Local-Search \Cref{algo:localsearch}& $ d \log \left(1+ \frac{\bar{s} \sigma^2_{\max}}{d(1+\sigma_{\max})} \right)$ \vphantom{\bigg|} & $
	 \bar{s}\log(\bar{s})$ \\
			\hline
Sampling \Cref{algo:sampling} & $ n\log(x_{\min}) -(n-s)\log(1+\delta)$ & $\bar{s} \log\left(\frac{\bar{s}}{n}\right) + \log \left(\binom{n}{\bar{s}}\right)$ \vphantom{\Big|}\\
			\hline
		\end{tabular}%
		\vspace{+.1em}
		\begin{tablenotes}
			\item {$\bar{s} :=\min\{s, n-s\}; \sigma_{\max}:=\max_{i\in [n]}\bm a_i^{\top} \bm C^{-1}\bm a_i; \delta = \lambda_{\max}(\bm A^{\top}\bm C^{-1} \bm A)$} 
		\end{tablenotes} 
	\end{threeparttable}
	\vspace{-.8em}
\end{table}

\noindent \textit{Notation:}	The following notation is used throughout the paper. We use bold lower-case letters (e.g., $\bm{x}$) and bold upper-case letters (e.g., $\bm{X}$) to denote vectors and matrices, respectively, and we use corresponding non-bold letters (e.g., $x_i$) to denote their components. We let $\Re^n_+$ denote the set of all $n$-dimensional nonnegative vectors. Given positive integers $s<n$, we let $[n]:=\{1,2,\cdots, n\}$, $[s,n]:=\{s,s+1,\cdots, n\}$, and $\bar s := \min\{s,n-s\}$. Further, we let $\Z_s$ denote the collection of feasible solutions satisfying the cardinality constraint, i.e.,
	$\Z_s := \{\bm x \in \{0,1\}^n: \sum_{i\in [n]} x_i = s\}.$
We let $\S_+^n$ (resp., $\S_{++}^n$) denote the cone of $n\times n$ symmetric positive semidefinite (resp., definite) matrices.
We let $\bm{C}^{\frac{1}{2}}$ denote the square root of matrix $\bm C$, i.e., $\bm{C}^{\frac{1}{2}} \bm{C}^{\frac{1}{2}} = \bm C$. We let $\bm{C}^{\dag}$ denote the Moore-Penrose pseudo-inverse of $\bm{C}$. 
We let $\bm{I}_n$ denote the $n\times n$ identity matrix, and we let $\bm{e}_i\in \Re^n$ denote the $i$-th standard-unit vector. We let $\bm 1$ denote a vector with all entries being 1.
For $\bm x\in\Re^n$, we denote its support
by $\supp(\bm x)\subseteq [n]$. 
We let $|S|$ denote the cardinality of a finte set $S$. 
Overloading the notation $\binom{s}{k}$ for the number of $k$-subsets of an $s$-element set,
given a set $S$, we let $\binom{S}{k}$ denote the collection of all the cardinality-$k$ subsets of $S$. Given an $m \times n$ matrix $\bm{X}$ and two subsets $S\subseteq [m]$, $T\subseteq [n]$, we let $\bm{X}_{S,T}$ denote the submatrix of $\bm{X}$ with rows and columns indexed by sets $S, T$, respectively, and we let $\bm{X}_S$ denote the submatrix of $\bm{X}$ with columns indexed by $S$.
Given a symmetric matrix $\bm{X}$, we let $\lambda_{\min}(\bm{X}),\lambda_{\max}(\bm{X})$ denote the least and greatest eigenvalues of $\bm{X}$, respectively. For $\bm X, \bm Y\in \Re^{n\times n}$, we let $\bm X \circ \bm Y$ denote their Hadamard product, and $\tr \bm X$ denote the trace of $\bm X$. 
Additional notation is introduced as needed.

\begin{remark}\label{rem:notation} Throughout the paper, considering the existing FIM $\bm C$, the $n$ candidate $d$-dimensional data points $\{\bm a_i \in \Re^d\}_{i\in [n]}$ to be selected in \ref{model}, and $\bm{A}:=[\bm{a}_1,\cdots,\bm{a}_n]\in \Re^{d\times n}$, we use the definitions: $\bm b_i := \bm{C}^{-\frac{1}{2}} \bm a_i$, for $ i \in [n]$, 
 $\bm{B}:=[\bm{b}_1,\cdots,\bm{b}_n]\in \Re^{d\times n}$, $\bm q_i := (\bm C + \bm A\bm A^{\top})^{-\frac{1}{2}}\bm a_i$, for $i\in [n]$; we let $\bm V^{\top} \bm V$ be the 
the Cholesky factorization of $\bm I_n + \bm B^{\top} \bm B$, and let $\bm v_i$ be the $i$-th column of $\bm V$, for $i\in [n]$.
\end{remark}

\vspace{0.1in} 

\noindent \textit{Organization:} The remainder of the paper is organized as follows. In \Cref{sec:cip}, we develop two CIP formulations for \ref{model}, and their corresponding complementary problems and Lagrangian duals. In Section \ref{sec:approx}, we develop and analyze two approximation algorithms. In Section \ref{sec:bc}, we present our exact algorithmic approach with the introduction of valid submodular inequalities and optimality cuts. In Section \ref{sec:pmu}, we present numerical results on a real-world application in power systems. Finally, Section \ref{sec:conclusion} contains brief conclusions.

\section{Two convex integer-programming formulations}\label{sec:cip}
Next, we present two convex integer-programming (CIP) formulations for \ref{model}, termed R-DDF and M-DDF, as well as their complementary problems and the Lagrangian-dual of their relaxations.



\subsection{First CIP formulation: R-DDF}
To formulate \ref{model} as a mathematical program, we introduce the binary variable $x_i=1$, if the $i$-th data point is selected, and 0 otherwise, for each $i\in [n]$. Our first formulation of \ref{model} is as follows.
\begin{restatable}{proposition}{themdopt} \label{them:dopt}
	\ref{model} is equivalent to
	\begin{equation} \label{model_dopt}\tag{R-DDF}
	z^* := \ldet \bm C +
	\max_{\bm{x}\in \Z_s} \Bigg\{\ldet \bigg(\bm{I}_d+\sum_{i \in [n]} x_i \bm{b}_i \bm{b}_i^{\top} \bigg) \Bigg\}. 
	\end{equation}
\end{restatable}

\begin{proof}
	Let $\bm x$ be the 0/1 characteristic vector of $S\subseteq[n]$. Then, we have 
	\[
	\begin{array}{ll}
 \det \bigg(\bm{C} + \sum_{i \in S} \bm{a}_i\bm{a}_i^{\top}\bigg)&= \det \bigg(\bm{C} + \sum_{i \in [n]}x_i \bm{a}_i\bm{a}_i^{\top} \bigg)
	= \det (\bm C) \det \bigg(\bm{I}_d+ \sum_{i \in [n]}x_i \bm{b}_i\bm{b}_i^{\top} \bigg).
	\end{array}
	\]
	The result follows from the definition of $\bm{b}_i$ (see Remark \ref{rem:notation}) and by taking the logarithm on both sides of the above equation.
	\qed
\end{proof}

We refer to the problem as \ref{model_dopt} because the objective function is similar to that of the regularized D-optimal design problem.
The following result presents the Lagrangian dual of the continuous relaxation of \ref{model_dopt}.

\begin{restatable}{proposition}{themdoptld}\label{prop:dopt_ld}
	The Lagrangian dual of 
	the continuous relaxation of \ref{model_dopt} is 
		\begin{equation} \label{eq:dopt_ld}
\hat z_R := \ldet \bm C + 
\min_{\substack{{\bm{\Lambda} \in \S_{++}^d,}\\{ \nu, \bm{\mu} \in \Re^n_+}} }	
\bigg\{-\ldet (\bm{\Lambda} ) + \tr \bm{\Lambda} +s\nu +\sum_{i \in [n]}\mu_i-d: 
	\bm{b}_i^{\top}\bm{\Lambda}\bm{b}_i \le \nu+ \mu_i, ~ i \in [n] \bigg\}.
	\end{equation}
\end{restatable}
	\begin{proof}
	First, we introduce an auxiliary matrix variable $\bm X \in \S_{++}^d$ and reformulate \ref{model_dopt} as
	\[
	\ldet \bm C + \max_{\bm X \in \S_{++}^d,
	\bm{x} \in \Z_s} \bigg\{\ldet (\bm X): \bm{I}_d+\sum_{i \in [n]} x_i \bm{b}_i \bm{b}_i^{\top} \succeq \bm X \bigg\} .
	\]
	Then we derive the Lagrangian dual of the above maximization problem over $\bm x, \bm X$.
	Let $\bm{\Lambda} \in \S_{++}^d$, $\nu\in \Re$, $\bm{\upsilon}\in \Re^{n}_+$, $\bm{\mu}\in \Re^{n}_+$ denote the Lagrangian multipliers. The Lagrangian function $L$ is
	\[
	L(\bm{x}, \bm{X}, \bm{\Lambda}, \nu, \bm{\upsilon},\bm{\mu}) = \ldet (\bm{X} ) + \tr \bm{\Lambda} - \tr(\bm{X}\bm{\Lambda}) + \sum_{i \in [n]}\left(\bm{b}_i^{\top}\bm{\Lambda}\bm{b}_i-\nu+\upsilon_i-\mu_i \right)x_i + s\nu+\sum_{i \in [n]}\mu_i.
	\]
	Maximizing $L$ over $(\bm{x}, \bm{X})$ yields
	\[
	\bm{\Lambda} = \bm{X}^{-1},~ \bm{b}_i^{\top}\bm{\Lambda}\bm{b}_i-\nu+\upsilon_i-\mu_i = 0, \forall i \in [n]. 
	\]
	Then the Lagrangian dual problem can be obtained by plugging the above result into $L$, removing $\bm{\upsilon}$, and minimizing $L$ over $(\bm \Lambda, \nu, \bm \mu)$. 
\qed
\end{proof}

{We further derive the gradient, Hessian, and Lipschitz constant of the objective function of \ref{model_dopt}, which allow us to use first- or second-order methods (e.g., Frank-Wolfe algorithm) to compute $\hat{z}_R$ with a proven convergence rate. 
We demonstrate the following result.}
\begin{restatable}{proposition}{propdoptlip}\label{prop:dopt_lips}
	For any $\bm x\in [0,1]^n $, the gradient $\bm g \in \Re^n$ and the Hessian $\bm H \in \Re^{n\times n}$ of the objective function in the continuous relaxation of \ref{model_dopt} are
	\[\bm g(\bm x):=[\bm b_1^{\top} \bm X^{-1}\bm b_1, \cdots, \bm b_n^{\top} \bm X^{-1}\bm b_n], \quad 
	\bm H(\bm x) :=- (\bm B^{\top} \bm X^{-1} \bm B)\circ (\bm B^{\top} \bm X^{-1} \bm B), \mbox{ and } \bm H(\bm x) \succeq - \delta^2 \bm I_n, \]
	where $\bm X :=\bm I_d + \sum_{i \in [n]} x_i \bm b_i \bm b_i^{\top}$ and $\delta := \lambda_{\max}(
\bm B^{\top}\bm B 
	)$.
\end{restatable}

\begin{proof}
	For any $\bm Y \in \S_{++}^d$, it is well-known that the function $\ldet(\bm Y)$ has first- and second-order derivatives: $\bm Y^{-1}$ and $-\bm Y^{-1}\partial \bm Y \bm Y^{-1} $. Thus, the gradient and Hessian of the objective function over $\bm x$ in \ref{model_dopt} can be derived.
	
	According to \cite{johnson1974hadamard}, the least eigenvalue of Hessian matrix $\bm H(\bm x)$ satisfies 
	$$\lambda_{\min}\left(\bm H(\bm x)\right) =-\lambda_{\max}\left(-\bm H(\bm x)\right) \ge - \lambda_{\max}^2(\bm B^{\top}\bm X^{-1}\bm B )\ge - \lambda_{\max}^2(\bm B^{\top}\bm B ),$$
	where the last inequality is due to $\bm X \succeq \bm I_d$.
	\qed
\end{proof}

We adopt the well-known Frank-Wolfe algorithm to compute $\hat{z}_R$ in the numerical study.
Due to \Cref{prop:dopt_lips} and \cite[Theorem 4]{li2020best}, the Frank-Wolfe algorithm admits a convergence rate of $O(\bar{s}^2\lambda_{\max}^2(\bm B^{\top}\bm B )/\kappa)$, where $\kappa$ denotes the number of iterations.

{An alternative interpretation of \ref{model} is via excluding $n-s$ ineffective data points, which leads to the complementary formulation of \ref{model_dopt} and the resulting Lagrangian dual problem below.}
\begin{restatable}{proposition}{propdoptcom}\label{prop:dopt_com}
\ref{model_dopt} is equivalent to
\begin{equation}\label{model_doptc}\tag{R-DDF-comp}
z^*= \ldet(\bm C + \bm A \bm A^{\top}) + \max_{\bm x \in \Z_{n-s}} \bigg\{\ldet \bigg(\bm{I}_d-\sum_{i \in [n]} x_i \bm{q}_i \bm{q}_i^{\top} \bigg)\bigg\},
\end{equation}
and the Lagrangian dual of the continuous relaxation of \ref{model_doptc} is
	\begin{equation} \label{eq:doptc_ld}
\hat z_R = \ldet(\bm C \!+\! \bm A \bm A^{\top}) + \min_{\substack{\bm{\Lambda} \in \S_{++}^d,\\ \nu, \bm{\mu} \in \Re^n_+}} \!\!\!\bigg\{\! -\ldet \bm{\Lambda} + \tr \bm{\Lambda} +(n\!-\!s)\nu +\sum_{i \in [n]}\mu_i-d : 
-\bm{q}_i^{\top}\bm{\Lambda}\bm{q}_i \le \nu+ \mu_i, ~ i \in [n] \bigg\}.
\end{equation}
\end{restatable}
\begin{proof}
For any $\bm x\in \Z_s$, the objective function of \ref{model_dopt} can be written as 
\begin{align*}
&\ldet \bigg(\bm{C} + \sum_{i \in [n]} x_i\bm{a}_i\bm{a}_i^{\top}\bigg) = \ldet \bigg(\bm{C} + \bm A \bm A^{\top}- \sum_{i \in [n]} (1-x_i)\bm{a}_i\bm{a}_i^{\top}\bigg) \\
&\quad \quad = \ldet(\bm C + \bm A \bm A^{\top}) + \ldet \bigg(\bm{I}_d - \sum_{i \in [n]} (1-x_i) (\bm C + \bm A \bm A^{\top})^{-\frac{1}{2}} \bm{a}_i\bm{a}_i^{\top} (\bm C + \bm A \bm A^{\top})^{-\frac{1}{2}} \bigg).
\end{align*}
Then, replacing variable $\bm x$ by $\bm 1-\bm x$, and considering the definition of $\bm{q}_i$ (see Remark \ref{rem:notation}), we arrive at the equivalent formulation of \ref{model_dopt} given by \ref{model_doptc}.

Following the similar derivation as \Cref{prop:dopt_ld}, we can also derive the Lagrangian dual of the continuous relaxation of \ref{model_doptc}. 
\qed
\end{proof}

We remark that (i) the greatest eigenvalue of $\sum_{i\in [n]} \bm q_i\bm q_i^{\top}$ is strictly less than one, so the matrix $\bm I_d-\sum_{\in [n]}x_i \bm q_i \bm q_i^{\top}$ in \ref{model_doptc} is always positive-definite and the objective value is finite; (ii) the continuous relaxation value of \ref{model_dopt} does not vary with the complementary transformation, 
but considering the two Lagrangian dual problems \eqref{eq:dopt_ld} and \eqref{eq:doptc_ld} together, we obtain an approximation bound that is
better than by considering only one of the Lagrangian duals, as discussed in Section~\ref{sec:approx}.

\subsection{Second CIP formulation: M-DDF} \label{subsec:mesp}





\begin{restatable}{lemma}{lemnik}[\cite{nikolov2015randomized}, Lemma 13]\label{lem:iota}
 Let $\bm \lambda\in\mathbb{R}_+^n$ with $\lambda_1\geq \lambda_2\geq \cdots\geq \lambda_n$, and let $0<s\leq n$. There exists a unique integer $k$, with $0\leq k< s$, such that
 \begin{equation}\label{reslemiota}
 \lambda_{k}>\frac{1}{s-k}\sum_{i\in [k+1, n]} \lambda_{i}\geq \lambda_{k+1},
 \end{equation}
 with the convention $\lambda_0=+\infty$.
\end{restatable}

\begin{restatable}{lemma}{lemmars}\label{lem:mars}
 Let $\bm \lambda\in\mathbb{R}_+^n$ with $\lambda_1\geq \lambda_2\geq \cdots\geq
 \lambda_t> \lambda_{t+1}=\cdots=
 \lambda_s>\lambda_{s+1}=\cdots
 =\lambda_n=0$.
 Then, the $k$ satisfying
 \eqref{reslemiota} is
 precisely $t$. 
\end{restatable}

\begin{proof}
The result is immediate because
\[
\frac{1}{s-t}\sum_{i\in [t+1, n]} \lambda_{i} = \lambda_{t+1}.
\]
 \qed
\end{proof}

\begin{definition}[\cite{nikolov2015randomized}]\label{def_f}
	For a matrix $\bm X \in \S_+^n$ with eigenvalues $\lambda_1 \ge \cdots \ge \lambda_{n}\ge 0$, we let
	\begin{align*}
	f(\bm X) = \prod_{i\in [k]} \lambda_i \times \bigg(\frac{1}{s-k} \sum_{i\in [k+1, n]} \lambda_i\bigg)^{s-k},
	\end{align*}
	where $k<s$ is unique non-negative integer satisfying
 \eqref{reslemiota}.
\end{definition}
It has been shown in \cite{nikolov2015randomized} that $f(\cdot)$ is a concave function. 
With the notation above, we are ready to define \ref{model_mesp}.
\begin{restatable}{theorem}{themmesp} \label{them:mesp}
	\ref{model} is equivalent to the following CIP formulation
	\begin{align}\label{model_mesp}\tag{M-DDF}
	& z^{*} = \ldet \bm C + 
	\max_{\bm{x} \in \Z_s} \bigg \{ \log f \bigg(\sum_{i\in [n]}x_i \bm{v}_i\bm{v}_i^{\top}\bigg) \bigg \}.
	\end{align}

\end{restatable}
\begin{proof}
	It is sufficient to prove that the objective functions of \ref{model_mesp} and \ref{model_dopt} are equal for any 
	$\bm{x}\in \Z_s$. 
	So, let $S:=\supp({\bm x})$ for an $\bm{x}\in \Z_s$.
	Clearly $|S|=s$, and thus we have
	\begin{align*}
	\det \bigg(\bm{I}_d+\sum_{i \in [n]}x_i \bm{b}_i\bm{b}_i^{\top}\bigg)
	&	= \det (\bm{I}_d+ \bm{B}_{S}\bm{B}_S^{\top})	 = \det (\bm{I}_{s} + \bm{B}_{S}^{\top}\bm{B}_S) = \det (\bm{I}_n+\bm B^{\top} \bm B)_{S,S} \\
	&	= \det(\bm V_S^{\top} \bm V_S) = f \bigg(\sum_{i \in [n]} x_i \bm v_i \bm v_i^{\top}\bigg),
	\end{align*}
	where 
	the second equation is because $\bm{B}_{S}\bm{B}_S^{\top}$ and $\bm{B}_{S}^{\top}\bm{B}_S=( \bm B^{\top}\bm B)_{S,S}$ have the same non-zero eigenvalues, 
	and the last one follows from the definition of $\bm{v}_i$ (see Remark \ref{rem:notation}) and from 
	\Cref{lem:mars}. This completes the proof. \qed
\end{proof}

The matrix in the objective of \ref{model_mesp} is of order $n$, while that of \ref{model_dopt} is of order $d$. We may consider choosing between them in practice, based on the two parameters $n$ and $d$ and on the continuous-relaxation bounds. 


\begin{restatable}{proposition}{propmespld}[\cite{li2020best}] \label{prop:mesp_ld}
The Lagrangian dual of 
	the continuous relaxation of \ref{model_mesp} is 
	\begin{align}\label{eq:mesp_ld}
	\hat z_{M} := \ldet \bm C + \min_{\substack{\bm \Lambda \in \S_{++}^n,\\ \nu, \bm{\mu} \in \Re^n_+}} \bigg \{ -\log\det\limits_s \bm \Lambda + s\nu + \sum_{i\in [n]}\mu_i -s: \bm v_i^{\top}\bm \Lambda \bm v_i \le \nu + \mu_i, ~ i \in [n] \bigg \},
	\end{align}
where the function $\det \limits_s(\cdot)$ denotes the product of $s$ least eigenvalues.
\end{restatable}
We note that the continuous relaxation value $\hat z_{M}$ of \ref{model_mesp} can be computed efficiently via the Frank-Wolfe algorithm (see \cite{li2020best}).
Similar to \ref{model_dopt}, \ref{model_mesp} also admits a complementary formulation, which has been widely-studied in the \ref{MESP} literature (e.g., see \cite{AFLW_IPCO,AFLW_Using,anstreicher2020efficient,Mixing,chen2022computing,FL2022}).

According to the identity 
$$\det (\bm{I}+\bm B^{\top} \bm B)_{S,S} = \det (\bm{I}_n+\bm B^{\top} \bm B) \det ((\bm{I}_n+\bm B^{\top} \bm B)^{-1})_{[n]\setminus S,[n]\setminus S}$$ 
for any subset $S\subseteq [n]$, see \cite[Section 0.8.4]{johnson1985matrix}, we can derive the complementary formulation for \ref{model_mesp} and then the Lagrangian 
dual of its continuous relaxation:
\begin{restatable}{proposition}{propmespcom}\label{prop:mesp_com}
	 \ref{model_mesp} is equivalent to
	\begin{align}\label{model_mespc}\tag{M-DDF-comp}
	z^*= \ldet(\bm C + \bm A \bm A^{\top}) + \max_{\bm x \in \Z_{n-s}} \bigg\{f \bigg(\sum_{i \in [n]} x_i \bm{v}_i \bm{v}_i^{\top} \bigg)\bigg\},
	\end{align}
	and the Lagrangian dual of the continuous relaxation of \ref{model_mespc} is
	\begin{equation}
		\label{eq:mespc_ld}
\hat z_{M}^c := \ldet(\bm C \!+\! \bm A \bm A^{\top}) + \!\! \min_{\substack{\bm \Lambda \in \S_{++}^n,\\ \nu, \bm{\mu} \in \Re^n_+}} \!\bigg \{ \!-\log \det \limits_{n-s}\bm \Lambda + (n\!-\!s)\nu + \!\sum_{i\in [n]}\mu_i -(n\!-\!s): 
\bm v_i^{\top}\bm \Lambda \bm v_i \le \nu + \mu_i, ~ i \in [n] \bigg \}.
	\end{equation}
\end{restatable}

Contrary to \Cref{prop:dopt_com}, we will observe in our numerical study in Section \ref{sec:pmu}, that the Lagrangian dual problem \eqref{eq:mesp_ld} is not equivalent to \eqref{eq:mespc_ld}. Furthermore, the complementary problem \ref{model_mespc} and the Lagrangian dual \eqref{eq:mespc_ld} motivate us to further improve the approximation bounds of the local-search and sampling algorithms in \cite{li2020best}, as shown in the next section.

The two alternative formulations \ref{model_dopt} and \ref{model_mesp}, together with their complementary problems, often provide us with distinct continuous-relaxation solutions, which can help improve the analyses of the approximation algorithms. 
Each formulation has its own advantage under different circumstances. For example, if existing data contain more accurate information, i.e., if the existing FIM $\bm C$ dominates the overall FIM matrix, then we recommend \ref{model_dopt} because its continuous relaxation provides a tighter upper bound. On the other hand, if the information from new data points is more valuable, i.e., the effect of $\bm C$ is negligible, then \ref{model_mesp} tends to yield a stronger continuous relaxation bound. Our theoretical analyses and numerical study will further confirm these phenomena.

\section{Two approximation algorithms for \ref{model}} \label{sec:approx}
Motivated by our two CIP formulations of \ref{model}, we investigate simple and scalable approximation algorithms (i.e., local-search and randomized-sampling algorithms) for providing near-optimal selections of the new data points. 



\subsection{A local-search algorithm}
In this subsection, we study a local-search algorithm for \ref{model}, which has been successfully applied to many combinatorial optimization problems (see, for example, \cite{singh2020approximation,li2020best,li2020exact}). 
The algorithm runs as follows: (i) first, we start with a cardinality-$s$ subset $\hat{S}\subseteq [n]$; (ii) next, we swap one element from the set $\hat{S}$ with one from the unchosen set $[n]\setminus\hat{S}$, and we update the chosen set 
if the swapping strictly increases the objective value; and (iii) the algorithm terminates when there is no improvement. Motivated by \ref{model_dopt}, we provide an efficient implementation of the local-search algorithm, as shown in Algorithm \ref{algo:localsearch}, with time complexity of $O(d^2s(n-s))$ at the for-loop (i.e., Steps 5-11). Specifically, at Step 6, the strict improvement $\det(\bm I_d + \bm X - \bm b_i \bm b_i^{\top}+ \bm b_j \bm b_j^{\top}) > \det(\bm I_d + \bm X)$, can be efficiently computed as
		\begin{equation} \label{eq:dopt_local}
\begin{aligned}
& \det(\bm{I}_d + \bm{X} - \bm{b}_i\bm{b}_i^{\top} +\bm{b}_j\bm{b}_j^{\top}) > \det(\bm{I}_d + \bm{X})	\\
\Longleftrightarrow & \det(\bm{I}_d + \bm{X}) (1+\bm b_i^{\top} \bm \Lambda \bm b_i)\left[1+\bm b_j^{\top} (\bm{I}_d + \bm{X} - \bm{b}_i\bm{b}_i^{\top} )^{-1} \bm b_j\right] > \det(\bm{I}_d + \bm{X}) \\
\Longleftrightarrow 	& 	\bm{b}_j^{\top} \bm{\Lambda} \bm{b}_j - \bm{b}_i^{\top} \bm{\Lambda} \bm{b}_i \bm{b}_j^{\top} \bm{\Lambda} \bm{b}_j + \bm{b}_i^{\top} \bm{\Lambda} \bm{b}_j \bm{b}_j^{\top} \bm{\Lambda} \bm{b}_i > \bm{b}_i^{\top} \bm{\Lambda} \bm{b}_i, 
\end{aligned}
\end{equation}
which follows from the Sherman–Morrison formula. The update of matrix $\bm \Lambda$ at Step 8 also follows from Sherman–Morrison formula, to avoid calculations of inverses from scratch.

\begin{algorithm}[ht]
	\caption{Local-Search Algorithm} \label{algo:localsearch}
	\begin{algorithmic}[1]
		\STATE \textbf{Input:} vectors $\{\bm{b}_i \in \Re^n\}_{i \in [n]}$ and a positive integer $s\in [n]$
		\STATE Initialize a cardinality-$s$ subset $\hat{S} \subseteq [n]$, matrix $\bm X := \sum_{i\in \hat{S}} \bm b_i \bm b_i^{\top}$, and matrix $\bm \Lambda := (\bm I_d + \bm X)^{-1}$
		\STATE Compute $\bm b_l^{\top} \bm \Lambda \bm b_l$ for each $j\in [n]$ and $\bm b_i^{\top} \bm \Lambda \bm b_j$ for each $i\in \hat{S}$, $j\in [n]\setminus \hat{S}$
		\STATE		\textbf{do}
		\STATE 	\quad \textbf{for} {each pair {$(i,j) \in \hat{S} \times ([n]\setminus \hat{S})$}}
		\STATE	\quad \quad 	\textbf{if}	{ $\bm b_j^{\top} \bm \Lambda \bm b_j- \bm b_i^{\top} \bm \Lambda \bm b_i \bm b_j^{\top} \bm \Lambda \bm b_j + (\bm b_i^{\top} \bm \Lambda \bm b_j)^2 > \bm b_i^{\top} \bm \Lambda \bm b_i$}
		\STATE \quad \quad \quad Update $\hat{S} := \hat{S} \cup \{j\} \setminus \{i\}$ and $\bm X := \bm X - \bm{b}_i\bm{b}_i^{\top} +\bm{b}_j\bm{b}_j^{\top}$
		\STATE \quad \quad \quad Compute $\bm \Lambda_i:=\bm \Lambda + \frac{\bm \Lambda \bm b_i \bm b_i^{\top} \bm \Lambda}{1-\bm b_i^{\top} \bm \Lambda \bm b_i}$, and update $\bm \Lambda:= \bm \Lambda_i - \frac{\bm \Lambda_i\bm b_j \bm b_j^{\top} \Lambda_i}{1+\bm b_j^{\top} \bm \Lambda_i \bm b_j}$
		\STATE \quad \quad \quad Update $ \bm b_l^{\top} \bm \Lambda \bm b_l$ for each $l\in [n]$, and update $\bm b_{i'}^{\top} \bm \Lambda\bm b_{j'}$ for each $i'\in \hat{S}$, $j'\in [n]\setminus \hat{S}$
		\STATE	\quad \quad 	\textbf{end if}	
		\STATE	\quad 	\textbf{end for}	
		\STATE		\textbf{while} {there is still an improvement (i.e., Step 6 is true for some pair $(i,j)$)}
		\STATE \textbf{Output:} $\hat S$
	\end{algorithmic}
\end{algorithm}


We use the proposed Lagrangian-dual of \ref{model_dopt}, \ref{model_mesp}, and their complements \ref{model_doptc}, \ref{model_mespc} to theoretically guarantee the quality of the output of \Cref{algo:localsearch} by constructing feasible dual solutions: 

\begin{restatable}{theorem}{themlocal}\label{them:local}
	Let $\hat{S}$ denote the output of the local-search Algorithm \ref{algo:localsearch}, and let $\bar{s}:=\min\{s, n-s\}$, then the set $\hat{S}$ yields a $\min\{d\log(1+(\bar{s}/d)\sigma_{\max}^2/(1+\sigma_{\max} )), \bar{s} \log(\bar{s})\}$-approximation bound for \ref{model}, i.e., 
	\begin{align*}
&	\ldet\bigg(\bm{C}+\sum_{i \in \hat{S}} \bm{a}_i \bm{a}_i^{\top} \bigg) \ge 	 z^*- \min\bigg\{d\log\bigg(1+ \frac{\bar{s}\sigma_{\max}^2}{d(1+
		\sigma_{\max})}\bigg), \bar{s} \log(\bar{s}) \bigg \},
	\end{align*}
	where the constant $\sigma_{\max}:=\max_{i\in [n]} \bm a_i^{\top} \bm C^{-1}\bm a_i$.
\end{restatable}

\begin{proof}
	The approximation bound attains the minimum of $d \log(1+(\bar{s}/d)\sigma_{\max}^2/(1+\sigma_{\max}) )$ and $\bar{s} \log(\bar{s})$, where they are derived based on \ref{model_dopt}, \ref{model_mesp}, and their complementary problems, respectively. Exploring the local optimality of the output solution $\hat{S}$, we can show that
	\[
\begin{array}{l}
z^* \le \hat{z}_R \le \ldet \bm C + \ldet\bigg(\bm{I}_d+\sum_{i\in \hat S} \bm b_i \bm b_i^{\top}\bigg) + d \log \left (1+ \frac{\bar{s} \sigma^2_{\max}}{d(1+\sigma_{\max})} \right ),
\\
z^* \le \hat{z}_M \le \log f\bigg(\sum_{i\in \hat S} \bm v_i \bm v_i^{\top} \bigg) + \ldet (\bm C) + s\min\left\{\log(s), \log\left(n-s-\frac{n}{s}+2\right)\right\},
\\
z^* \le \hat{z}_M^c \le \log f\bigg(\sum_{i\in \hat S} \bm v_i \bm v_i^{\top} \bigg) + \ldet (\bm C) + (n-s)\min\left\{\log(n-s), \log\left(s-\frac{n}{n-s}+2\right)\right\}.
\end{array}
\]
	The detailed proof can be found in Appendix \ref{proof:them:local}.
%
	\qed
\end{proof}

We make the following remarks concerning \Cref{them:local}:
	\begin{enumerate}[(i)]
\item Either approximation bound with the `min' in \Cref{them:local} is invariant with $s$ and $n-s$ by leveraging the two CIPs and their complements;
\item The second approximation bound attains zero when $\bar{s}=1$, implying that the output solution of the local-search \Cref{algo:localsearch} is optimal, and when $s=1$, we have $\hat{z}_M=z^*$, and when $s=n-1$, we have $\hat{z}_M^c=z^*$, as summarized in \Cref{cor:ls};
\item The second approximation bound (i.e., $\bar{s} \log(\bar{s})$) of the local-search \Cref{algo:localsearch} improves on the one for \ref{MESP} (i.e., $s\log[s-(s-1)/s (2s-n)_+]$) derived by \cite{li2020best}, when the covariance matrix of \ref{MESP} is positive-definite.
 \Cref{fig_LS_bound} illustrates the comparisons of two approximation bounds (i.e., our new bound $\bar{s} \log(\bar{s})$ versus the existing bound $s\log[s-(s-1)/s (2s-n)_+]$) with $n=10,50$, where the greatest improvement of our bound over that in \cite{li2020best} is indicated by a black dashed line. We see that when $s > n/2$, our new bound provides the local-search \Cref{algo:localsearch} with a tighter performance guarantee;
\item The first approximation bound involving the constant $\sigma_{\max}$ will be discussed after \Cref{them:samp}, along with that of the sampling algorithm which is also derived based on \ref{model_dopt};
\item Another side product of \Cref{them:local} is to provide \ref{model} with the optimality gaps of three proposed Lagrangian dual bounds: $\hat{z}_R$, $\hat{z}_M$, and $\hat{z}_M^c$, as presented in \Cref{cor:ldb}.
	\end{enumerate}

\begin{corollary} \label{cor:ls}
When $s=1$ and $s=n-1$, we have $\hat{z}_M=z^*$ and $\hat{z}_M^c=z^*$, respectively. In both cases, the local-search \Cref{algo:localsearch} returns an optimal solution.
\end{corollary}
\begin{figure}[hbtp]
	\centering
	
	\subfigure[{$n$=10}] {
		\includegraphics[width=7.5cm,height=5.3cm]{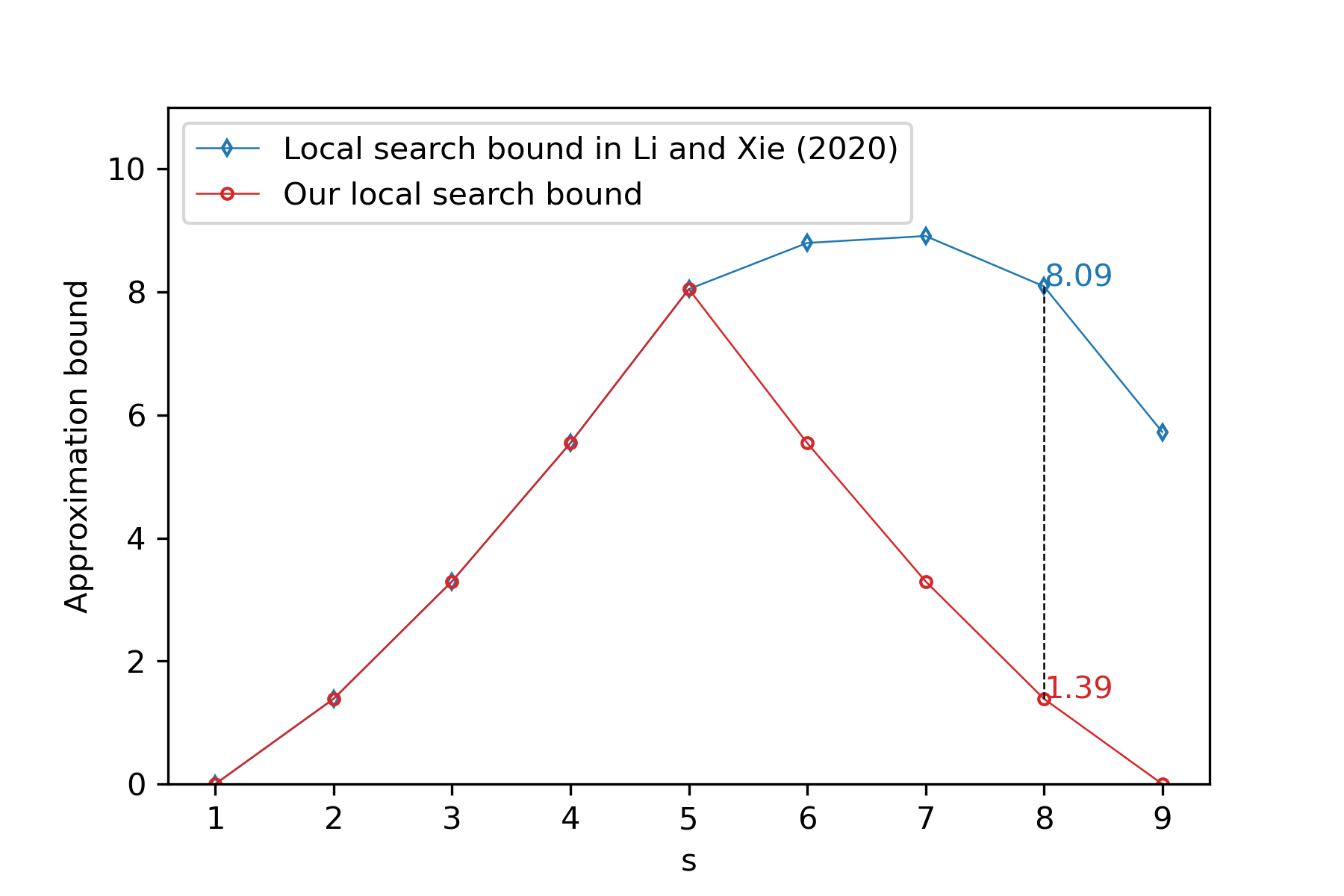}
	}
	\hspace{1em}
	\subfigure[{$n$=50}] {
		\centering
		\includegraphics[width=7.5cm,height=5.3cm]{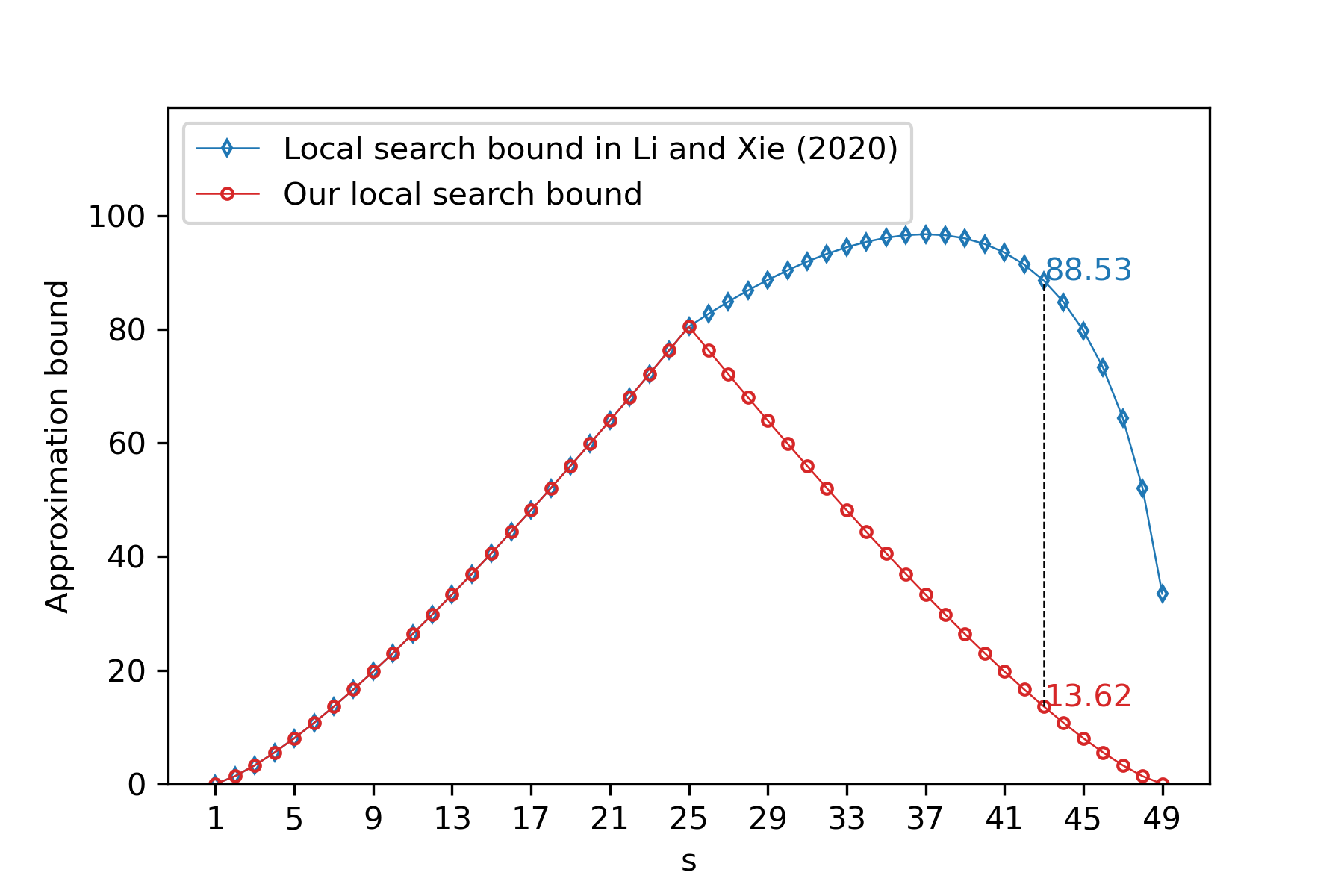}
	}
	\caption{Comparison between our approximation bound from the local-search \Cref{algo:localsearch} with the existing one}\label{fig_LS_bound}
\end{figure}

\begin{corollary}\label{cor:ldb}
			The continuous relaxation values $\hat{z}_R$ of \ref{model_dopt}, $\hat{z}_M$ of \ref{model_mesp}, and $\hat{z}_M^c$ of the complementary M-DDF \eqref{model_mespc} satisfy
			\[
			\begin{array}{l}
z^* \le \hat{z}_R \le z^* + n\log\bigg(1+ \frac{\bar{s}\sigma_{\max}^2}{n(1+
	\sigma_{\max})}\bigg);\\
z^* \le \hat{z}_M \le z^* + s\log\left(s - \frac{s-1}{s} (2s-n)_+\right);\\
z^* \le \hat{z}_M^c \le z^* + (n-s)\log\left(n-s - \frac{n-s-1}{n-s} (n-2s)_+\right).
			\end{array}
\]
\end{corollary}
\begin{proof}
The proof follows from that of \Cref{them:local}. \qed
\end{proof}

As mentioned before, \ref{model_dopt} and its complement \ref{model_doptc} have the same continuous relaxation value, whereas \ref{model_mesp} and its complement 
\ref{model_mespc}
do not. This fact is captured by \Cref{cor:ldb}, where we demonstrate a symmetric optimality gap of $\hat{z}_R$ and non-symmetric gaps of $\hat{z}_M$ and $\hat{z}^c_M$.

\subsection{A randomized-sampling algorithm}
In this subsection, we study a randomized-sampling algorithm which relies on the optimal continuous-relaxation solutions of \ref{model_dopt}, \ref{model_mesp}, and \ref{model_mespc}. Given an optimal continuous-relaxation solution $\bm{\hat x}$ of either problem, our Algorithm \ref{algo:sampling} samples a cardinality-$s$ subset $S \subseteq [n]$ with appropriate probability.
\begin{algorithm}[htbp]
	\caption{Sampling Algorithm} \label{algo:sampling}
	\begin{algorithmic}[1]
		\STATE \textbf{Input:} 
	An optimal solution $\hat{\bm x}\in [0,1]^n$ of the continuous relaxation of \ref{model_dopt}, \ref{model_mesp}, or \ref{model_mespc} and a positive integer $s\in [n]$
		\STATE For a cardinality-$s$ subset $S\in [n]$, its probability to be chosen is 
		\begin{align*}
		\mathbb{P}[\tilde{S}=S] = \frac{\prod_{i \in S} \hat{x}_i}{ \sum_{\bar{S} \in \binom{[n]}{s} } \prod_{i \in \bar{S}} \hat{x}_i }~.\\[-30pt]
		\end{align*}
		\STATE \textbf{Output:} Random set $\tilde{S}$
	\end{algorithmic}
\end{algorithm}

The sampling procedure has $O(n\log n)$ complexity, and its detailed efficient implementation can be found in \cite[Section 3.1]{singh2020approximation}. The approximation bound of the output of \Cref{algo:sampling} depends on the choice of the relaxation, as presented in \Cref{them:samp}. 


\vbox{
\begin{restatable}{theorem}{themsamp}\label{them:samp}
	Let $\bm{\hat x}_D, \bm{\hat x}_M, \bm{\hat x}_{M}^c$ denote optimal continuous-relaxation solutions of \ref{model_dopt}, \ref{model_mesp}, and \ref{model_mespc}, respectively. Suppose that Algorithm \ref{algo:sampling} generates random sets $\tilde{S}_D$, $\tilde{S}_M$, $\tilde{S}_{M^c}$ with $\bm{\hat x}_D, \bm{\hat x}_M, \bm{\hat x}_{M}^c$ as inputs, respectively, then we have
	\begin{enumerate}[(i)]
	 \item $\log \mathbb{E} \left[ \det\left(\bm{C}+\sum_{i \in \tilde{S}_D} \bm{a}_i \bm{a}_i^{\top} \right)\right] \ge z^* + n\log(x_{\min}) -(n-s)\log(1+\delta)$,
	where $x_{\min}>0$ is the least non-zero entry in $\bm{\hat x}_D$ and $\delta:=\lambda_{\max}(\bm A^{\top} \bm C^{-1}\bm A) $;
	
\item $\log \mathbb{E} \left[ \det\left(\bm{C}+\sum_{i \in \tilde{S}_M} \bm{a}_i \bm{a}_i^{\top} \right) \right] \ge z^*-s\log\left(\frac{s}{n}\right)- \log \left(\binom{n}{s}\right) $;

\item $\log \mathbb{E} \left[ \det\left(\bm{C}+\sum_{i \in \tilde{S}_{M}^c} \bm{a}_i \bm{a}_i^{\top} \right) \right] \ge z^*-(n-s)\log\left(\frac{n-s}{n}\right)- \log \left(\binom{n}{n-s}\right) $;
	
\item	 Algorithm \ref{algo:sampling} can be derandomized as a polynomial-time algorithm with the same performance guarantees.
	\end{enumerate}
\end{restatable}

}

		We establish the approximation bounds for \Cref{algo:sampling}, using the solutions of \ref{model_dopt}, \ref{model_mesp}, and \ref{model_mespc}. The detailed proof can be found in Appendix \ref{proof:themsamp}. 
We further remark:
\begin{enumerate}[(i)]
\item The performance of \Cref{algo:sampling} depends on the quality of the continuous-relaxation solution $\hat{\bm x}$, i.e., a tighter continuous relaxation bound yields better sampling results. It can be also seen that the running time of \Cref{algo:sampling} is dependent of the dimensionality of the data, i.e., $d$;
\item Using the optimal continuous-relaxation solutions from \ref{model_mesp} and \ref{model_mespc}, the output from \Cref{algo:sampling} can be at most $\bar{s}\log (\bar{s}/n)+\log (\binom{n}{\bar{s}})$ away from the optimal value (recall that $\bar{s}:=\min\{s, n-s\}$), which improves the approximation bound (${s}\log (\bar{s}/n)+\log (\binom{n}{{s}})$) of \Cref{algo:sampling} for solving \ref{MESP} in \cite{li2020best}. The comparison between these two bounds is displayed in \Cref{fig_comp_upper}; and
\item Similar to the corollaries of \Cref{them:local}, \Cref{them:samp} also implies the optimality of \Cref{algo:sampling} for two special cases in \Cref{cor:samp} and alternative optimality gaps of the proposed continuous relaxation values in \Cref{cor:upper}.
\end{enumerate}

\begin{figure}[hbtp]
	\centering
	
	\subfigure[{$n$=10}] {\label{LS_bound}
		\includegraphics[width=7.5cm,height=5.3cm]{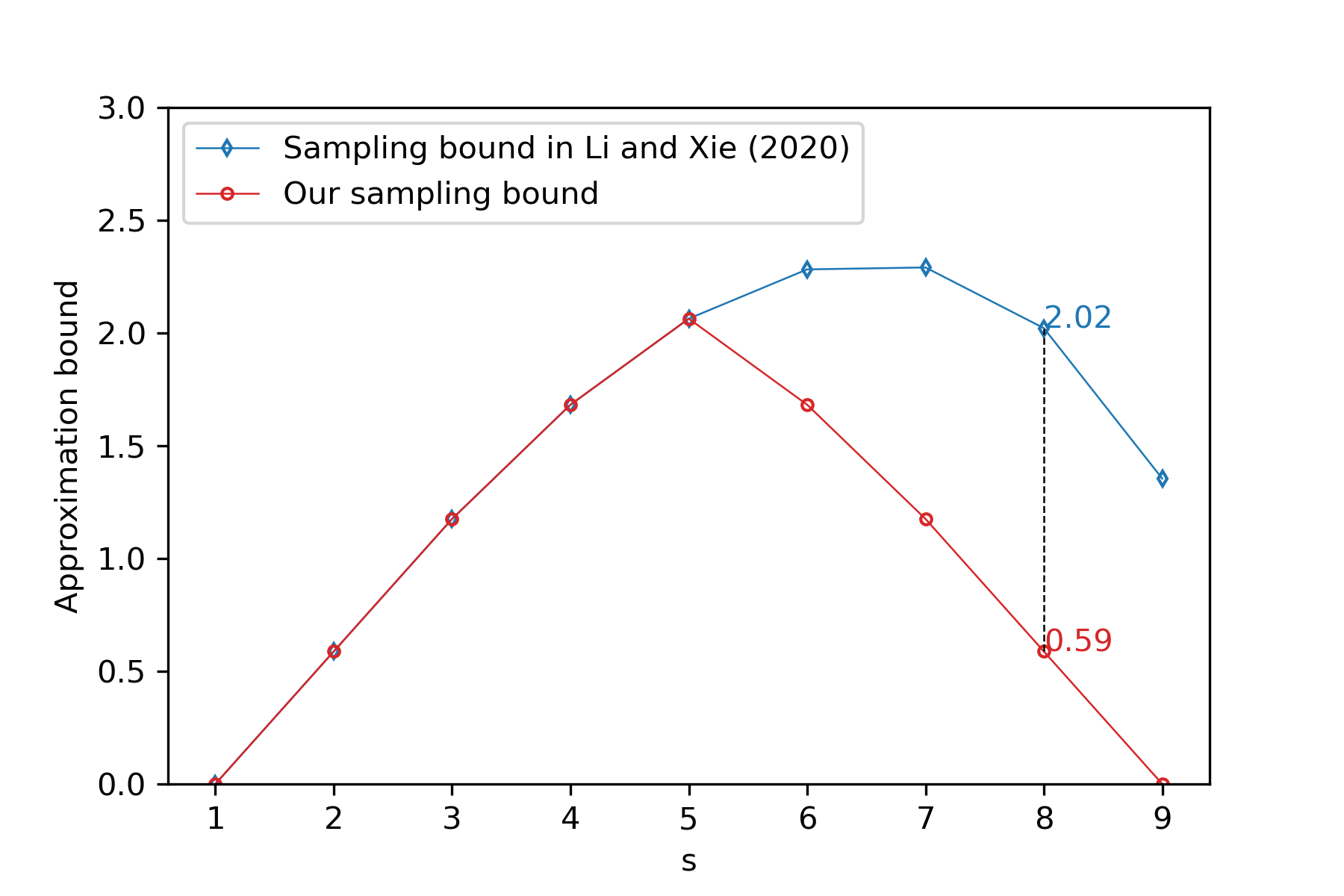}
	}
	\hspace{1em}
	\subfigure[{$n$=50}] {\label{sample_bound}
		\centering
		\includegraphics[width=7.5cm,height=5.3cm]{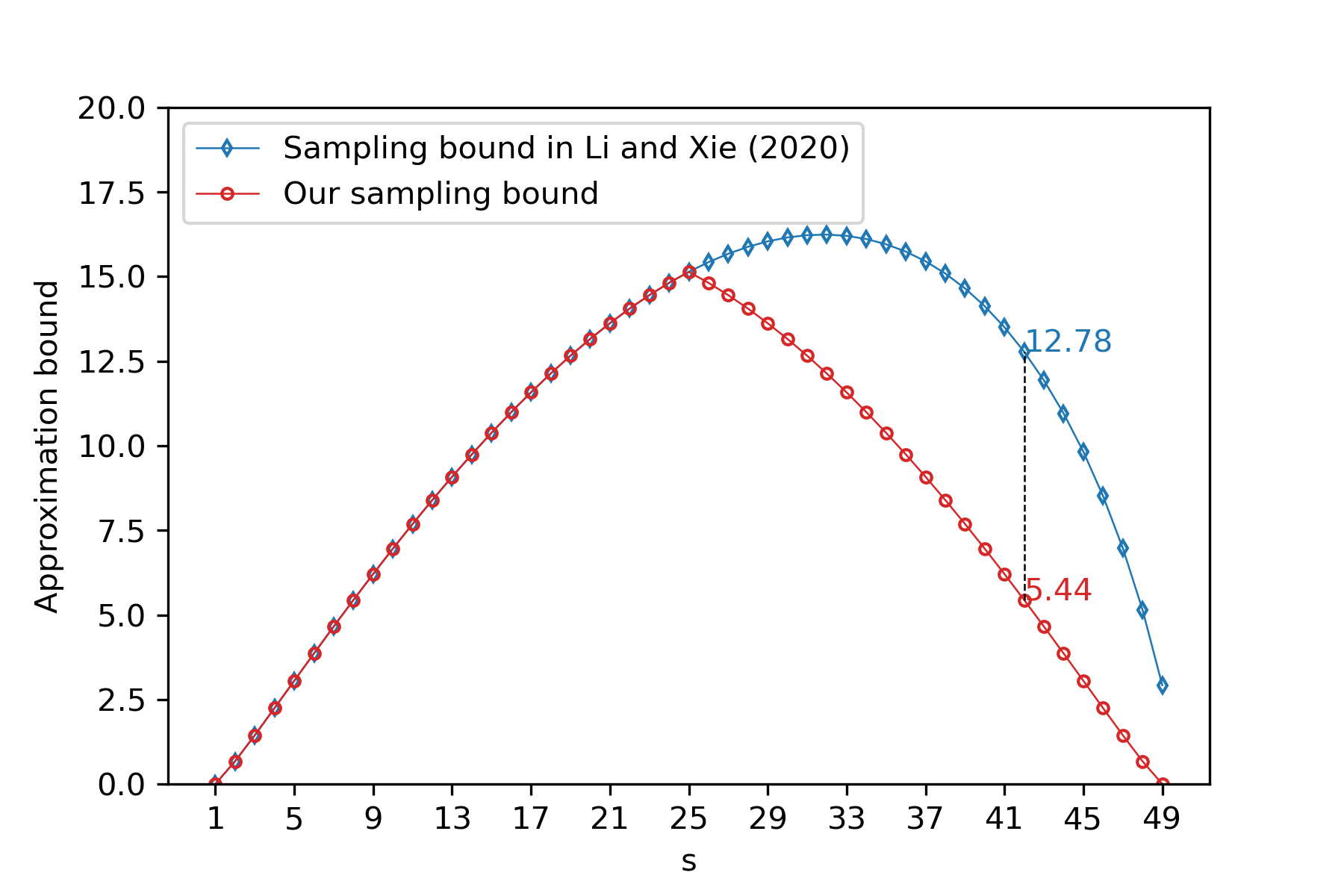}
	}
 	\caption{Comparison between our approximation bound from the sampling \Cref{algo:sampling} with the existing one }\label{fig_comp_upper}
\end{figure}

\begin{corollary} \label{cor:samp}
	When $s=1$ and $s=n-1$, \Cref{algo:sampling} returns an optimal solution when using $\hat{\bm x}_M$ and $\hat{\bm x}_M^c$ as the input, respectively.
\end{corollary}
\begin{proof}
	According to \Cref{them:samp}, if $s=1$ and $\hat{\bm x}_M$ is used as the input, the approximation bound of \Cref{algo:sampling} is equal to 0. When $s=n-1$, the same argument follows. \qed
\end{proof}

\begin{corollary} \label{cor:upper}
		The continuous relaxation values $\hat{z}_R$ of \ref{model_dopt}, $\hat{z}_M$ of \ref{model_mesp}, and $\hat{z}_M^c$ of \ref{model_mespc} satisfy
		\[
		\begin{array}{l}
z^* \le \hat{z}_R \le z^* -n\log(x_{\min}) +(n-s)\log(1+\delta);\\
z^* \le \hat{z}_M \le z^* + s\log\left(\frac{s}{n}\right)+ \log \left(\binom{n}{s}\right) ;\\
z^* \le \hat{z}_M^c \le z^* + (n-s)\log\left(\frac{n-s}{n}\right)+\log \left(\binom{n}{n-s}\right) .
		\end{array}
\]
\end{corollary}
	\begin{proof}
	The proof follows directly from that of \Cref{them:samp}. \qed
\end{proof}

Note that the theoretical optimality gaps of the three continuous relaxation values in \Cref{cor:ldb} and \Cref{cor:upper} are not comparable; thus, taking the minimum of both values yields a better optimality gap. Specifically, for the continuous relaxation value $\hat{z}_R$, the two optimality gaps from \Cref{cor:ldb} and \Cref{cor:upper} depend on parameters $\sigma_{\max}$, $\delta$, and $x_{\min}$ and thus are not comparable. The explicit comparison of alternative optimality gaps for $\hat{z}_M$ and $\hat{z}_M^c$ with $n=10, 50$ can be found in \Cref{fig_rel}, where the blue diamond line and red circle line represent the results in \Cref{cor:ldb} and \Cref{cor:upper}, respectively. We observe that (i) for either continuous relaxation value, the optimality gap in \Cref{cor:upper} is tighter than that of \Cref{cor:ldb}, except one case; and (ii) if $s \le n/2$, then $\hat{z}_M$ outperforms $\hat{z}_M^c$ and is a tighter upper bound for \ref{model}. 
\begin{figure}[ht]
	\centering
	\subfigure[{Optimality gap for $\hat{z}_M$, $n$=10 }] {
		\centering
		\includegraphics[width=7.7cm,height=5.3cm]{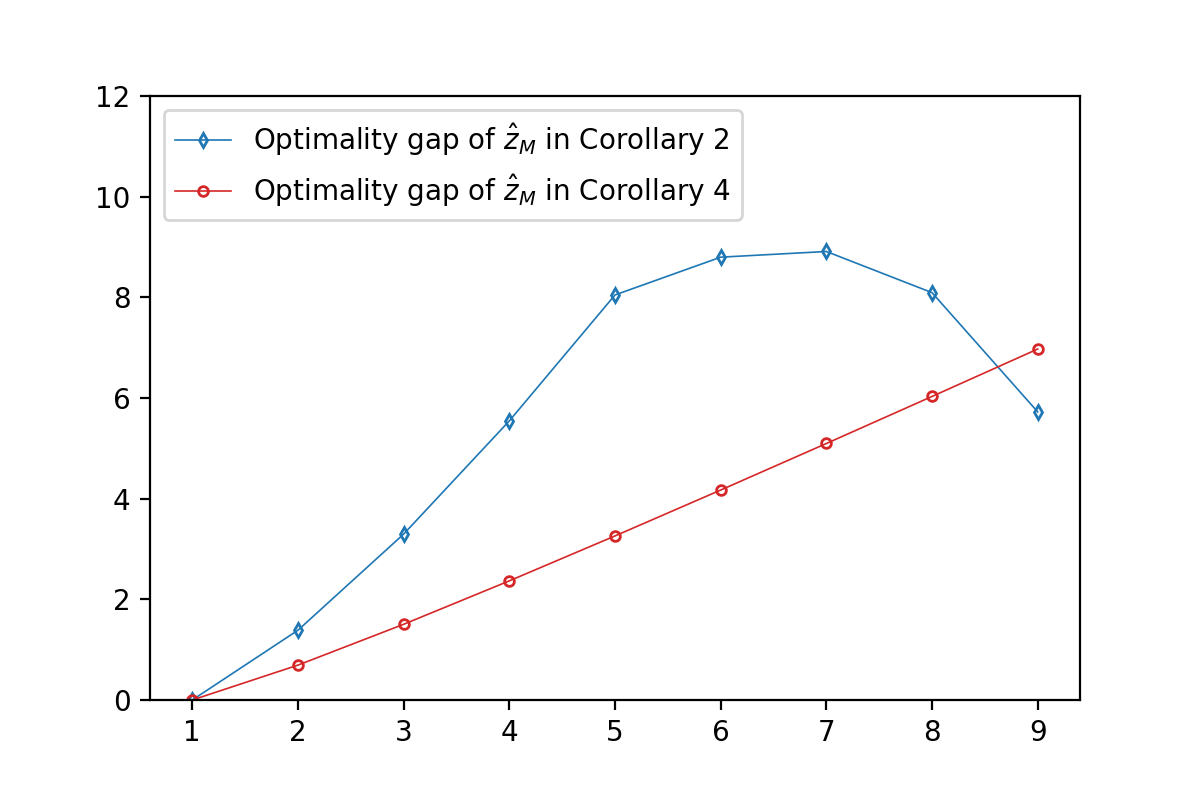}
	}
	\subfigure[{Optimality gap for $\hat{z}_M$, $n$=50}] {
		\includegraphics[width=7.7cm,height=5.3cm]{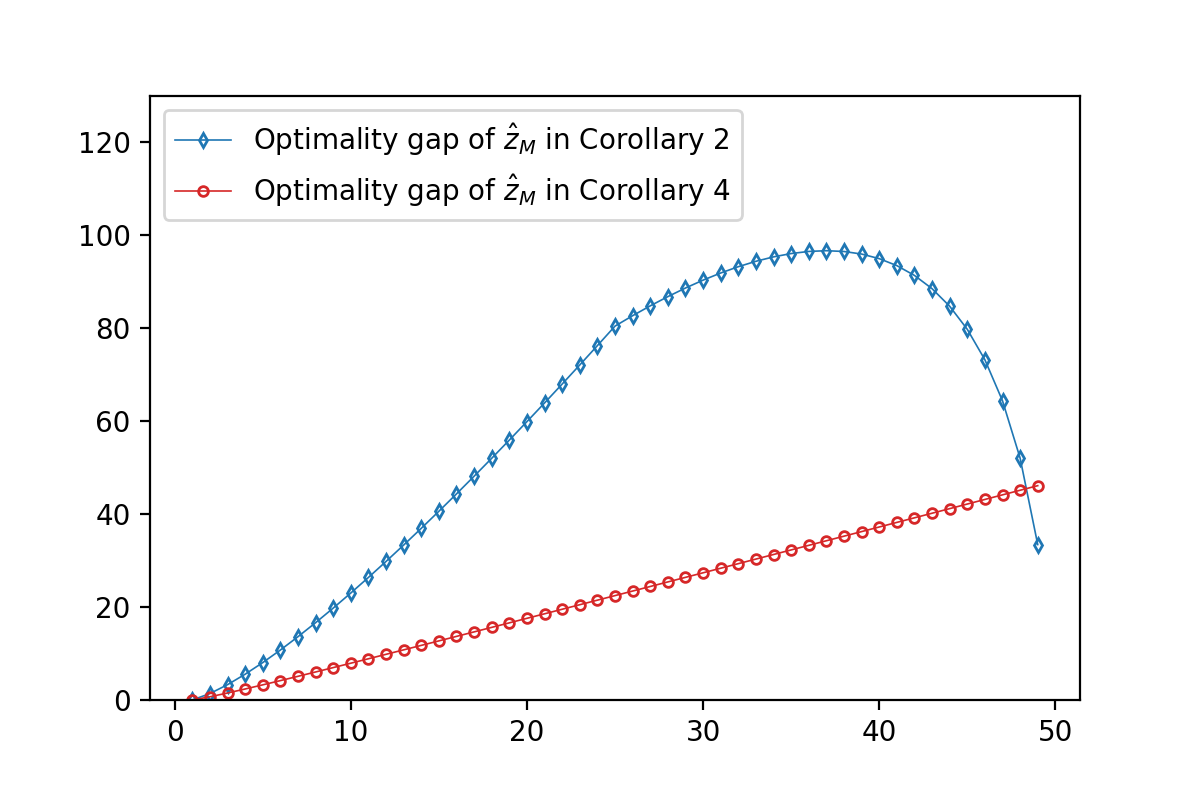}
	}
	\subfigure[{Optimality gap for $\hat{z}_M^c$, $n$=10}] {
		\centering
		\includegraphics[width=7.7cm,height=5.3cm]{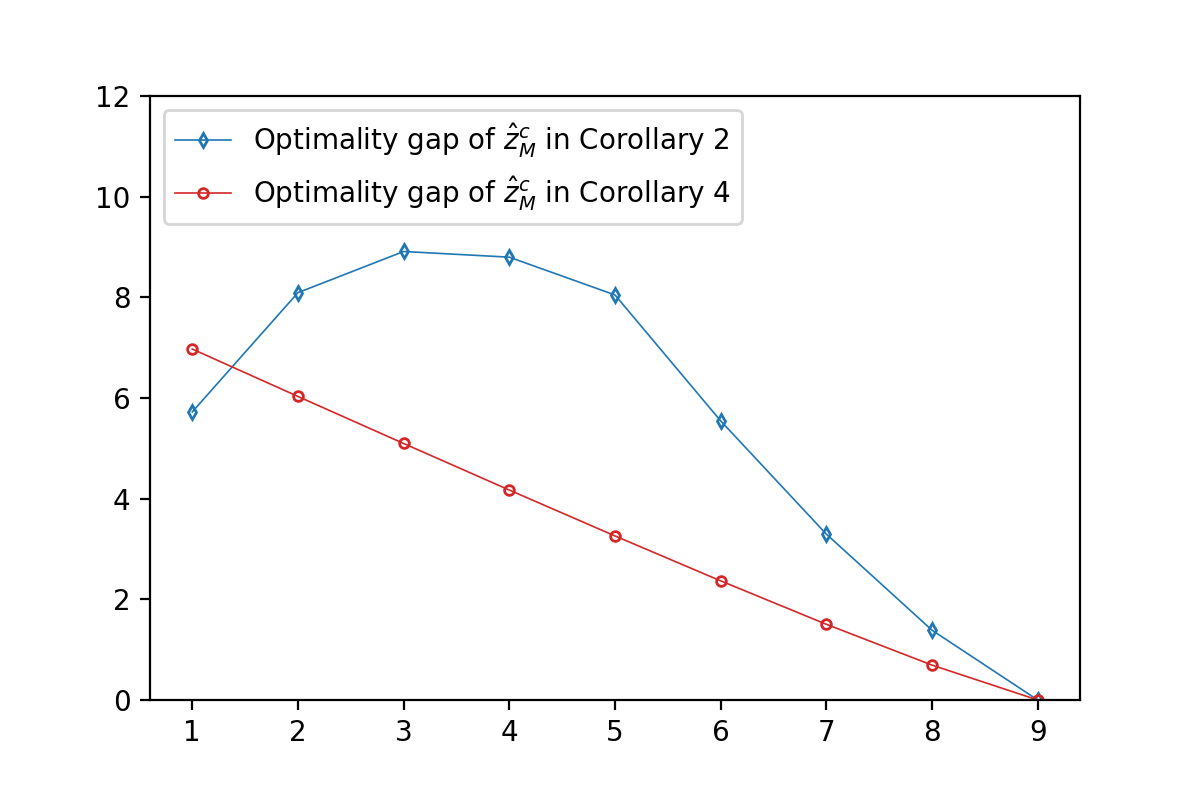}
	}
	\subfigure[{Optimality gap for $\hat{z}_M^c$, $n$=50}] {
		\centering
		\includegraphics[width=7.7cm,height=5.3cm]{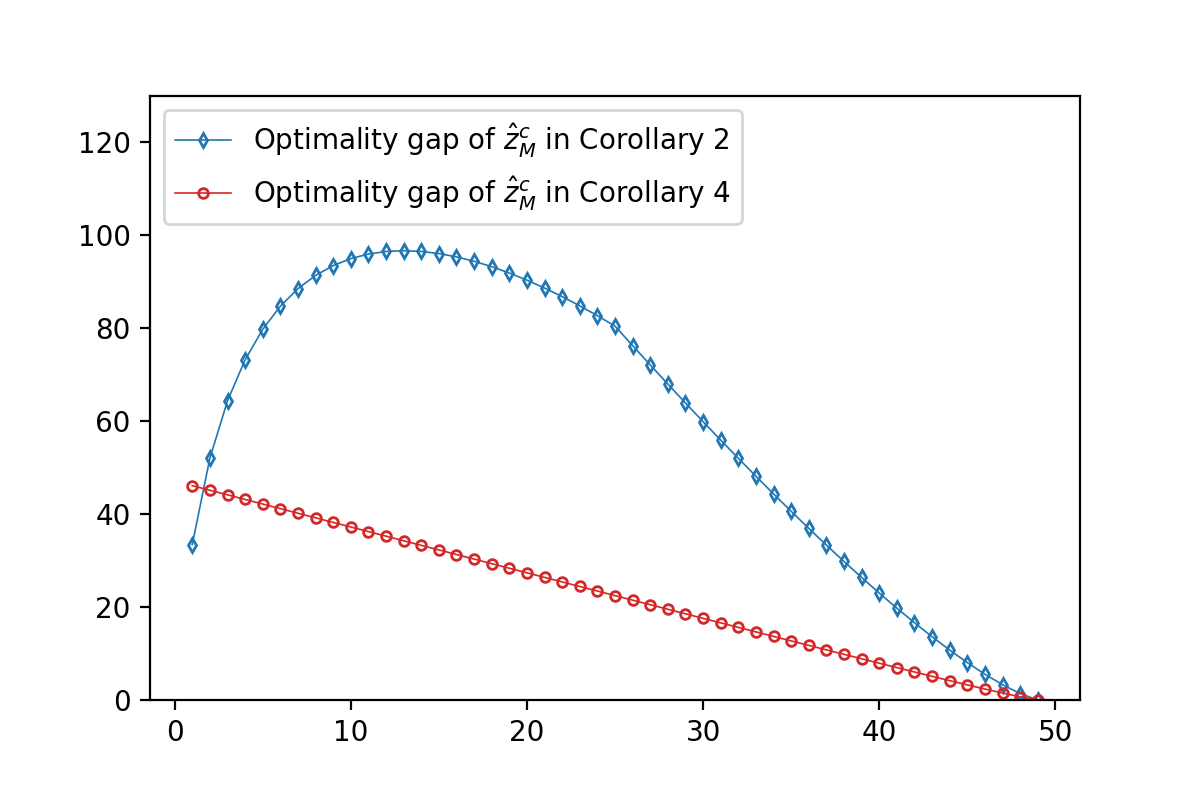}
	}
	\caption{Comparison between optimality gaps for the two continuous relaxation values $\hat{z}_M$ and $\hat{z}_M^c$ 
	}\label{fig_rel}
\end{figure}

Note that in \Cref{them:local} and \Cref{them:samp}, \Cref{cor:ldb} and \Cref{cor:upper},
 the magnitudes of $\sigma_{\max}$ and $\delta$ depend on the contained information difference of new and existing data. If the existing data are 
 very informative, i.e., $\sigma_{\max}$ and $\delta$ tend to be small, and
 the proposed approximation bounds based on \ref{model_dopt} get tighter. Hence, for Algorithm \ref{algo:sampling}, we recommend applying the continuous-relaxation solution of \ref{model_dopt} if the existing data are more informative; otherwise, applying \ref{model_mesp} and its complement.

\section{Our exact algorithmic approach for \ref{model}}
\label{sec:bc}
To solve \ref{model} to optimality, we propose an enhancement on the LP/NLP B\&B algorithm proposed in \cite{QG92}. We first 
formulate \ref{model} as 
\begin{align}\label{model_mip}\tag{DDF-MINLP}
&z^*= \ldet \bm C+ \max_{\substack{z \in \Re,\\ \bm x \in \Z_s}} \bigg\{z \!~:~\! z \le \ldet \bigg(\bm I_d+\sum_{i\in [n]} x_i\bm b_{i}\bm b_{i}^{\top} \bigg), z \le f \bigg(\sum_{i\in [n]} x_i\bm v_{i}\bm v_{i}^{\top} \bigg) \bigg\},
\end{align}
for which the optimal value of the continuous relaxation is the best bound for \ref{model} given by the continuous relaxations of both formulations \ref{model_dopt} and \ref{model_mesp}. 

Using the concavity of the objective functions of both formulations, LP/NLP B\&B considers, for a given $S \subseteq [n]$, the following linear relaxations of the two nonlinear inequalities in \ref{model_mip}:
\begin{equation} \label{eq:grad}
\begin{aligned}
\text{(Linearization)} \ \ 
& z \le \ldet \left[\bm B(S) \right] - \sum_{i \in S} \bm b_i^{\top} [\bm B(S)]^{-1} \bm b_i +\sum_{i \in [n]} \bm b_i^{\top} [\bm B(S)]^{-1} \bm b_i x_i, \\
& z \le f \left[\bm V(S) \right] - \sum_{i \in S} \bm v_i^{\top} \bm g(S) \bm v_i +\sum_{i \in [n]} \bm v_i^{\top} \bm g(S) \bm v_i x_i, 
\end{aligned}
\end{equation}
where for any subset $T\subseteq [n]$, $\bm g(T) := \bm V^{\dag}(T) + 1/\lambda_{\min}(\bm V(T)) [ \bm I_d - \bm V(T) \bm V^{\dag}(T) ]$ is a subgradient of $f(\cdot)$ at $\bm V(T) := \sum_{i\in T} \bm v_i \bm v_i^{\top}$, according to Proposition 2 in \cite{li2020best}.

Then, at iteration $\ell$ of LP/NLP B\&B, a mixed-integer linear-programming (MILP) problem $M^\ell$ is solved. The so-called master problem $M^\ell$ is a relaxation of \ref{model_mip} obtained by replacing the nonlinear inequalities in $z$ by the inequalities in \eqref{eq:grad}, constructed for all $S$ in a given set $\mathcal{S}^\ell$ of linearization points represented by elements of $\binom{[n]}{s}$. $M^\ell$ is solved by a branch-and-bound algorithm, and every time a feasible solution $\bm{\hat{x}}$ of $M^\ell$ is obtained during the execution of the algorithm, $\supp(\bm{\hat{x}})$ is included in $\mathcal{S}^\ell$. In \cite{MFR20}, the authors highlight that LP/NLP B\&B can be efficiently implemented in advanced MILP packages (e.g., CPLEX, Gurobi) with the use of lazy constraints and callback functions. As the set of linearization points increases at each iteration of LP/NLP B\&B, the solution values of $M^\ell$ form a non-increasing sequence of upper bounds for \ref{model_mip} and the algorithm stops when the difference between this upper bound and the best known lower bound is small enough (see \cite{QG92}).

\subsection{Submodular cuts}

Next, we propose a first enhancement on LP/NLP B\&B for \ref{model_mip}. We note that, because the objective functions of \ref{model_dopt} and \ref{model_mesp} are monotone (due to \eqref{eqmesp}) and submodular (via the Hadamard-Fischer inequalities: see, for example, \cite[Section 7.8, Problem 14]{johnson1985matrix}),
according to the results in \cite{wolsey1999integer,ahmed2011maximizing}, the following submodular linear inequalities are valid for \ref{model_mip}, for any $S \subseteq [n]$: 
\begin{equation} \label{eq:submod}
\begin{aligned}
\text{(Submodular)} \quad &z \le \ldet \left[\bm B(S) \right] - \sum_{i \in S}\rho_i(N\setminus \{i\}) (1-x_i) + \sum_{i \in [n]\setminus S}\rho_i(S) x_i , \\
& z \le \ldet \left[\bm B(S) \right] - \sum_{i \in S}\rho_i(S\setminus \{i\}) (1-x_i) + \sum_{i \in [n]\setminus S}\rho_i(\emptyset) x_i ,
\end{aligned}
\end{equation}
where for any $T\subseteq [n]$ and $i\in [n]\setminus T$, we define the difference function $\rho_i(T):= \ldet \left[\bm B(T) + \bm b_i \bm b_i^{\top}\right] - \ldet \left[\bm B(T) \right]$ with $\bm B(T) := \bm{I}_d+\sum_{i\in T} \bm b_{i}\bm b_{i}^{\top}$.

Moreover, by 
using the identity for rank-one update of the determinant of a symmetric matrix, we have the closed-form expression for the difference function $\rho_i(T) = \log[1+\bm b_i^{\top} [\bm A(T)]^{-1} \bm b_i ]$. Similarly, for any subset $T\subseteq [n]$ and $i\in T$, we have $\rho_i(T\setminus\{i\}) = -\log[1-\bm b_i^{\top} [\bm A(T)]^{-1} \bm b_i ]$. Hence, the constraint coefficients involved with the difference functions can be easily computed.

Then, to tighten the MILP relaxation of \ref{model_mip} at LP/NLP B\&B, besides including the standard linearization inequalities \eqref{eq:grad} in $M^\ell$, we also include the submodular inequalities \eqref{eq:submod}, for each set $S$ added to $\mathcal{S}^\ell$.

\subsection{Optimality cuts}

We consider choosing one or multiple data points and fixing their corresponding binary variables in \ref{model_mip} to either one or zero, and then probing the restricted DDF to derive effective optimality cuts on these binary variables, which can help significantly reduce the size of the feasible region of \ref{model} while maintaining the optimal value. 
Specifically, suppose that sets $S^1, S^0\subseteq [n]$ denote the index set of data points being selected (i.e., $x_i=1$ for each $i\in S^1$) and being discarded (i.e., $x_i=0$ for each $i\in S^0$), respectively. Then a restricted problem of \ref{model} is defined as
\begin{equation} \label{eq:reduce}
\begin{aligned}
\text{(Restricted DDF)} \quad z(S^1, S^0):= \max_{\substack{S\subseteq [n]\setminus (S^1\cup S^0),\\ |S|=s-|S^1|}}\ldet\bigg(\bm C+\sum_{i\in S^1 } \bm a_i \bm a_i + \sum_{i\in S} \bm a_i \bm a_i\bigg),
\end{aligned}
\end{equation}
where sets $S^1, S^0$ are disjoint and set $S^1$ is of cardinality no larger than $s$.

Clearly, if $S^1=S^0=\emptyset$, then $z(S^1, S^0) = z^*$ and the formulation is equivalent to \ref{model}. Otherwise, if $z(S^1, S^0) < z^*$, then 
at least one constraint built on sets $S^1$, $S^0$ is violated by an optimal solution and we thus obtain an optimality cut for \ref{model_mip} that cuts off a subset of sub-optimal solutions. 
This result is summarized below.

\begin{theorem}\label{them:fix}
	For any two disjoint sets $S^1, S^0$ with $|S^1|\le s$ and $S^1\cup S^0\subseteq [n]$, if $z(S^1, S^0)< z^*$, then at least one of the two inequalities below is an optimality cut of \ref{model_mip}.
\begin{align}\label{eq:valid}
 \sum_{i \in {S^1}}x_i \le |S^1|-1, \ \ \sum_{j \in {S^0}}x_j \ge 1,\ \ \forall \bm x \in \Z_s.
\end{align}
\end{theorem} 
\begin{proof}
Given an optimal solution $\bm x^*$ of \ref{model_mip}, if $\bm x^*$ satisfies the constraints $\sum_{i \in {S^1}}x_i^* = |S^1|$ and $\sum_{j \in {S^0}}x_j^* = 0$, 
 then $\supp(\bm x^*)$ will be feasible to the restricted problem \eqref{eq:reduce} with the same objective value $z^*$, which contradicts $z(S^1, S^0) < z^*$. Therefore, $\bm x^*$ must violate at least one of the equality constraints above and using the fact that $\bm x^*$ is binary, we complete the proof.
 \qed
\end{proof}

For the result of \Cref{them:fix}, we remark that (i) if either $S^1$ or $S^0$ is empty, then the inequality in \eqref{eq:valid} based on the non-empty set must be an optimality cut; (ii)
 if one of sets $S^1$ and $S^0$ is singleton and the other one is empty, the inequalities in \eqref{eq:valid} recover the well-known variable-fixing ones \cite{fischetti2010heuristics}; and (iii) if both sets $S^1$ and $S^0$ are non-empty, the optimality cuts in \eqref{eq:valid} can be enforced via disjunctive programming.
 The optimality cuts are effective at reducing the search space and significantly improve the LP/NLP B\&B algorithm as shown in our numerical results. 

Albeit being effective, a common criticism of the probing technique in mixed-integer programming is its computation expense \cite{achterberg2020presolve}, e.g., the optimal values $z(S^1, S^0)$ and $z^*$ in \Cref{them:fix} may not be easily computable.
Motivated by our near-optimal approximation algorithms and strong Lagrangian dual bounds for \ref{model}, a compromise is that if an upper bound for $z(S^1, S^0)$
is less than a lower bound for $z^*$ (denoted by $z^{lb}$), then the conclusion in \Cref{them:fix} holds. Besides, following the spirit of the two Lagrangian duals for the continuous relaxations of \ref{model} in \Cref{sec:cip}, the restricted problem \eqref{eq:reduce} also admits two alternative upper bounds as follows, corresponding to Lagrangian dual problems \eqref{eq:dopt_ld} and \eqref{eq:mesp_ld}, respectively.
	\begin{equation}\label{eq:reduceld}
\begin{aligned}
&\hat{z}_R(S^1, S^0):= \ldet \bm C+ \min_{\bm{\Lambda} \in \S_{++}^d, \nu, \bm{\mu} \in \Re^{n}_+} 	\bigg\{-\ldet \bm{\Lambda} + \tr \bm{\Lambda} +\nu +\sum_{i \in [n]}\mu_i -d \\
&\qquad+\bigg(\sum_{j\in S^1} (\bm b_j^{\top} \bm \Lambda \bm b_j -\nu -\mu_j)\bigg) +\bigg( -\sum_{l\in S^0} \mu_l\bigg): 
\bm{b}_i^{\top}\bm{\Lambda}\bm{b}_i \le \nu+ \mu_i, \forall i \in [n]\setminus \{S^1\cup S^0\} \bigg\}, \\
&\hat {z}_{M}(S^1, S^0) := \ldet \bm C + \min_{\bm \Lambda \in \S_{++}^n, \nu, \bm{\mu} \in \Re^n_+} \bigg \{ -\log \det\limits_s(\bm \Lambda) + s\nu + \sum_{i\in [n]}\mu_i -s\\
&\qquad+\bigg(\sum_{j\in S^1} (\bm v_j^{\top} \bm \Lambda \bm v_j -\nu -\mu_j)\bigg) +\bigg( -\sum_{l\in S^0} \mu_l\bigg): \bm v_i^{\top}\bm \Lambda \bm v_i \le \nu + \mu_i, \forall i \in [n] \setminus \{S^1\cup S^0\}\bigg \}.
\end{aligned}
\end{equation}

We observe that for some appropriately selected sets $S^1$ and $S^0$, the Lagrangian dual bounds \eqref{eq:reduceld} can be smaller than the lower bound of \ref{model}, i.e., $\hat{z}_R(S^1, S^0)<z^{lb}$ or $\hat {z}_{M}(S^1, S^0)< z^{lb}$. Our selection strategy of sets $S^1$ and $S^0$
 is a unification of the primal and dual perspectives, with an aim of reducing values: $\hat{z}_R(S^1, S^0)$ and $\hat {z}_{M}(S^1, S^0)$, which is discussed below.\par
\begin{enumerate}[(i)]
\item \textbf{Primal:} Given an optimal solution $\hat{\bm x}$ to the continuous relaxation of \ref{model_dopt} or \ref{model_mesp}, we let $S^1\subseteq \{i: \hat{x}_i\le \xi^0, \forall i \in [n]\}$ with $\xi^0\in [0,1]$ being a positive number close to 0 and $S^0\subseteq \{i: \hat{x}_i\ge \xi^1, \forall i \in [n]\}$ with $\xi^1\in [0,1]$ being close to 1. In this case, we expect a big reduction on the restricted Lagrangian dual bounds in \eqref{eq:reduceld}, when compared to the unrestricted bounds $\hat{z}_R$ and $\hat {z}_{M}$. 
Our numerical experiments suggest that this selection strategy performs very well in exploring appropriate subset $S^1$ to construct an optimality cut.

\item \textbf{Dual:} Using Lagrangian dual formulations \eqref{eq:dopt_ld} and \eqref{eq:mesp_ld}, 
 we can ensure that the Lagrangian dual bounds of restricted DDF problem \eqref{eq:reduce} decrease by at least a given threshold. 
 According Lagrangian dual formulations in \eqref{eq:reduceld}, given an optimal dual solution $(\hat{\bm \Lambda}, \hat{\nu}, \hat{\bm \mu})$ of \eqref{eq:dopt_ld} or \eqref{eq:mesp_ld}, we see that the corresponding restricted Lagrangian dual bound achieves a reduction of at least $ \sum_{j\in S^1} (\bm b_j^{\top} \hat{\bm \Lambda} \bm b_j -\hat{\nu} -\hat{\mu}_j) -\sum_{l\in S^0} \hat\mu_l$ or $\sum_{j\in S^1} ( \bm v_j^{\top} \hat{\bm \Lambda} \bm v_j -\hat{\nu} -\hat{\mu}_j) -\sum_{l\in S^0} \hat\mu_l$, compared with the original optimal dual value (see \Cref{prop:reduceld} below). This inspires us to identify sets $S^1$ and $S^0$ satisfying $ \sum_{j\in S^1} (\bm b_j^{\top} \hat{\bm \Lambda} \bm b_j -\hat{\nu} -\hat{\mu}_j) <0$ or $\sum_{j\in S^1} ( \bm v_j^{\top} \hat{\bm \Lambda} \bm v_j -\hat{\nu} -\hat{\mu}_j) <0$, and $\sum_{l\in S^0} \hat{\mu}_l >0$, such that each restricted Lagrangian dual problems in \eqref{eq:reduceld} yields a smaller upper bound for restricted DDF \eqref{eq:reduce} than the original Lagrangian dual value.
 It is worth mentioning that we can warm-start the solution procedure of the restricted dual problems \eqref{eq:reduceld} by using the optimal solution $(\hat{\bm \Lambda}, \hat{\nu}, \hat{\bm \mu})$ of the original Lagrangian dual problem \eqref{eq:dopt_ld} or \eqref{eq:mesp_ld}. 
We also observe in the numerical study that the dual selection strategy is good at exploring an appropriate subset $S^0$ to construct an optimality cut.
\end{enumerate}

\begin{proposition}\label{prop:reduceld}
Let $(\hat{\bm \Lambda}, \hat{\nu}, \hat{\bm \mu})$ be a feasible solution of Lagrangian dual problem \eqref{eq:dopt_ld} (resp., \eqref{eq:mesp_ld}) with the objective value $z^{ub}$. Then, $(\hat{\bm \Lambda}, \hat{\nu}, \hat{\bm \mu})$ is also feasible to its corresponding restricted Lagrangian dual problem in \eqref{eq:reduceld} and the resulting objective value is equal to
	\[ z^{ub} + \sum_{j\in S^1} (\bm b_j^{\top} \hat{\bm \Lambda} \bm b_j -\hat{\nu} -\hat{\mu}_j) -\sum_{l\in S^0} \hat\mu_l \quad \left(\text{resp., }\; z^{ub} + \sum_{j\in S^1} ( \bm v_j^{\top} \hat{\bm \Lambda} \bm v_j -\hat{\nu} -\hat{\mu}_j) -\sum_{l\in S^0} \hat\mu_l\right). \]
\end{proposition}
\begin{proof}
	If $(\hat{\bm \Lambda}, \hat{\nu}, \hat{\bm \mu})$ is a feasible solution of Lagrangian dual problem \eqref{eq:dopt_ld}, it is easy to check that it is also feasible to the first optimization problem in \eqref{eq:reduceld} whose objective value is $z^{ub} + \sum_{j\in S^1} (\bm b_j^{\top} \hat{\bm \Lambda} \bm b_j -\hat{\nu} -\hat{\mu}_j) -\sum_{l\in S^0} \hat\mu_l$. Similarly, for the Lagrangian dual problem \eqref{eq:mesp_ld}, the same result holds by replacing $\bm b_i$ by $\bm v_i$ for all $i\in [n]$.
\qed
\end{proof}

We note that (i) our selection strategies are easy-to-implement because each continuous relaxation of \ref{model} and its corresponding Lagrangian dual problem can be efficiently solved by the primal-dual Frank-Wolfe algorithm with a sublinear rate of convergence, (ii) the primal and dual selection strategies do not dominate each other and are complementary as shown in our numerical experiments, and (iii) all the analyses and selection strategies can be directly extended to complementary formulations of \ref{model}.

\section{Numerical study on sensor fusion in power systems} \label{sec:pmu}
In this section, we present a real-world sensor-fusion problem in power systems, which can be formulated as \ref{model}. We tested the proposed formulations and algorithms with varying-scale instances. All the experiments were conducted in Python 3.6 with calls to Gurobi 9.0 on a PC with 2.8 GHz Intel Core i5 processor and 8G of memory. All times reported are wall-clock times.

In power systems, phasor measurement units (PMUs) are the most accurate and high-speed time-synchronized devices used to measure phasors of bus voltages and currents in an electric grid (see \cite{liu2001state}). PMUs have broad applications, including state estimation, security assessment, system monitoring, and wide-area control (see \cite{de2010synchronized,theodorakatos2017application}). In particular, reliable state estimation is an essential component of managing a modern energy system, aiming at determining the true voltages at all buses, based on available measurements and information that consist of observed voltage angles, power flows, and injections (see \cite{terzija2010wide}). 

Before the advent of PMUs, the conventional sensors, including supervisory control and data acquisition (SCADA) meter readings, were widely used to perform state estimation in power systems (see \cite{monticelli2000electric}). It is recognized that, in many cases, one barely equips a power grid with a sufficient number of PMUs to achieve the state estimation fully (see \cite{zhou2006alternative}), due to the budget and resource constraints. Therefore, 
it is important to make an intelligent choice of fusing PMUs and SCADA when deciding the PMU locations out of non-reference buses, in order to collect maximum information to best improve the state estimation. According to the PMU and SCADA measurement model in \cite{li2011phasor}, for an $n$-bus power system, the overall Fisher information matrix (FIM) is defined as
\begin{align}
\bm{C} + \sum_{i \in S} \frac{1}{\hat{\sigma}_i^2} \bm{e}_i\bm{e}_i^{\top}, \label{eq:PMU_cov}
\end{align}
where matrix $\bm{C}\in \S_{++}^{n }$ denotes the FIM obtained from conventional sensors and is positive-definite, the set $S\subseteq [n]$ denotes the bus locations of installed PMUs, and for each $i\in [n]$, $\hat{\sigma}_i$ denotes the standard variance of PMU measurements at $i$-th bus.
The FIM \eqref{eq:PMU_cov} of the sensor fusion problem in power systems can reduce to the objective matrix in \ref{model} by letting $\bm a_i=1/\hat{\sigma_i} \bm e_i$ for each $i\in [n]$ and $d=n$.
When employing D-optimality as the information selection criterion of the sensor fusion problem, based on the FIM \eqref{eq:PMU_cov}, it follows that \ref{model} exactly formulates this sensor fusion problem (see \cite{li2011phasor}).

\subsection{A comparison of continuous relaxations: IEEE 118- and 300-bus instances}
From our theoretical analysis, we see that
a tight continuous-relaxation bound for \ref{model} has an important role in the implementation of \Cref{algo:sampling} and the derivation of optimality cuts. Therefore, we first investigated the three continuous-relaxation bounds of \ref{model_dopt}, \ref{model_mesp}, and \ref{model_mespc}, i.e., $\hat{z}_R$, $\hat{z}_M$, and $\hat{z}_M^c$, respectively, using two IEEE benchmark instances with $118$ and $300$ buses (see \cite{aminifar2009contingency}) of the PMU placement problem that provide the matrix $\bm C$ in \ref{model}.
To compare the three alternative continuous-relaxation values, we conducted a controlled experiment with respect to the PMU standard deviations $\{\hat{\sigma}_i\}_{i\in [n]}$, where large and small PMU standard deviations separately represent the two cases where either the existing sensors or the new sensors are more accurate for state estimation. 


We used the Frank-Wolfe algorithm to compute the three upper bounds on the optimal value of \ref{model}.
The computational results for the two instances are displayed in Figure \ref{fig_gap118} and Figure \ref{fig_gap300}, where the optimality gap is equal to the difference between an upper bound and a lower bound for \ref{model} returned by our local-search \Cref{algo:localsearch}.
We note that the Frank-Wolfe algorithm and our local-search \Cref{algo:localsearch} are very efficient, and their computational time is negligible (i.e., less than one minute), so we do not report them.


For the 118-bus instance (\Cref{fig_gap118}), we consider cases where the number of installed PMUs $s\in \{10, 15, \cdots, 105\}$, for $n=118$, in order to compare the upper bounds for a wide range of the parameter $s$. In Figure \ref{fig_gap118}(a), we sample the PMU standard deviations $\{\hat{\sigma}_i\}_{i\in [n]}$ as independent uniform random variables in the range $[0,100]$, and the new sensors contribute less to the state estimation than the existing ones. We see that the continuous-relaxation value $\hat{z}_R$ is smaller than both $\hat{z}_M$ and $\hat{z}_M^c$ in most cases in \Cref{fig_gap118}(a). 
By contrast, in \Cref{fig_gap118}(b), we sample the PMU standard deviations as $n$ independent uniform random variables in the range $[0,0.01]$, assuming new more accurate sensors. In this setting, we see that the continuous-relaxation values $\hat{z}_M$ and $\hat{z}_M^c$ are much smaller than $\hat{z}_R$, so we use a pair of vertical axes to illustrate their performance. The comparison results parallel our theoretical findings in \Cref{sec:cip}. 

The comparison of the three upper bounds is also illustrated in \Cref{fig_gap300} for the 300-bus instance and the conclusions are the same.
Thus, both theoretical analyses and numerical comparisons in \Cref{fig_gap118} and \Cref{fig_gap300} demonstrate that the continuous relaxation of \ref{model_mesp} is more stable and tighter than that of \ref{model_dopt} when the new sensors are subject to smaller measurement variances.
In addition, we observe in both figures that $\hat{z}_M$ tends to be stronger than $\hat{z}_M^c$ if $s$ is small; otherwise $\hat{z}_M^c$ is stronger. 
 In practice, PMUs are much more accurate than other sensors (see \cite{zhao2019impact}), and the measurement error is controlled within the range of $[-0.6^{\circ},0.6^{\circ}]$; and the number of PMUs to be installed is usually small due to budget constraints, i.e., $s$ is small. Thus, we identify the continuous-relaxation value $\hat{z}_{M}$ of \ref{model_mesp} as a better upper bound for the PMU placement problem in power systems that will be used in the following numerical study, where we also set the PMU standard deviation to $0.02^{\circ}$, a known PMU standard error in the literature (see \cite{zimmerman2010matpower}).

\begin{figure}[hbtp]
	\centering
	
	\subfigure[{Large PMU standard deviation}] {
		\includegraphics[width=7.5cm,height=5.8cm]{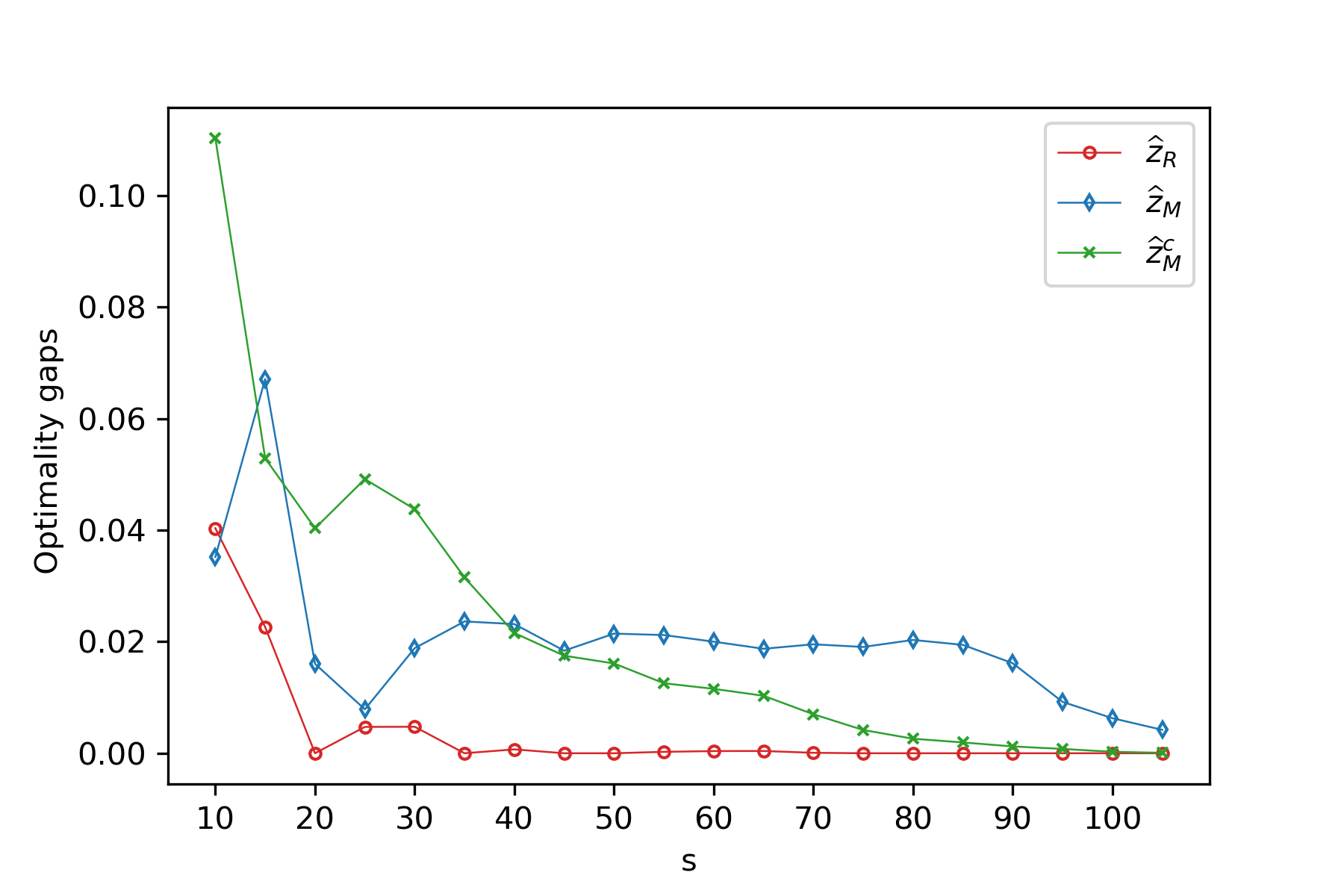}
	}
	\hspace{1em}
	\subfigure[{Small PMU standard deviation}] {
		\centering
		\includegraphics[width=7.5cm,height=5.3cm]{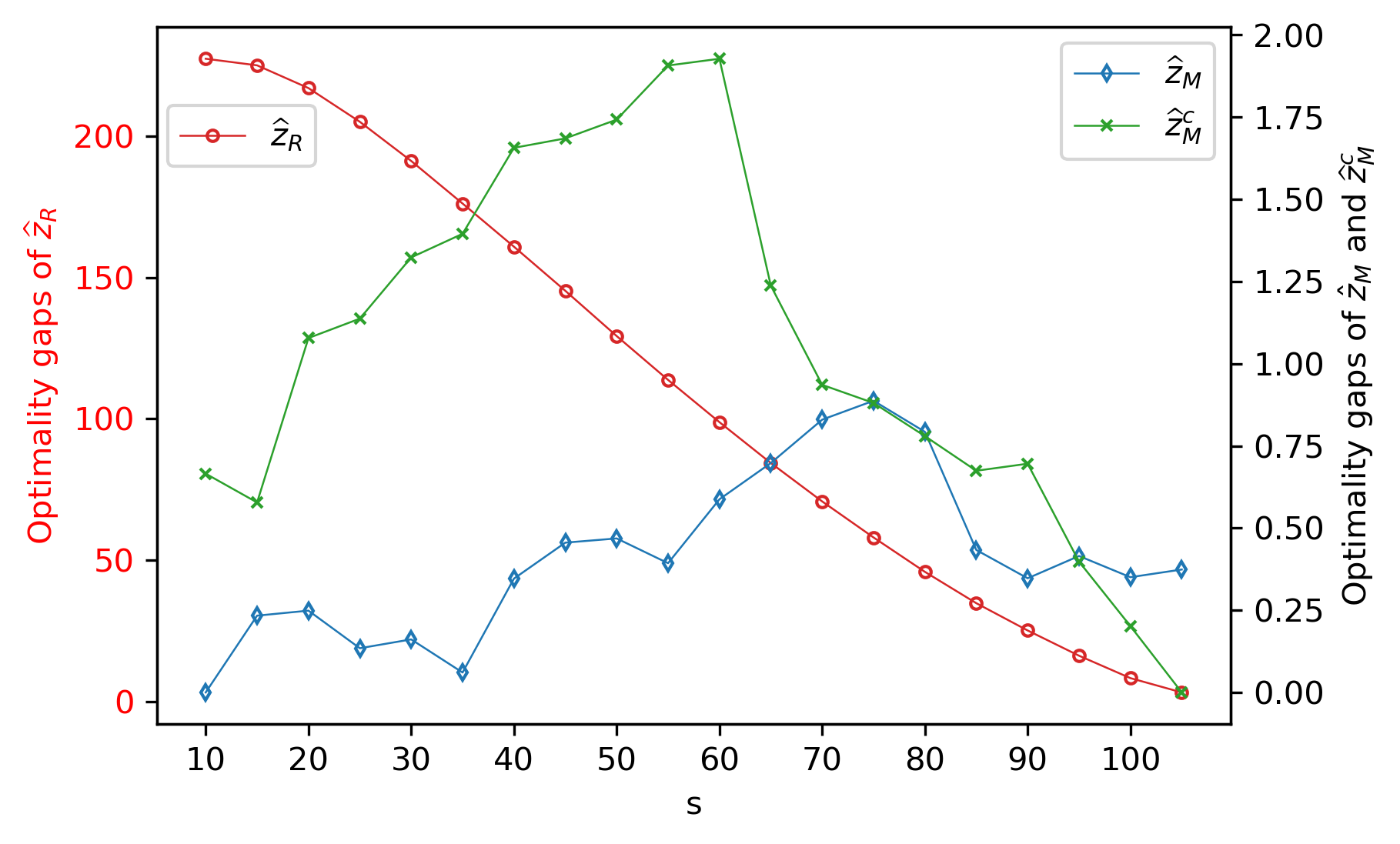}
	}
	\caption{Comparison between continuous-relaxation bounds on IEEE 118-bus instance}\label{fig_gap118}
\end{figure}

\begin{figure}[hbtp]
	\centering
	
	\subfigure[{Large PMU standard deviation}] {
		\includegraphics[width=7.5cm,height=5.8cm]{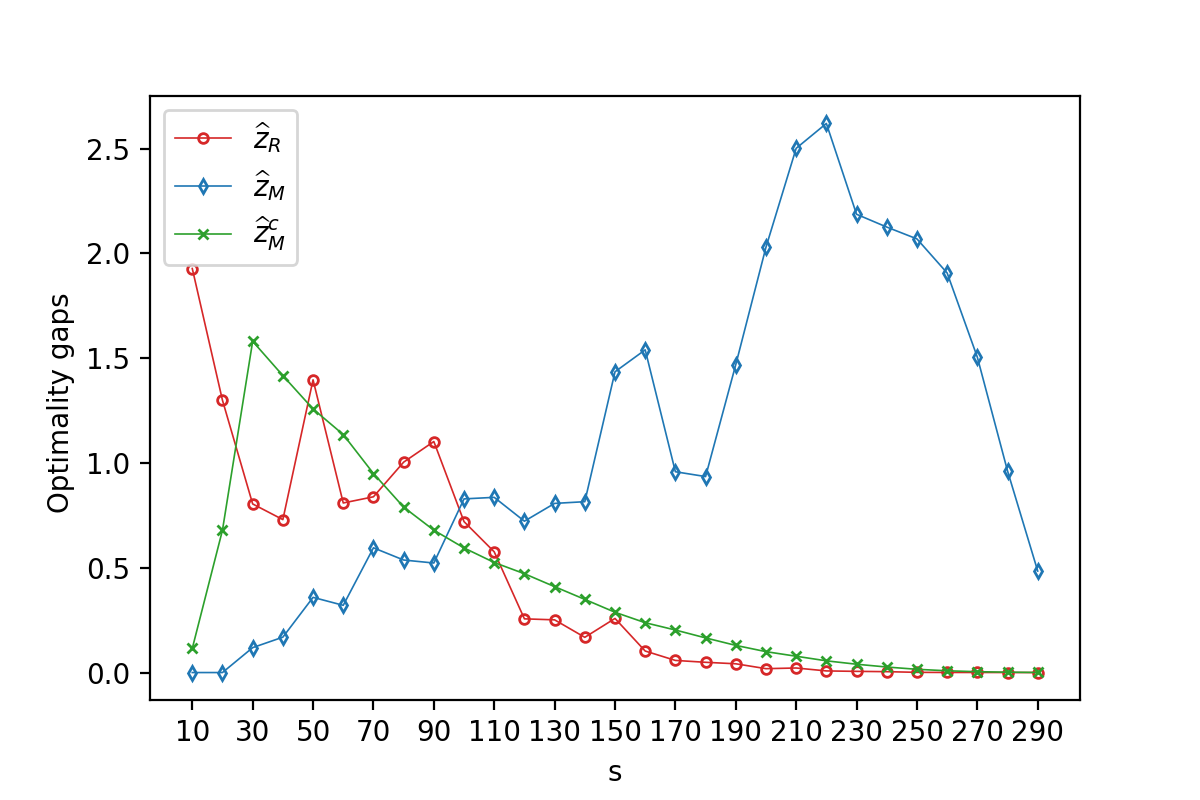}
	}
	\hspace{1em}
	\subfigure[{Small PMU standard deviation}] {
		\centering
		\includegraphics[width=7.5cm,height=5.3cm]{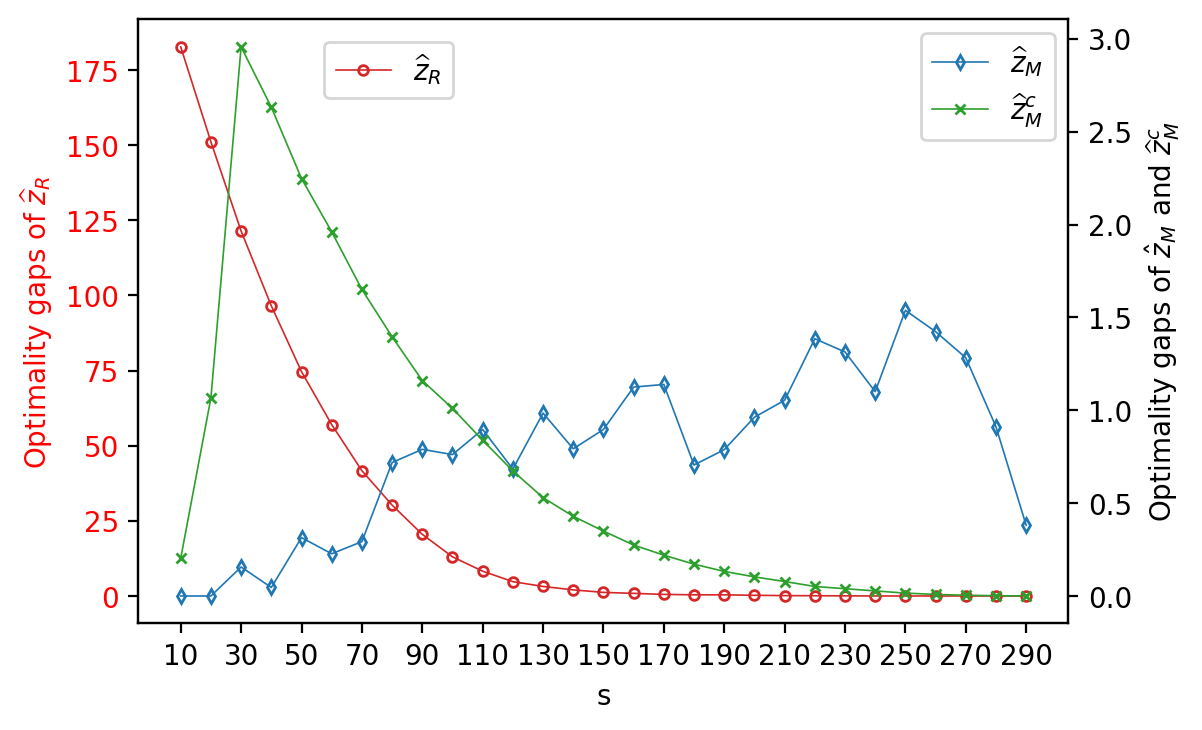}
	}
	\caption{Comparison between continuous-relaxation bounds on IEEE 300-bus instance}\label{fig_gap300}
\end{figure}

\subsection{Testing submodular and optimality cuts: IEEE 118- and 300-bus instances}
Next, 
we tested the submodular and optimality cuts introduced in \Cref{sec:bc}. 
For our experiments, the proposed primal and dual selection strategies of subsets $S^1$ and $S^0$ only depend on the continuous relaxation of \ref{model_mesp}, and its Lagrangian dual, due to their strength for the power-system instances. Similarly, for the selected subsets and their resulting restricted problems \eqref{eq:reduce}, we only computed the Lagrangian dual value $\hat{z}_{M}(S^1, S^0)$ in \eqref{eq:reduceld} to check whether this value is less than a lower bound returned by the local-search \Cref{algo:localsearch}. That is, for a pair $(S^1, S^0)$, we applied the Frank-Wolfe algorithm with a warm start to compute the restricted dual bound $\hat{z}_{M}(S^1, S^0)$, unless the optimality cuts based on $(S^1, S^0)$ could be directly obtained using \Cref{prop:reduceld}.

We considered different settings
 to generate a pair $(S^1, S^0)$ to construct optimality cuts, including: (a) $S^1=\emptyset$, $S^0$ is singleton; (b) $S^1$ is singleton, $S^0 =\emptyset$; (c) $S^1=\emptyset$, $S^0$ has 2 elements; (d) $S^1$ has 2 elements, $S^0=\emptyset$; and (e) both $S^1$ and $S^0$ are singletons. The subsets are set to be small due to the effectiveness of their corresponding optimality cuts and tightness of the continuous-relaxation bounds. 
Besides, except setting (e), as one of subsets is empty, we did not need auxiliary binary variables to construct the optimality cuts \eqref{eq:valid}. We used both the primal and dual strategies for selecting the subsets $S^1$ and $S^0$, because each one has its own advantage.

The well-known variable-fixing technique is a combination of settings (a) and( b), thus it can be viewed as a special case of our optimality cuts. It was originated for \ref{MESP} with \cite{AFLW_IPCO,AFLW_Using} and has recently been applied to \ref{MESP} in \cite{anstreicher2018maximum,anstreicher2020efficient,chen2022computing}. These existing works do not consider optimizing the restricted Lagrangian dual problem to strengthen the upper bound, nor explore the subsets based on the primal continuous-relaxation solution. \Cref{fig_fix118300} presents the number of fixed variables on two IEEE instances by using our strategy (i.e., the integration of settings (a) and (b)), compared to that of \cite{chen2022computing}. Because both methods manage to fix $s$ variables to one and the remaining $n-s$ variables to zero within one second, when $s \le 4$ in 118-bus instance and $s \le 31$ in 300-bus instances, we do not display these results. 
\begin{figure}[ht]
	\centering
	\subfigure[{Number of variables fixed at one for $n$=118}] {\label{fig_fix118300_a}
				\centering
		\includegraphics[width=7.7cm,height=5.3cm]{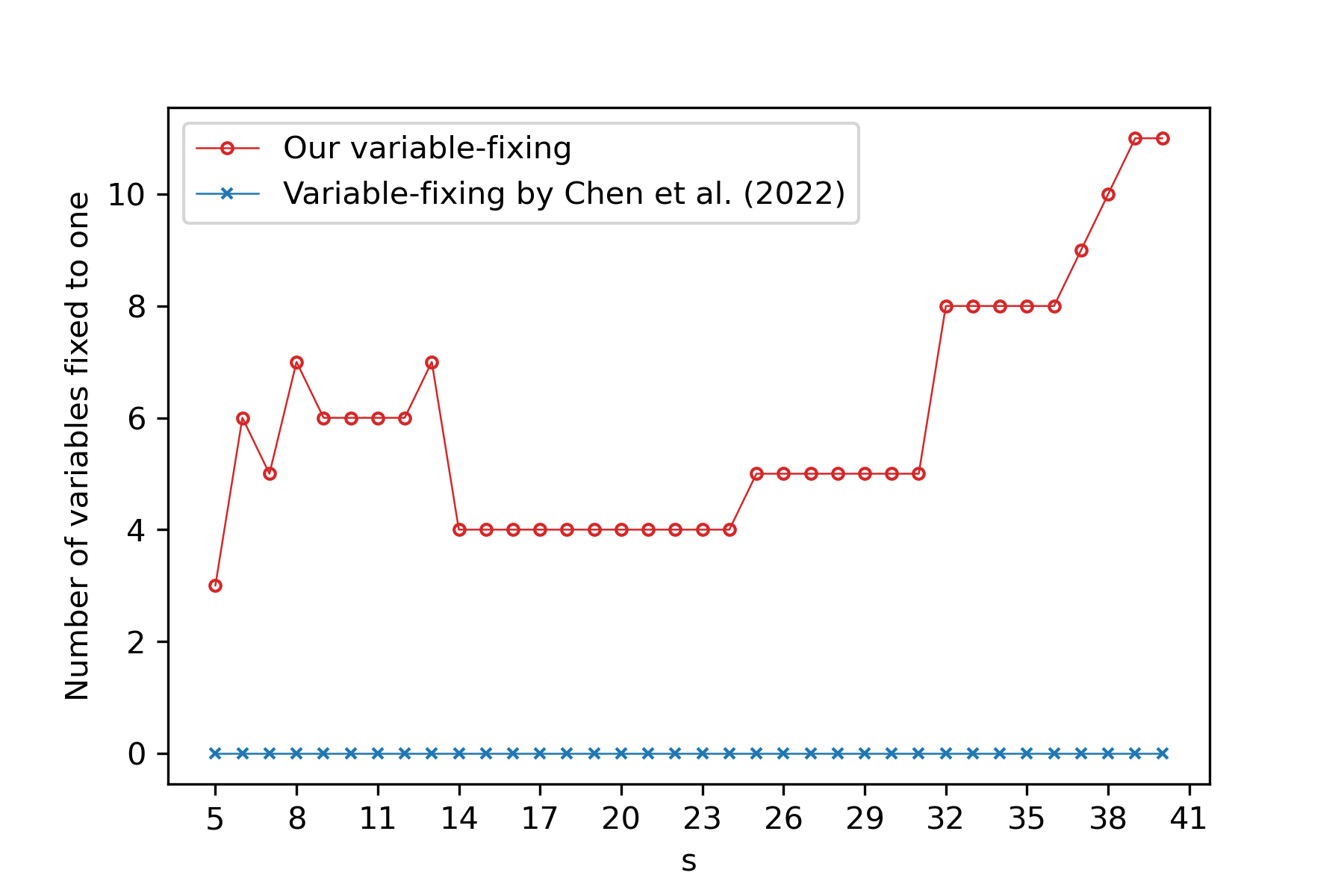}
	}
	\subfigure[{Number of variables fixed at one for $n$=300}] {\label{fig_fix118300_b}
		\includegraphics[width=7.7cm,height=5.3cm]{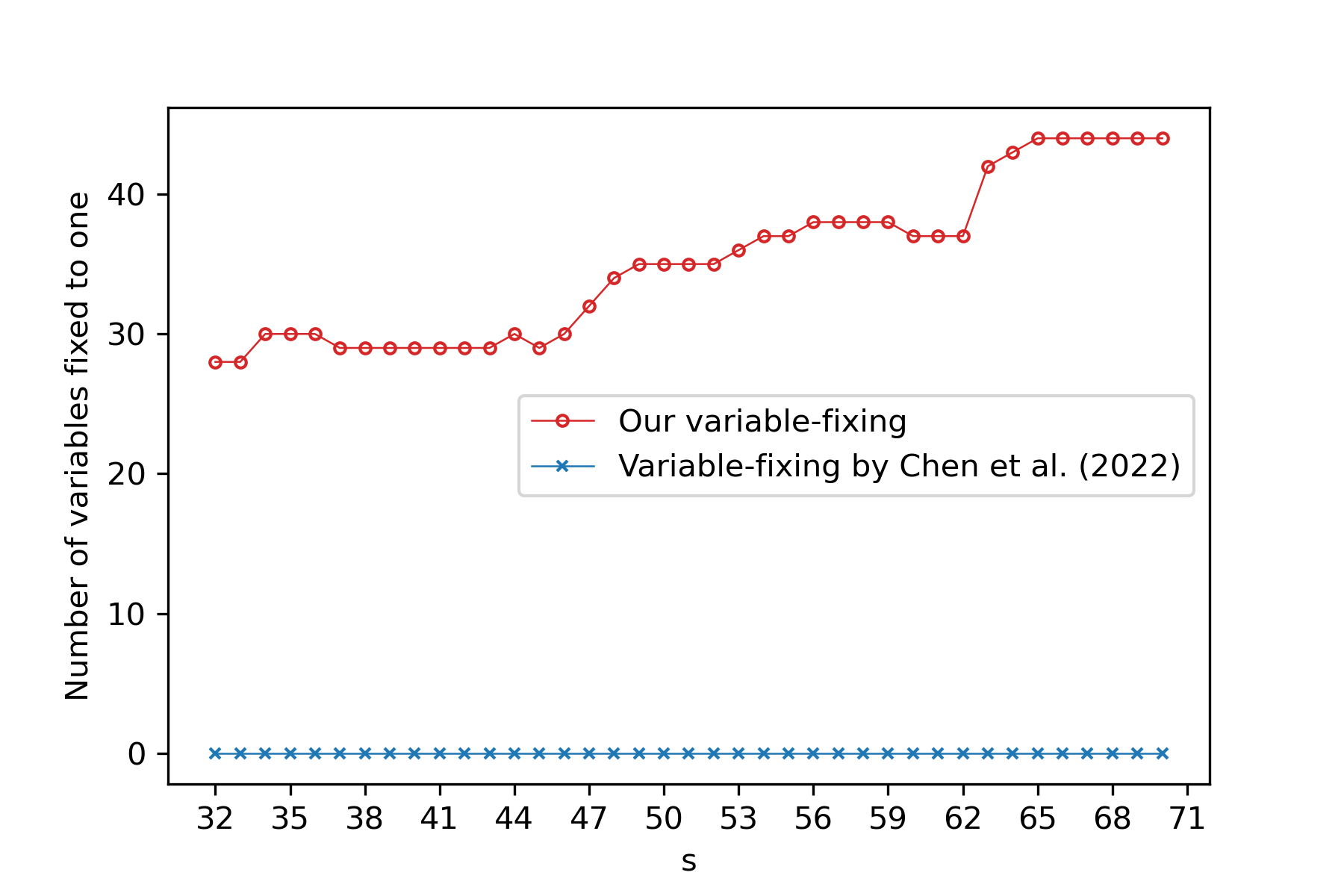}
	}
	\subfigure[{Number of variables fixed at zero for $n$=118}] {\label{fig_fix118300_c}
		\centering
		\includegraphics[width=7.7cm,height=5.3cm]{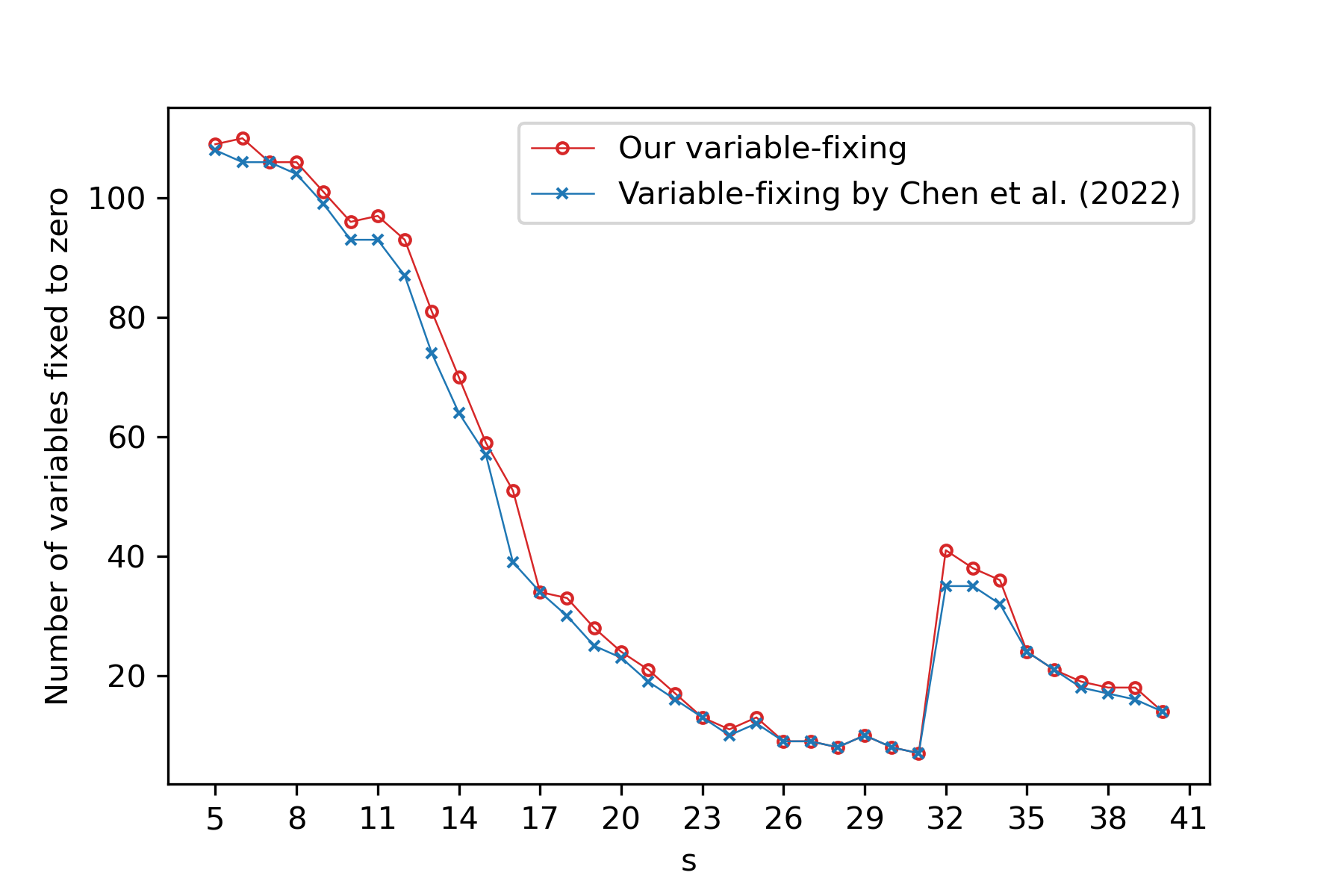}
	}
	\subfigure[{Number of variables fixed at zero for $n$=300}] {\label{fig_fix118300_d}
		\centering
		\includegraphics[width=7.7cm,height=5.3cm]{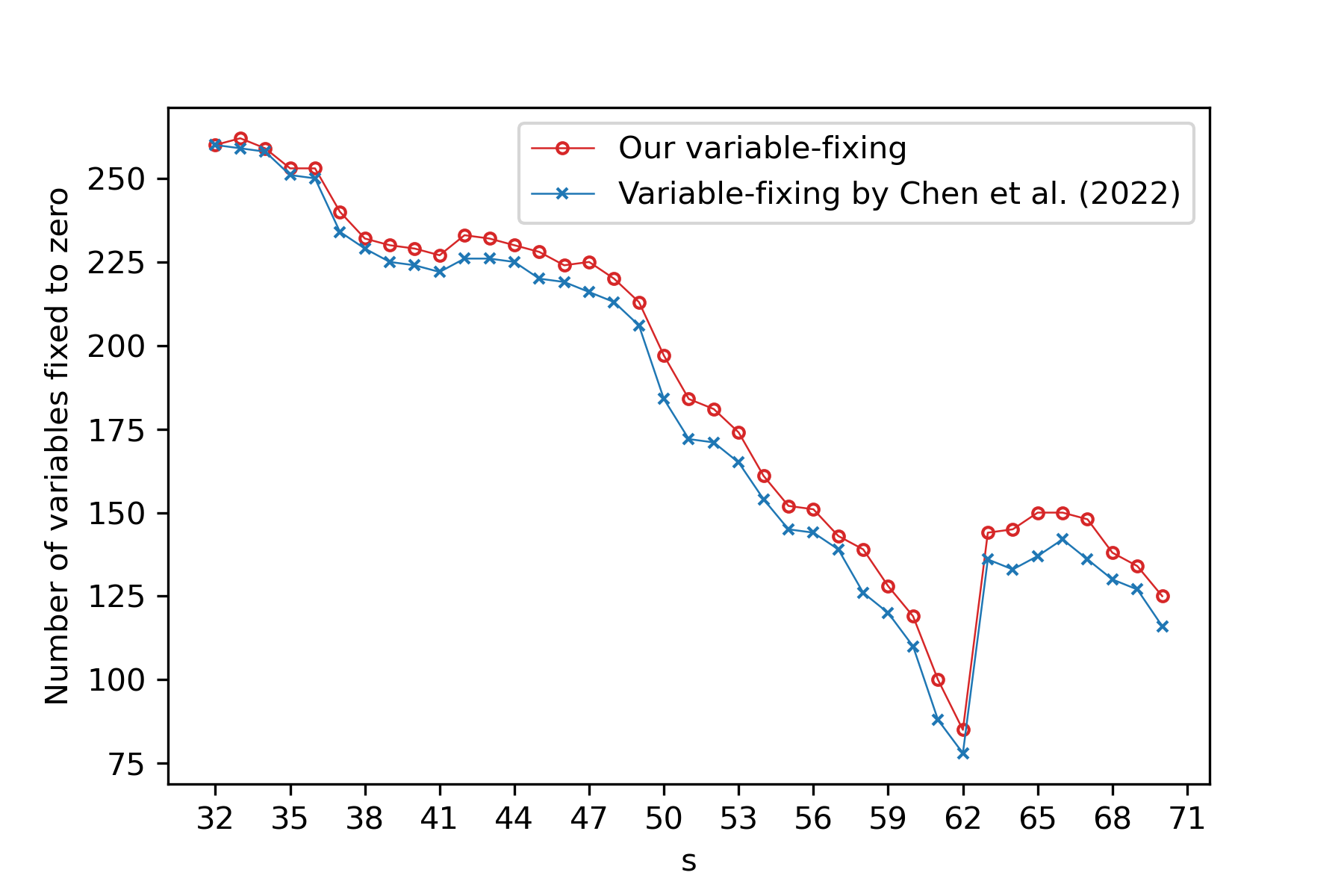}
	}
	\subfigure[{Time for $n$=118}] {\label{fig_fix118300_e}
	\centering
	\includegraphics[width=7.7cm,height=5.3cm]{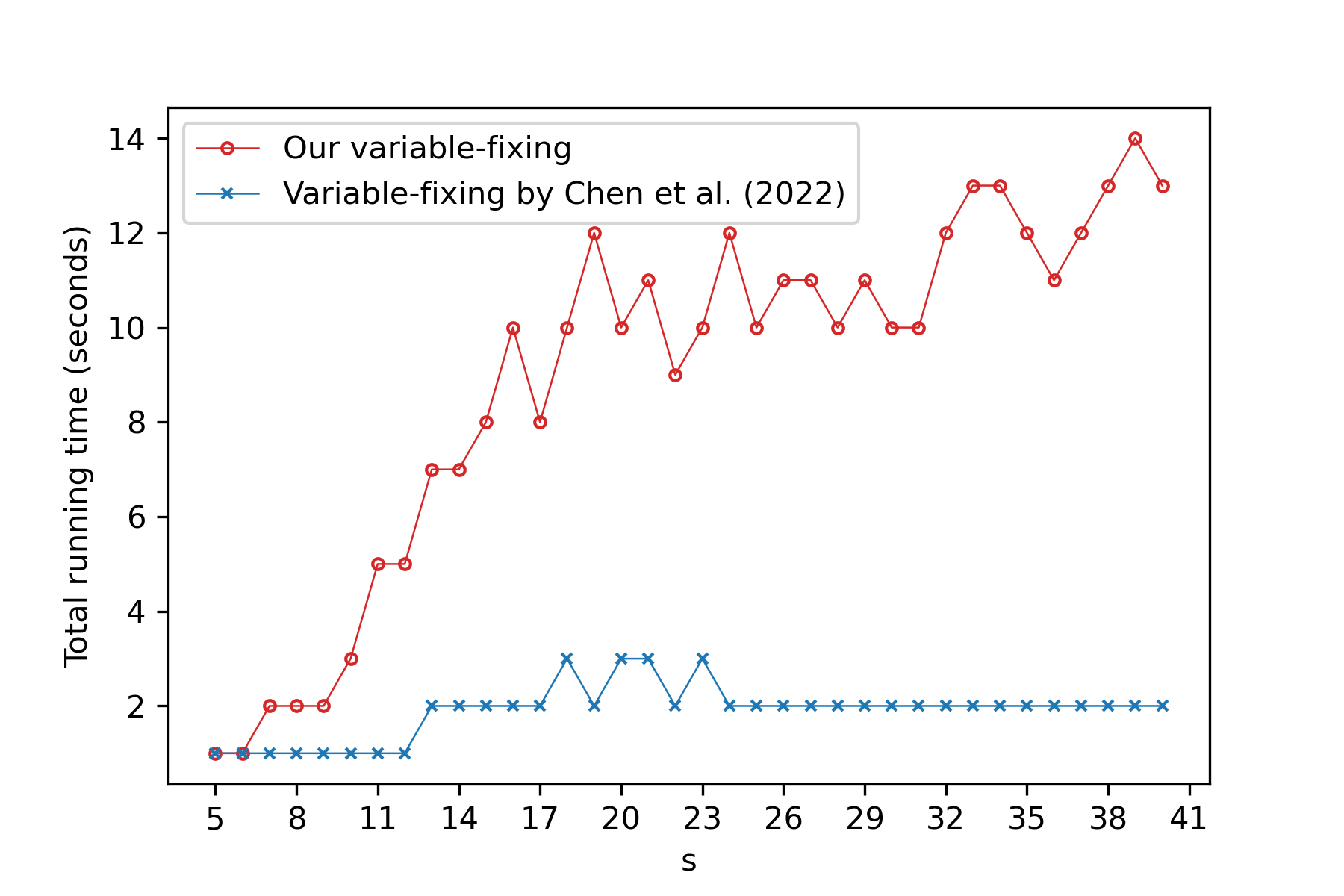}
}
	\subfigure[{Time for $n$=300}] {\label{fig_fix118300_f}
	\centering
	\includegraphics[width=7.7cm,height=5.3cm]{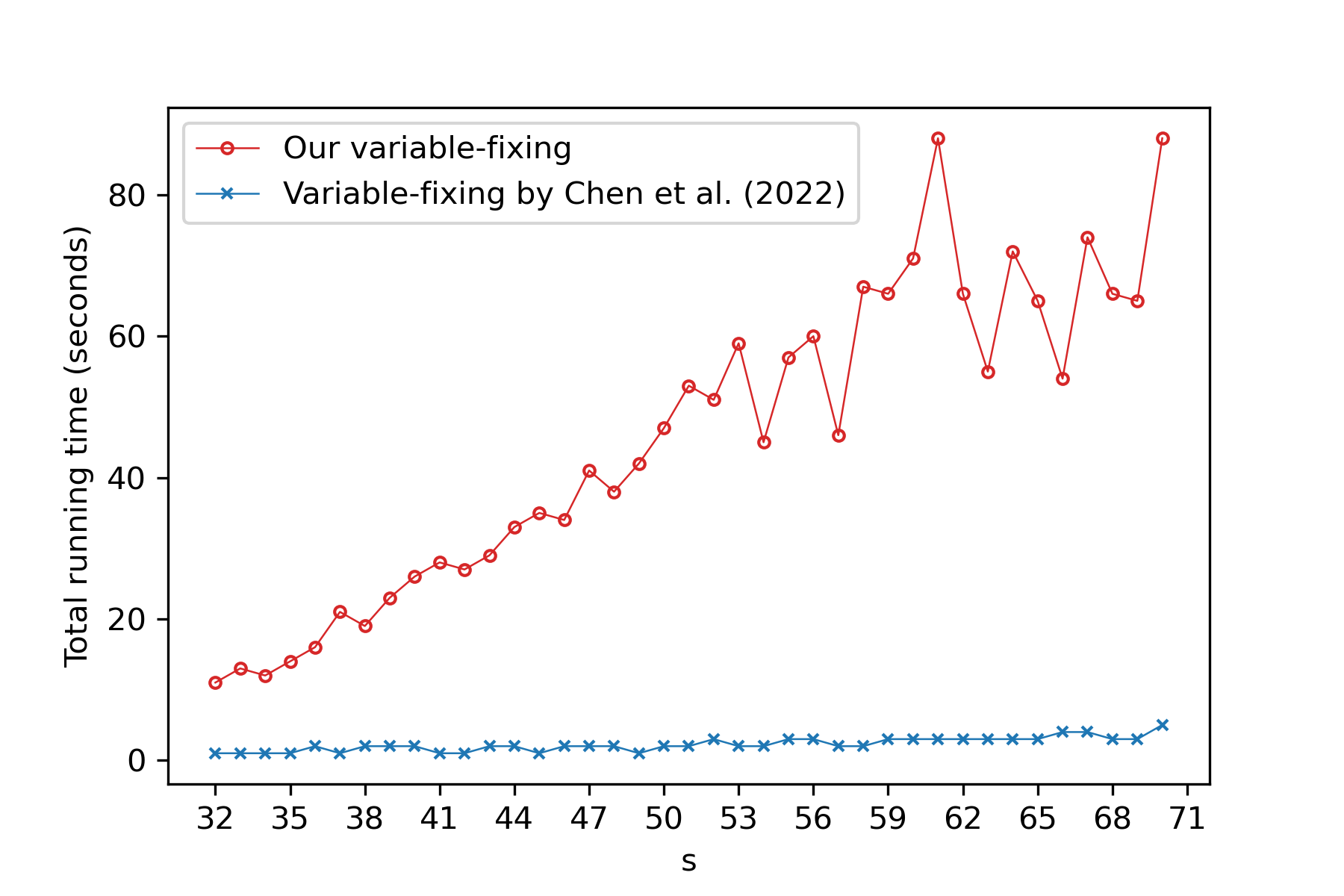}
}
	\caption{Comparison between variable-fixing techniques on IEEE instances}\label{fig_fix118300}
\end{figure}

In \Cref{fig_fix118300}, we see that the number of variables fixed to one tends to steadily increase and the number of variables fixed to zero decreases as $s$ increases. 
From \Cref{fig_fix118300_a} and \Cref{fig_fix118300_b}, we observe that our primal selection strategy successfully finds several optimal variables equal to one, but the variable-fixing by \cite{chen2022computing} fails to fix any variable to one in our test instances. It is worth mentioning that the variables being fixed to one can significantly reduce the problem size and represent the best buses in power systems for installation of new PMU sensors.
In \Cref{fig_fix118300_c} and \Cref{fig_fix118300_d}, our dual strategy slightly outperforms that of \cite{chen2022computing} when fixing variables to zero,
and we are able to fix 13 more variables compared to that of \cite{chen2022computing} for some cases. Finally, we compare the time for both methods in \Cref{fig_fix118300_e} and \Cref{fig_fix118300_f}. We see that the overall performance of our primal and dual strategies is better than that of \cite{chen2022computing}, and our method takes more time but is still negligible compared to what it takes to solve \ref{model} to optimality.

For the other three settings (c), (d), and (e) that involve two binary variables, their number of optimality cuts and overall time are illustrated in \Cref{fig_other118300}. 
We see that when $s$ increases, the number of optimality cuts tends to increase, but as seen in \Cref{fig_fix118300}, the variable-fixing gets worse, implying the fact that optimality cuts can be complementary to each other. Thus, in order to ensure the effectiveness of optimality cuts in our algorithmic framework, we may need to try various types of optimality cuts.

\begin{figure}[hbtp]
	\centering

	\subfigure[{Optimality cuts with two variables for $n$=118}] {
	\includegraphics[width=7.5cm,height=5.3cm]{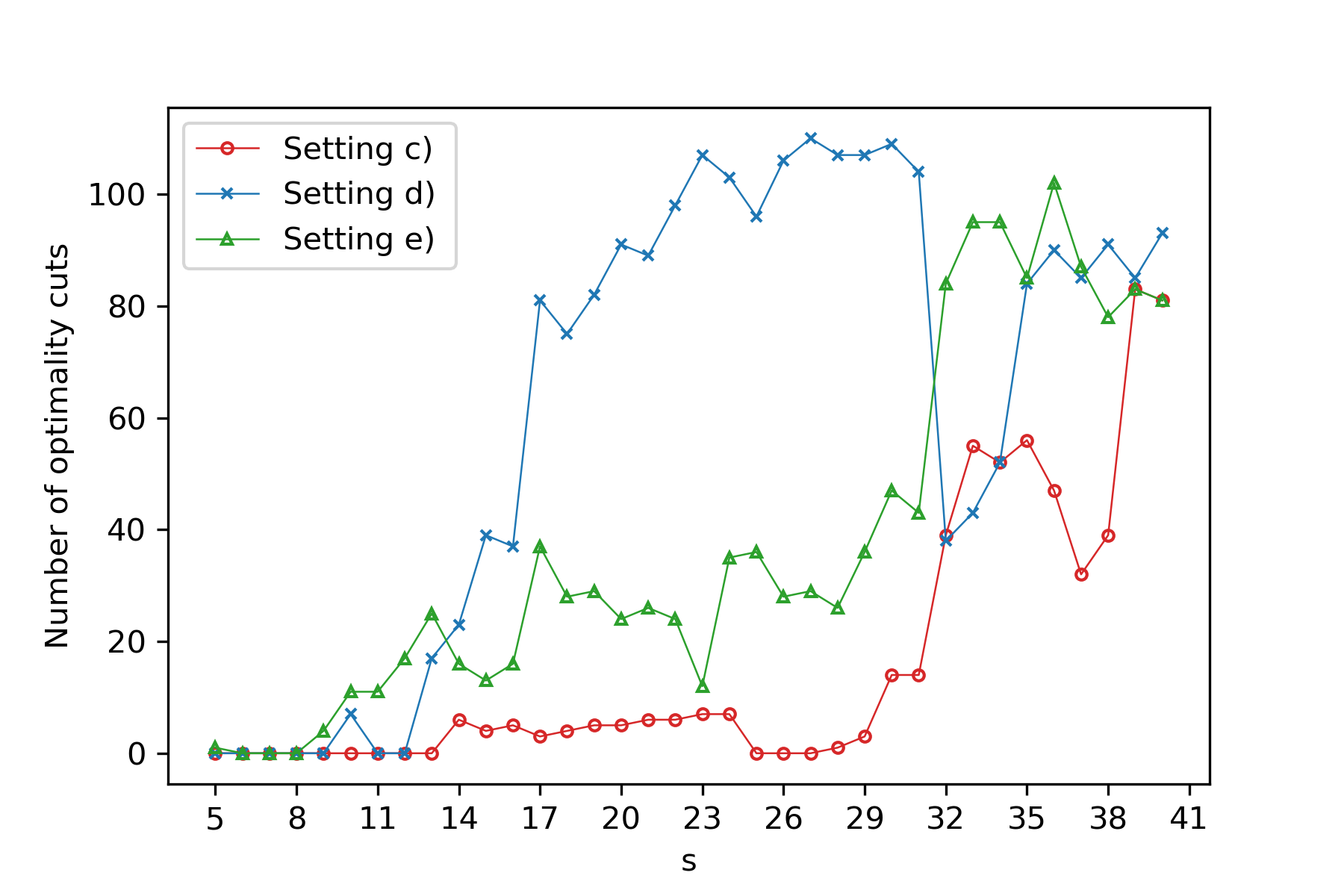}
}
	\subfigure[{Optimality cuts with two variables for $n$=300}] {
	\centering
	\includegraphics[width=7.5cm,height=5.3cm]{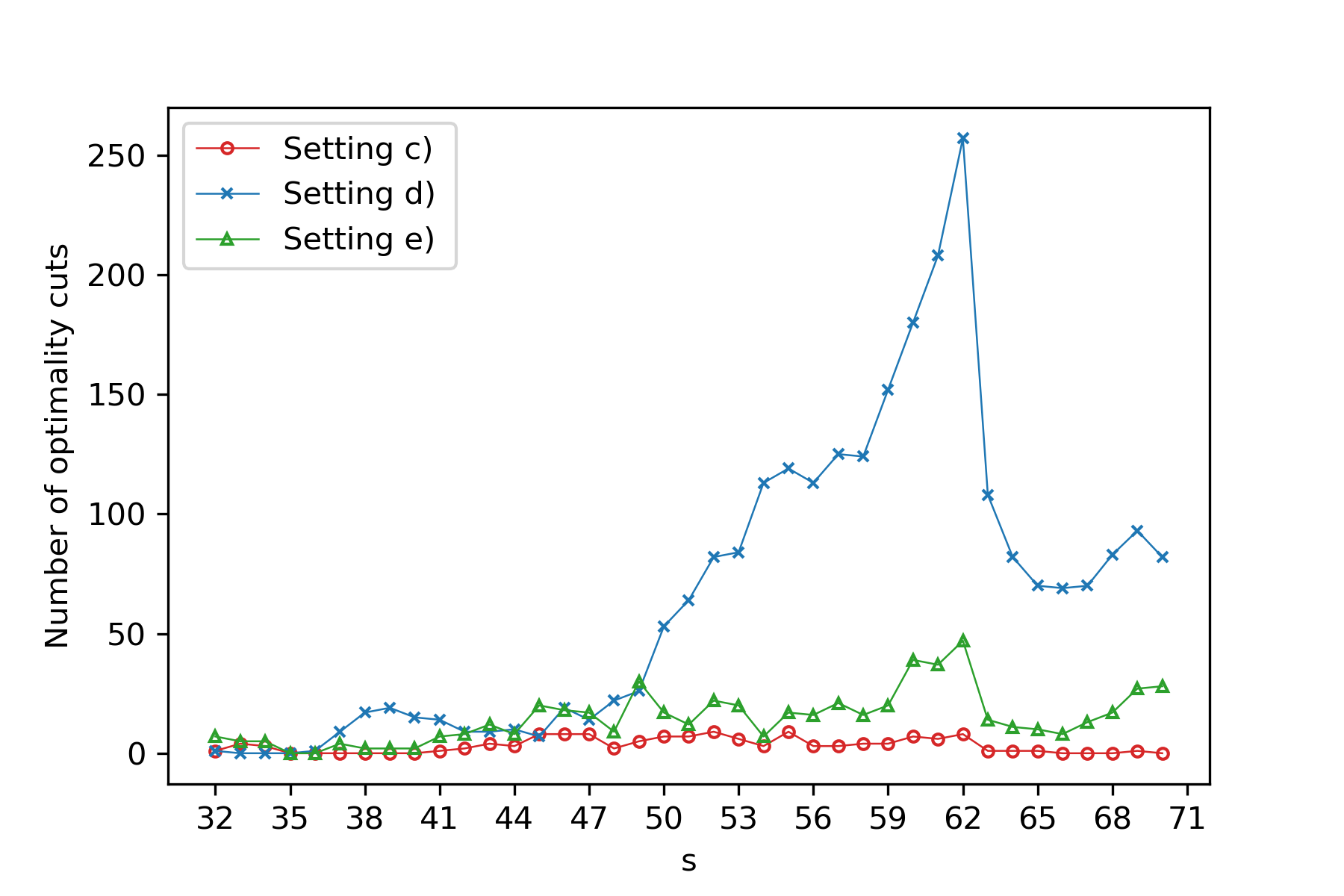}
}

	\subfigure[Time for {$n$=118}] {
	\centering
	\includegraphics[width=7.5cm,height=5.3cm]{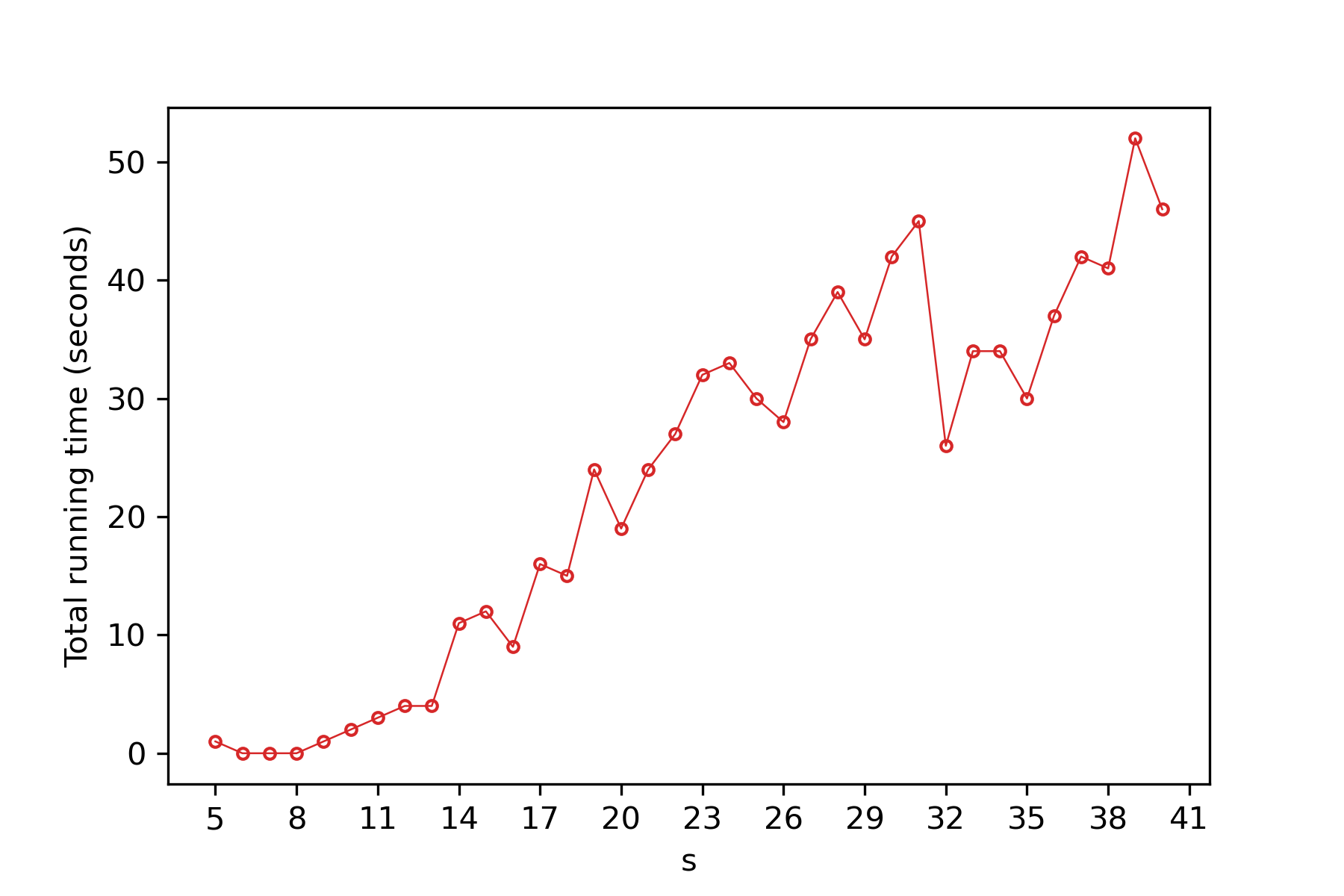}
}
\subfigure[Time for {$n$=300}] {
	\centering
	\includegraphics[width=7.5cm,height=5.3cm]{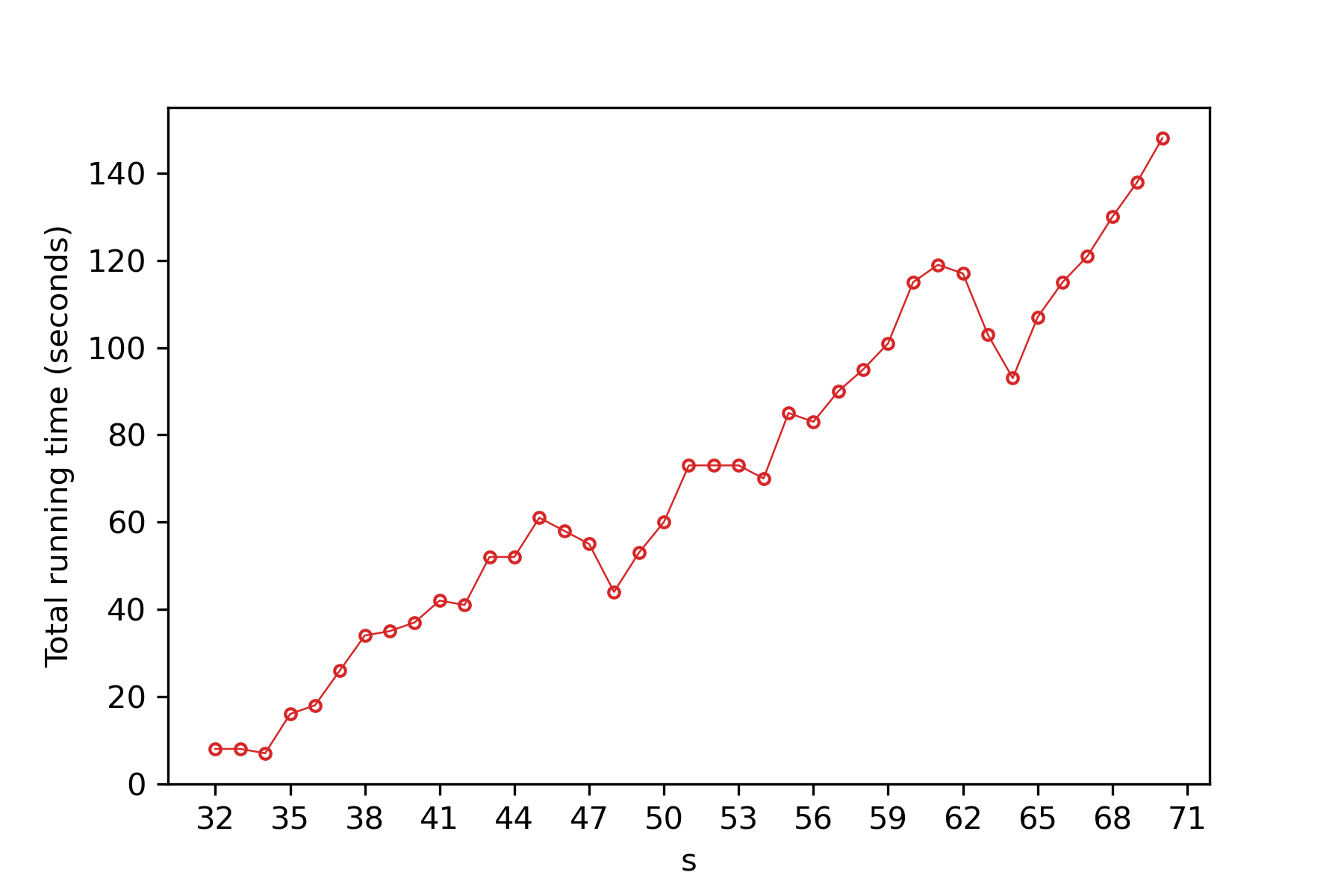}
}
	\caption{Performance of optimality cuts with two binary variables on IEEE instances}\label{fig_other118300}
\end{figure}

\subsection{Exact and approximation algorithms: IEEE 118- and 300-bus instances }
We tested the two IEEE instances and found that with the benefit of submodular and optimality cuts, the LP/NLP B\&B can efficiently solve \ref{model} and approximation algorithms can find high-quality solutions within one minute.

 In \Cref{table:BC300}, we compare the time and the optimality gap of the LP/NLP B\&B algorithm with and without submodular and optimality cuts. For LP/NLP B\&B with submodular and optimality cuts, we also present the time to generate the optimality cuts. The time to generate the submodular cuts is negligible. We used the optimality cuts based on subsets from settings (a) and (b) (i.e., variable-fixing) to reduce the problem size of \ref{model}, 
and we added the other optimality cuts from settings (c), (d), and (e) to the initial MILP relaxation $M^0$ solved in LP/NLP B\&B.
Columns ``$\#$a $-$ $\#$e'' present the number of optimality cuts corresponding to settings (a)$-$(e).
 For each test case, the time limit was set to four hours. 
 
Just using submodular cuts, we can solve five more cases within the time limit, and even for the cases that are not solved to optimality, the MIPgap is dramatically reduced. When additionally, optimality cuts are included, all cases are solved to optimality within the time limit. For the four cases that took a substantial amount of time using only submodular cuts, with optimality cuts we could solve them much faster. We can see that many variables were fixed for all test cases using optimality cuts based on settings (a) and (b).
According to \cite{li2011phasor}, it is difficult to compute the optimal value of \ref{model} for a power system with more than 100 buses. However, our LP/NLP B\&B using the submodular and optimality cuts enables us to effectively solve these cases to optimality. More details on this experiment are in Table \ref{table:suppBC300}, in Appendix \ref{app:supnum}.

\begin{table}[ht] 
	\centering
	\caption{ Impact of submodular and optimality cuts on LP/NLP B\&B on IEEE instances} 
	\begin{threeparttable}
		\setlength{\tabcolsep}{4pt}\renewcommand{\arraystretch}{1.1}
		\begin{tabular}{ c c|r r| r r| r r | r r r r r c}
			\hline 
			\multirow{3}{1em}{$n$}	& \multirow{3}{1em}{$s$} & \multicolumn{2}{c|}{LP/NLP B\&B } & \multicolumn{2}{c|}{ LP/NLP B\&B } & \multicolumn{8}{c}{LP/NLP B\&B } \\
		& & \multicolumn{2}{c|}{ } & \multicolumn{2}{c|}{$+$ submod. cuts}	&\multicolumn{8}{c}{ $+$ submod. and opt. cuts} 
			\\ \cline{3-14} 
			& &\multicolumn{1}{c}{MIPgap\tnote{1}} 	& \multicolumn{1}{c|}{time\tnote{2} } & \multicolumn{1}{c}{MIPgap\tnote{1}} & {time\tnote{2} } &\multicolumn{1}{c}{MIPgap\tnote{1}} & \multicolumn{1}{c|}{time\tnote{2} } &\multicolumn{1}{c}{ $\#$a} &\multicolumn{1}{c}{$\#$b} & \multicolumn{1}{c}{$\#$c} & \multicolumn{1}{c}{$\#$d} & \multicolumn{1}{c}{$\#$e} & \multicolumn{1}{l}{cut time\tnote{3} } \\ 
			\hline
			118 & 5	& 0.00 & 11& 0.00 & 1&0.00& 2 & 3 & 109 & 0 &0 &1 & 2	\\	
			118 & 10& 0.00 & 423& 0.00 & 5&0.00& 5 & 6 & 96 & 0 &7 & 11 & 5		\\	
			118 & 15& 3.18 & - & 0.00 & 1019 &0.00& 34 & 4 & 59& 4 & 39 & 13 & 20	\\	
			118 & 16& 4.27& - &0.00 & 5368&0.00& 58 & 4 & 51& 5 & 37 & 16 & 19	\\	
			118 & 17& 5.19 & - & 0.61 & - &0.00& 344 & 4 & 34& 3 & 81 & 37& 24	\\	
			118 & 18& 4.81 & - & 1.34 & - &0.00& 626 & 4 & 33& 4 & 75 & 28 & 25	\\	
			118 & 19& 8.47 & - &2.16 & - &0.00& 1542& 4 & 28& 5 & 82 & 29 & 36	\\	
			118 & 20& 11.11 & - & 2.20 & - &0.00& 7111& 4 & 24& 5 & 91 & 24 & 29	\\	
			300& 35 & 7.76 & - & 0.00 & 33 &0.00& 31 & 30 & 253 &0 & 0 & 0 & 30 \\
			300& 40 & 83.52 & - &0.00 & 505 &0.00& 69 &29 & 229 &0 & 15 & 2 & 63 \\
			300& 45 &48.06 & - & 0.00 & 1334 &0.00& 109 &29 & 228 & 8 & 7 & 20 & 96 \\
			300& 50 &52.67 & - & 0.88& - &0.00& 135 &35 & 197 &7& 53 & 17 & 107 \\
			300& 51 & 40.41 & - & 0.90 & - &0.00&218 &35 & 184 &7 & 64& 12 & 126 \\
			300& 52 & 140.05 & - & 2.44 & - &0.00& 219 &35 & 181 &9& 82 & 22 & 124 \\
			300& 53 & 186.13 & - & 3.27 & - &0.00& 328 &36 & 174 &6 &84 & 20 & 132 \\
			300& 54 & 168.61 & - & 3.55 & - &0.00& 560 &37 & 161 &3 & 113 & 7 & 115 \\
			300& 55 & 158.67 & - & 4.60 & - &0.00& 1240 &37 & 152 &9 & 119 & 17 & 142 \\
			300& 56 & 143.81& - & 5.29& - &0.00& 5072 &38 & 151 &3 & 113 & 16 & 143 \\
			300& 57 & 129.59 & - & 4.57 & - &0.00& 11462 &38 & 143 &3& 125 & 21 & 136 \\		
			\hline
		\end{tabular}%
		\begin{tablenotes}
			\item[1] MIPgap = upper bound $-$ best feasible-solution value (both obtained by LP/NLP B\&B) 
			\item[2] total time in seconds; ``-'': instance not solved within four hours 
			\item[3] time to generate optimality cuts in seconds
		\end{tablenotes}
	\end{threeparttable}
	\label{table:BC300}
\end{table}

Using the optimal values of \ref{model} in \Cref{table:BC300}, we evaluate the bounds given by the continuous relaxation of \ref{model_mesp} and the performance of the approximation algorithms on the same testing cases. The computational results are displayed in \Cref{table:approx}. The ``gap'' for \ref{model_mesp} is the difference between the continuous relaxation value $\hat{z}_M$ and the optimal value $z^*$. The ``gap'' for the approximation algorithms is the difference between $z^*$ and the solution returned by them. For comparison purposes, we also tested the greedy algorithm for solving the PMU placement problem in power systems studied in \cite{li2011phasor, li2012information}. 
We feed the continuous-relaxation solution of \ref{model_mesp} to \Cref{algo:sampling}, and thus the time for \Cref{algo:sampling} includes that of solving the continuous relaxation.
We see that the upper bound given by the solution of the continuous-relaxation of \ref{model_mesp} is always close to the optimal value. We also see that \Cref{algo:localsearch} consistently gives better times and gaps than \Cref{algo:sampling}. 
Although \Cref{algo:localsearch} and the greedy algorithm have the same gap for each instance, in our implementations \Cref{algo:localsearch} is much faster due to the rank-one updating technique. 

\begin{table}[ht] 
	\centering
	\caption{Relaxation and approximation algorithms on IEEE instances} 
	\begin{threeparttable}
 		\setlength{\tabcolsep}{3.5pt}\renewcommand{\arraystretch}{1.1}
		\begin{tabular}{c c| r r| r r| r r| r r }
			\hline 
		 \multirow{3}{1em}{$n$}	& \multirow{3}{1em}{$s$} 	& \multicolumn{2}{c|}{\ref{model_mesp}} &
			\multicolumn{2}{c|}{local-search} & \multicolumn{2}{c|}{ sampling} & \multicolumn{2}{c}{greedy}\\
		&	& \multicolumn{2}{c|}{relaxation} &
			\multicolumn{2}{c|}{\Cref{algo:localsearch}} & \multicolumn{2}{c|}{\Cref{algo:sampling} } & \multicolumn{2}{c}{algorithm}
			\\ \cline{3-10} 
		& & \multicolumn{1}{c}{gap} & {time\tnote{1} } & \multicolumn{1}{c}{gap} & {time\tnote{1} } & \multicolumn{1}{c}{gap} & {time\tnote{1} } & \multicolumn{1}{c}{gap} & {time\tnote{1} }\\ 
			\hline
		118& 5&0.10& $< 1$ &0.00& $<1$ & 0.00 & 5 &0.00& $<1$ \\
		118& 10 & 0.16 &1 & 0.00& $<1$ & 0.00& 5 & 0.00& $<1$\\
		118& 15 & 0.42& 1 &0.00& $<1$ & 0.08& 6 &0.00& $<1$ \\	
		118& 16 & 0.48& 1 &0.00& $<1$ & 0.32 & 6 &0.00& $<1$ \\		
		118& 17 & 0.55 &1 &0.00& $<1$ & 0.17 & 6 &0.00& $<1$ \\	
		118& 18 & 0.57 & 2 &0.00& $<1$ & 0.23 & 7 &0.00& $<1$ \\	
		118& 19 & 0.60 &1 &0.00& $<1$ & 0.79 & 6&0.00& $<1$ \\
		118& 20 & 0.64 & 2 &0.00& $<1$ & 0.87 & 6 &0.00& $<1$\\			
300& 35 & 0.11 & 6 &0.00 & $<1$ & 0.00 & 34 &0.00 & 14 \\
300& 40 &0.31 & 11 &0.00 & $<1$ & 2.04 & 46 &0.00 & 13 \\
300& 45 &0.26 & 5 &0.00 & $<1$ & 4.40& 41 &0.00 & 12\\
300& 50 &0.37 & 5 &0.00 & $<1$ & 3.98 & 37 &0.00 & 11 \\
300& 51 &0.41 & 7 &0.00 & $<1$ & 2.25 & 39 &0.00 & 12 \\
300& 52 &0.40 & 6 &0.00 & $<1$ & 4.18 & 45 &0.00 & 11 \\
300& 53 &0.42 & 8 &0.00 & $<1$ & 4.36 & 45 &0.00 & 12 \\
300& 54 &0.46 & 10 &0.00 & $<1$ & 3.56 & 38&0.00 & 12\\
300& 55 &0.47 & 10 &0.00 & $<1$ & 2.29 & 46 &0.00 & 12\\
300& 56 &0.47 & 14 &0.00 & $<1$ & 7.41 & 42 &0.00 & 13 \\
300& 57 &0.46 & 11 &0.05 & $<1$ & 5.07 & 46 &0.05 & 17 \\	
			\hline
		\end{tabular}%
		\begin{tablenotes}
			\item[1] time in seconds; ``$<1$'': time less than one second
		\end{tablenotes}
	\end{threeparttable}
	\label{table:approx}
\end{table}

\subsection{Scalability and near-optimality of approximation algorithms: IEEE 118-, 300-, and Polish 2383-bus instances}
To better understand the overall performance of our approximation algorithms, we tested cases with a full range of selected data points (i.e., $s$) on the two IEEE instances and on a larger instance with 2383 buses, one of the largest power systems in the literature (see \cite{zimmerman1997matpower}). Note that LP/NLP B\&B is unable to solve many of the large-$s$ cases to optimality within a four-hour time limit. Thus, here we use the continuous-relaxation value $\hat{z}_M$ of \ref{model_mesp} to evaluate the quality of the feasible solutions obtained using the approximation algorithms. The value for ``gap'' in \Cref{table:polish} and \Cref{fig_large} is equal to the difference between $\hat{z}_M$ and the output values produced by the approximation algorithm. 

We compare our proposed approximation algorithms with the greedy algorithm in \Cref{fig_large}, for which we test 23 cases with number of installed PMUs from $5, 10, \ldots, 115$ and 29 cases with number of installed PMUs from $10, 20, \ldots, 290 $
 for the two IEEE instances with 118 and 300 buses, respectively. We see that \Cref{algo:localsearch} clearly outperforms the other two, considering the gaps and times.
 
\begin{figure}[hbtp]
	\centering
	
	\subfigure[Gaps for $n$=118] {
		\includegraphics[width=7.5cm,height=5.3cm]{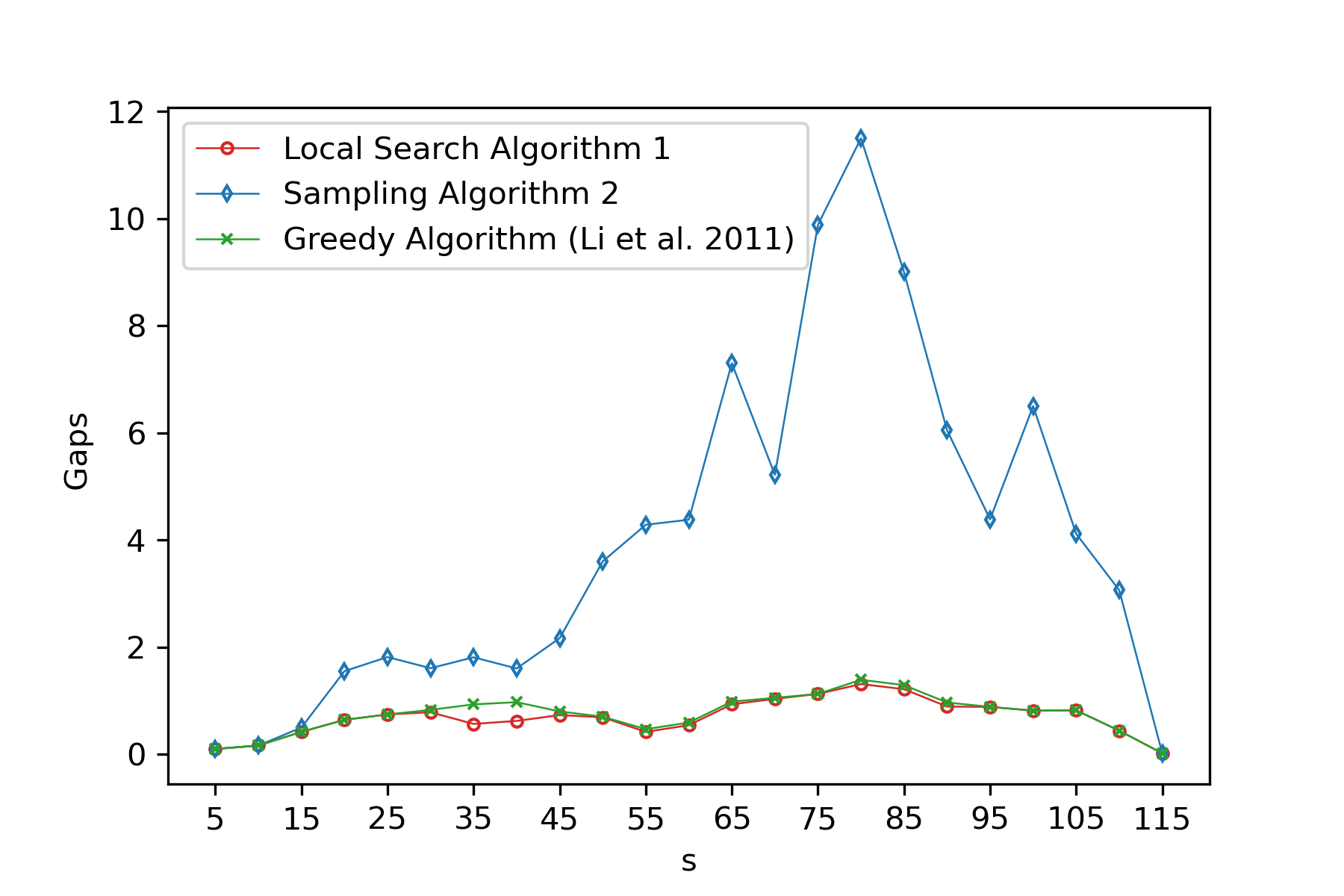}
	}
	\subfigure[Gaps for $n$=300] {
		\centering
		\includegraphics[width=7.5cm,height=5.3cm]{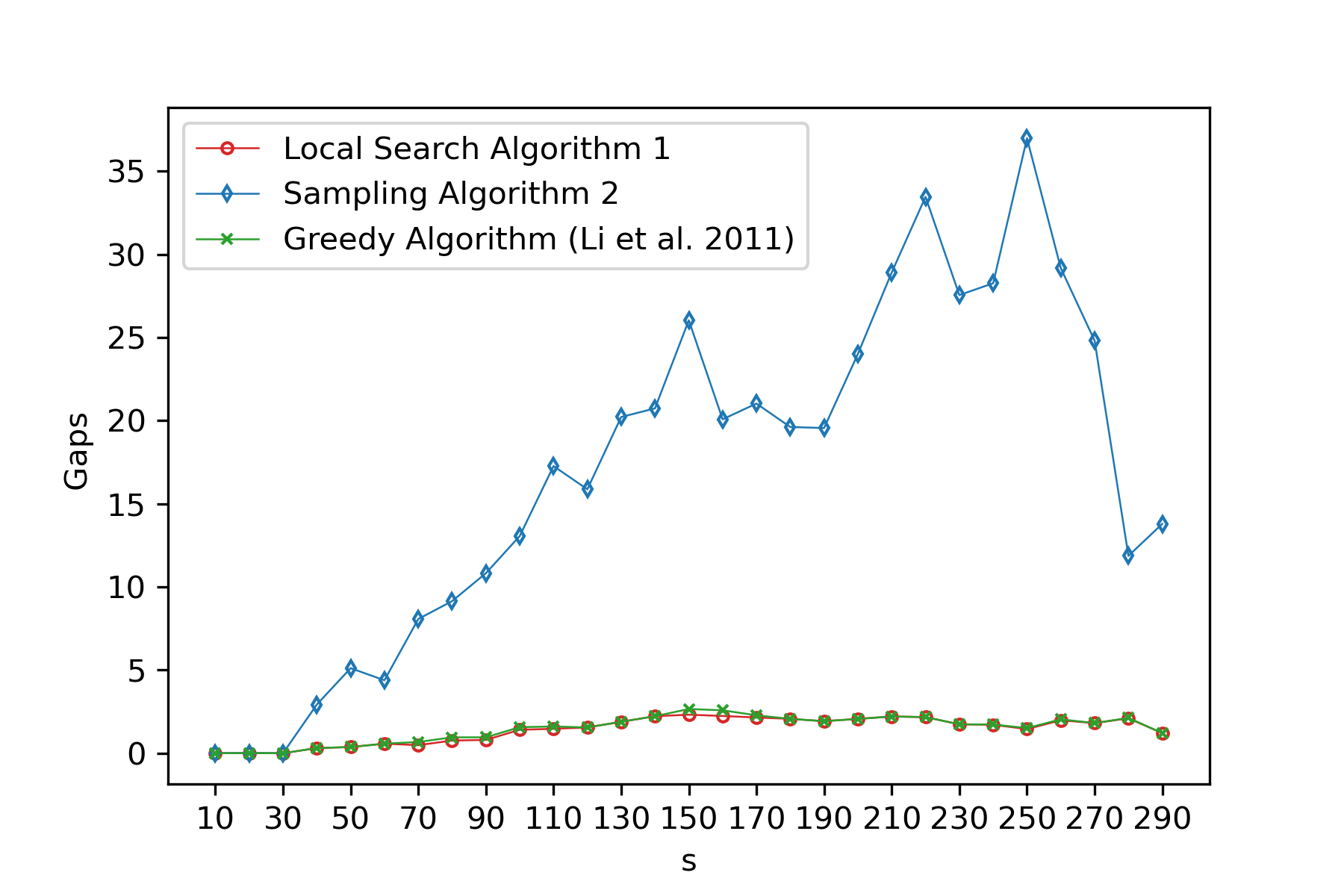}
	}
	
	\subfigure[Time for $n$=118] {
		\centering
		\includegraphics[width=7.5cm,height=5.3cm]{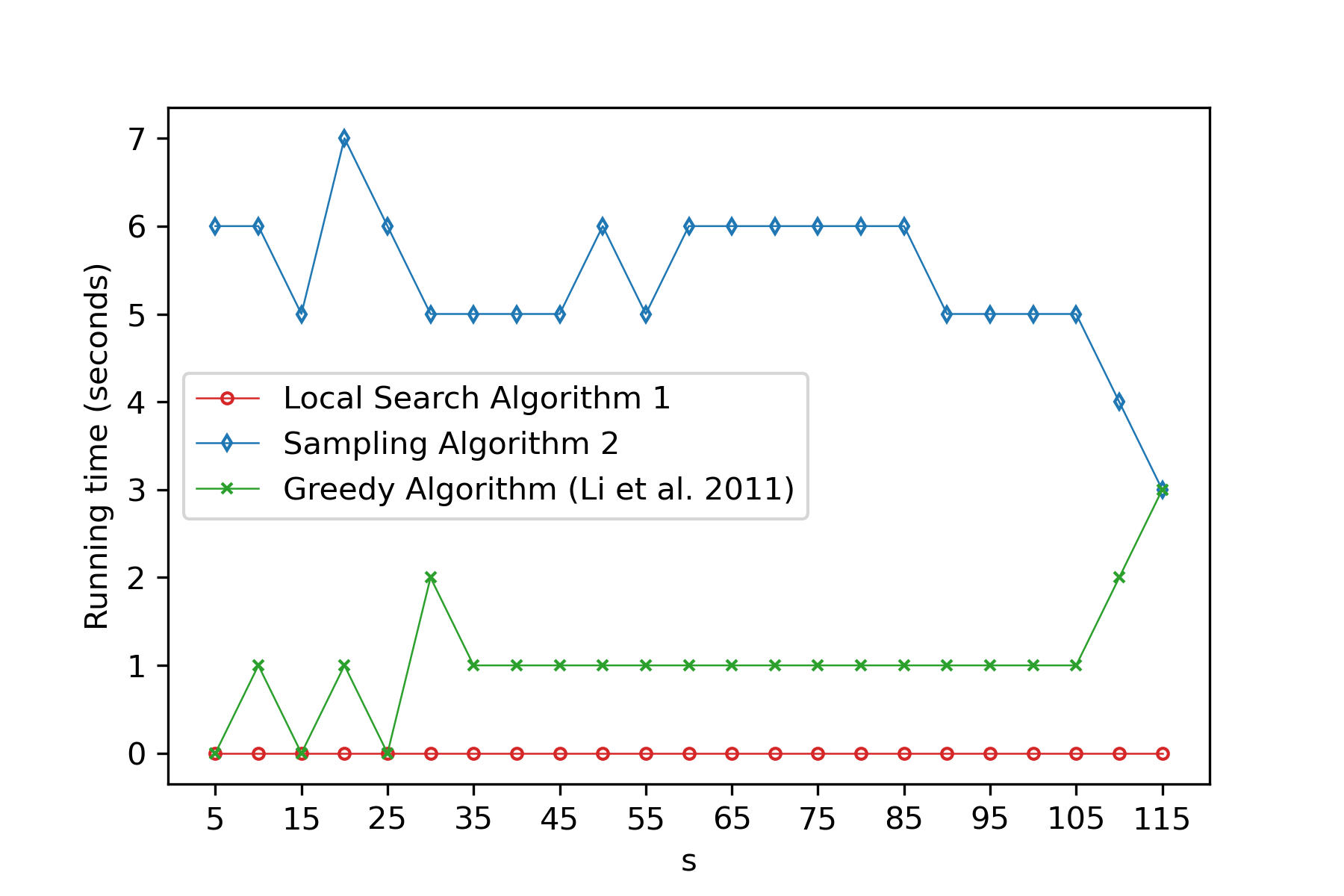}
	}
	\subfigure[Time for $n$=300] {
		\centering
		\includegraphics[width=7.5cm,height=5.3cm]{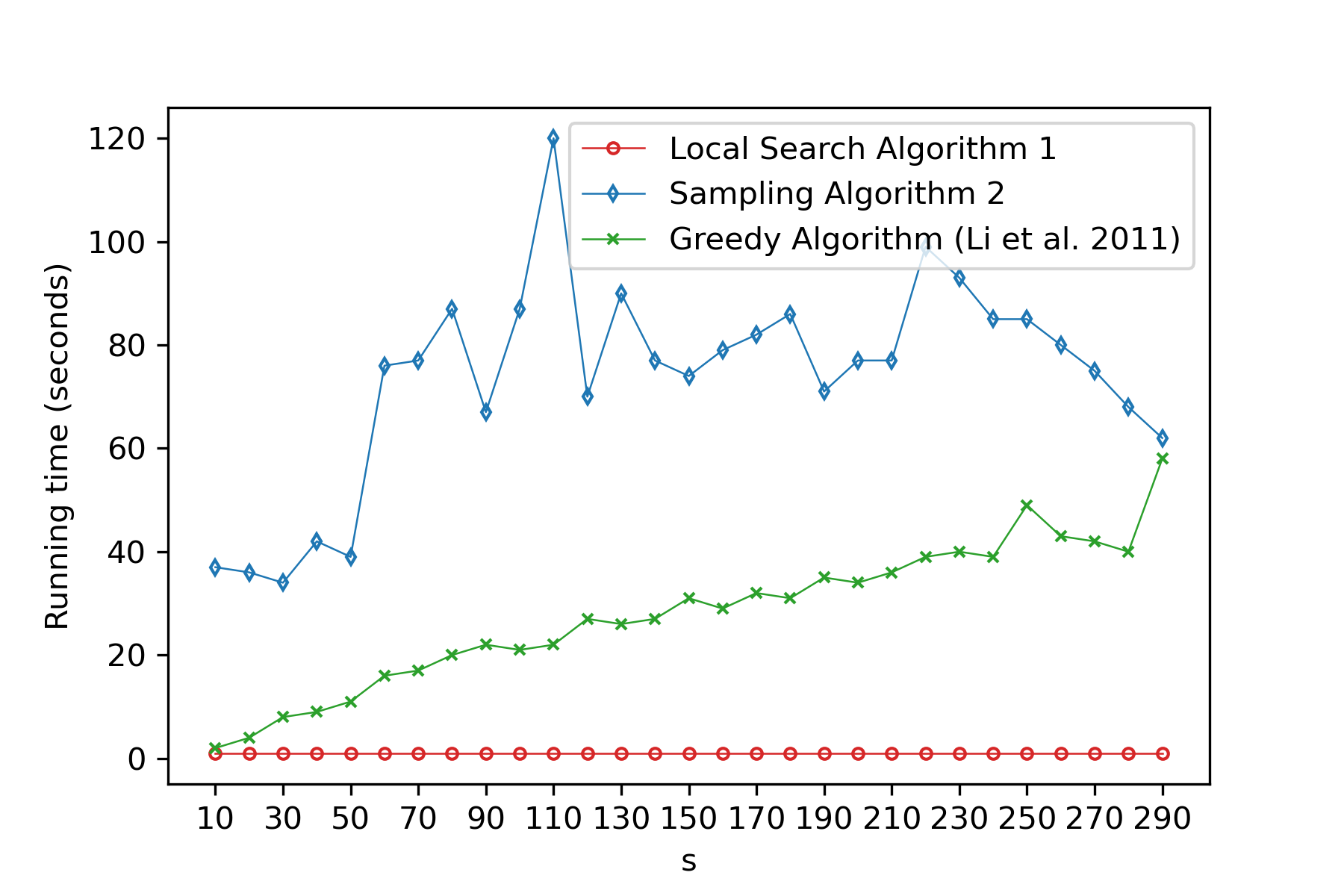}
	}
	
\caption{Comparison between approximation algorithms on IEEE instances}\label{fig_large}
\end{figure}

In Table \ref{table:polish}, we present results where we tested a large-scale instance with 2383 buses in the Polish power system (see \cite{zimmerman1997matpower}). The greedy algorithm is omitted in this table because it could not finish on these cases within four hours. On the other hand, our proposed approximation algorithms scale well. It is evident that 
\Cref{algo:localsearch} performs very well in time and solution quality, dominating the performance of \Cref{algo:sampling} in both respects. Thus, we recommend using \Cref{algo:localsearch} (with an efficient implementation) to solve practical PMU placement problems. We note that the gaps for \Cref{algo:localsearch} could be quite small if we could compare to the optimal value. Another observation, is that the small gaps for \Cref{algo:localsearch} establish the quality of the \ref{model_mesp} relaxation on these cases. 

\begin{table}[ht] 
	\centering
	\caption{Relaxation and approximation algorithms on the Polish instance} 
	\begin{threeparttable}
		\setlength{\tabcolsep}{3pt}\renewcommand{\arraystretch}{1.2}
		\begin{tabular}{c c|r r| r r| r r }
			\hline 
		 \multirow{3}{1em}{$n$}	& \multirow{3}{1em}{$s$} & \multicolumn{2}{c|}{\ref{model_mesp}} & \multicolumn{2}{c|}{local-search} &
			\multicolumn{2}{c}{sampling}\\
		&	& \multicolumn{2}{c|}{relaxation} & \multicolumn{2}{c|}{\Cref{algo:localsearch}} &
			\multicolumn{2}{c}{\Cref{algo:sampling}}
			\\ \cline{3-8} 
		&	& \multicolumn{1}{c}{$\hat{z}_M$} & {time\tnote{1} } &
			{gap} & {time\tnote{1} } & {gap} & {time\tnote{1} } \\ 
			\hline
		2383&	75 & 198.37 & 421& 0.98 & 9& 14.24 & 424 \\
				2383&	100 &614.00 & 545 & 1.24 & 17& 33.82 & 552 \\
				2383&	125 & 1024.55 & 608 & 1.43 & 17& 41.81 & 675 \\
				2383&	150 & 1430.34& 666 & 1.93 & 17&62.18& 740 \\
				2383&175 & 1831.40 & 807 & 2.58 & 25& 78.23 & 853 \\
				2383&200 & 2227.53 & 794 & 3.25 & 28& 61.03 & 822 \\
				2383&225 & 2620.43 & 1064 & 4.08 & 31& 93.11 & 1101 \\
				2383&250 & 3010.08 & 1301& 4.95 & 37& 105.06& 1368 \\
				2383&275 & 3396.66 & 1333 & 5.15 & 46&122.23 & 1426\\
				2383&300 & 3780.59 & 1244 & 5.41 & 47& 135.21& 1448 \\
				2383&325 & 4162.01 & 1293 & 6.11 & 45& 130.45& 1514 \\
				2383&350 & 4540.67 &1392 & 6.92 & 49& 149.43& 1661 \\
				2383&375 & 4916.61 & 1510 & 8.02 & 51& 123.27& 1768 \\
			\hline
		\end{tabular}%
		\begin{tablenotes}
			\item[1] time in seconds
		\end{tablenotes} 
	\end{threeparttable}
	\label{table:polish}
\end{table}

\FloatBarrier

\section{Conclusion}\label{sec:conclusion}
We studied the D-optimal data fusion problem, which can be of vital importance in many fields, such as monitoring, operation, planning, control, and decision making of various environmental, structural, agricultural, food processing, and manufacturing systems. The developed exact and approximation algorithms come with theoretical performance guarantees. Our numerical study confirms the efficacy of the proposed algorithms. We expect the proposed methods can be applicable to many machine learning problems under a cardinality constraint such as sparse PCA, sparse regression, sparse matrix completion, and so on.

\vbox{
\ACKNOWLEDGMENT{%
Y. Li and W. Xie were supported in part by NSF grants 2046426 and 2153607.
M. Fampa was supported in part by CNPq grants 305444/2019-0 and 434683/2018-3.
J. Lee was supported in part by AFOSR grants FA9550-19-1-0175
and FA9550-22-1-0172. 
}
}
\FloatBarrier
\bibliographystyle{alpha}
\bibliography{reference}
\titleformat{\section}{\large\bfseries}{\appendixname~\thesection 
	.}{0.5em}{}
\begin{appendices}
	\section{Proofs}\label{proofs}


	\subsection{Proof of \Cref{them:local}} \label{proof:them:local}
	
	\themlocal*
	\begin{proof}
		We will prove the two approximation bounds (i) $n \log(1+(\bar{s}/n)\sigma_{\max}^2/(1+\sigma_{\max}) )$ and (ii) $\bar{s}\log(\bar{s})$, using \ref{model_dopt}, \ref{model_mesp}, and their complements, respectively.
		
		\begin{enumerate}[(i)]
	\item $d \log(1+(\bar{s}/d)\sigma_{\max}^2/(1+\sigma_{\max}) )$-approximation bound.
	
Using the Lagrangian dual problem \eqref{eq:dopt_ld} of R-DDF $\eqref{model_dopt}$,	let us first show a non-symmetric approximation bound, $d \log(1+(\bar{s}/d)\sigma_{\max}^2/(1+\sigma_{\max}) )$.

Given the output $\hat{S}$ of the local-search \Cref{algo:localsearch}, let $\bm{X}= \sum_{\ell \in \hat{S}} \bm{b}_{\ell}\bm{b}_{\ell}^{\top}$ and $\bm{\Lambda} = (\bm{I}_d+\bm{X})^{-1}$. Because $\hat{S}$ is a locally-optimal solution of \ref{model}, for any $i \in \hat{S}$ and $j\in [n]\backslash \hat{S}$, following the inequalities in \eqref{eq:dopt_local}, the local-optimality conditions can be written as
\begin{equation} \label{eq:lsopt}
\begin{aligned}
	\bm{b}_j^{\top} \bm{\Lambda} \bm{b}_j - \bm{b}_i^{\top} \bm{\Lambda} \bm{b}_i \bm{b}_j^{\top} \bm{\Lambda} \bm{b}_j + \bm{b}_i^{\top} \bm{\Lambda} \bm{b}_j \bm{b}_j^{\top} \bm{\Lambda} \bm{b}_i \le \bm{b}_i^{\top} \bm{\Lambda} \bm{b}_i. 
\end{aligned}
\end{equation}

Next, we will explore the inequality \eqref{eq:lsopt} to construct a feasible solution to the Lagrangian dual problem \eqref{eq:dopt_ld} of the R-DDF. 
For any $i \in \hat{S}$ and $j\in [n]\backslash \hat{S}$, by dropping the nonnegative term $\bm{b}_i^{\top} \bm{\Lambda} \bm{b}_j \bm{b}_j^{\top} \bm{\Lambda} \bm{b}_i$ in the left-hand side of inequality \eqref{eq:dopt_local}, we obtain
\begin{align} \label{eq:upper1}
\bm{b}_j^{\top} \bm{\Lambda} \bm{b}_j \le \frac{1}{1- \bm{b}_i^{\top} \bm{\Lambda} \bm{b}_i} \bm{b}_i^{\top} \bm{\Lambda} \bm{b}_i.
\end{align}

On the other hand, for each $i \in \hat{S}$, we have
\begin{align} \label{eq:upper2}
\bm{b}_i^{\top} \bm{\Lambda} \bm{b}_i \le \bm{b}_i^{\top} (\bm{I}_d+\bm{b}_i\bm{b}_i^{\top})^{-1} \bm{b}_i = \frac{\bm{b}_i^{\top} \bm{b}_i}{1+\bm{b}_i^{\top}\bm{b}_i} = 
\frac{ \bm{a}_i^{\top} \bm{C}^{-1}\bm{a}_i }{1+ \bm{a}_i^{\top} \bm{C}^{-1}\bm{a}_i} \le \frac{\sigma_{\max} }{1+ \sigma_{\max}},
\end{align}
where the last inequality is due to $\sigma_{\max} := \max_{i \in [n]} \bm{a}_i^{\top} \bm{C}^{-1}\bm{a}_i$ and the fact that the function $h(t)=t/(1+t)$ is non-decreasing in $t\in \Re_+$.

Combining the results \eqref{eq:upper1} and \eqref{eq:upper2}, it follows that
\begin{align}\label{ineq1}
\bm{b}_j^{\top} \bm{\Lambda} \bm{b}_j \le \frac{1}{1- \bm{b}_i^{\top} \bm{\Lambda} \bm{b}_i} \bm{b}_i^{\top} \bm{\Lambda} \bm{b}_i \le (1+\sigma_{\max}) \bm{b}_i^{\top}\bm{\Lambda} \bm{b}_i, \forall i\in \hat{S}, \forall j\in [n]\setminus \hat{S},
\end{align}
where the second inequality is due to non-decreasing of $h(t)=1/(1-t)$ over $t\in [0,1)$.

Let us denote $\beta_{\min} := \min_{i \in \hat{S}} \bm{b}_i^{\top}\bm{\Lambda} \bm{b}_i$. Then, according to the inequality \eqref{ineq1}, a feasible solution to the Lagrangian dual problem \eqref{eq:dopt_ld} can be constructed by
\begin{align*}
\hat{\bm \Lambda} = t \bm{\Lambda}, \hat{\nu} = t (1+\sigma_{\max}) \beta_{\min}, 	\hat{\mu}_i = t (1+\sigma_{\max}) \bm{b}_i^{\top}\bm{\Lambda} \bm{b}_i -\hat{\nu}, \forall i \in \hat{S}, \hat{\mu}_i=0, \forall i \in [n]\setminus \hat{S},
\end{align*}
where $t>0$ is a scalar and will be specified later.

Plugging solution $(\hat{\bm \Lambda},\hat{\nu}, \hat{\bm \mu} )$ to problem \eqref{eq:dopt_ld}, the objective value with the scalar $t$ satisfies
\begin{align*} 
\hat{z}_R &\le \min_{t}\{\ldet \bm C - \ldet \bm \Lambda -d \log(t) 
+ t\tr\bm \Lambda + t(1+\sigma_{\max})\tr(\bm X \bm \Lambda) -d\}\\
&= \min_{t}\{\ldet \bm C - \ldet \bm \Lambda -d \log(t) 
+ td+ t\sigma_{\max}\tr(\bm X \bm \Lambda) -d\},
\end{align*}
where the equation is from the fact $\tr((\bm I_d +\bm X)\bm \Lambda)=d$.

Minimizing the left-hand side above over $t$, the optimal scalar is $t^* = d/(d+\sigma_{\max} \tr(\bm X \bm \Lambda))$ and thus the final objective value with $t^*$ of problem \eqref{eq:dopt_ld} is equal to
\begin{align*}
z^* \le \hat{z}_R \le &\ldet \bm C + \ldet(\bm{I}_d+\bm{X}) + d \log (1+\sigma_{\max}/d \tr(\bm{X}\bm{\Lambda}) )\\
\le& \ldet \bm C + \ldet(\bm{I}_d+\bm{X}) + d \log \left (1+ \frac{\sigma_{\max}}{d} \frac{s \sigma_{\max}}{1+\sigma_{\max}} \right ),
\end{align*}
where the first inequality is from Proposition \ref{prop:dopt_ld} and the last one is due to inequality \eqref{eq:upper2}. 

Now we use the Lagrangian dual problem \eqref{eq:doptc_ld} of \ref{model_doptc} to establish a complementary approximation bound of the previous one, $d \log(1+({n-s})/d\sigma_{\max}^2/(1+\sigma_{\max}) )$. Because the objective function of \ref{model_doptc} is $\ldet(\bm I_d -\sum_{i\in [n]\setminus \hat{S}} \bm q_i \bm q_i^{\top})$, for any $i \in [n]\setminus \hat{S}$ and $j\in \hat{S}$, the local-optimality condition becomes
	\begin{align*} 
\det(\bm{I}_d + \bm{X}) \ge \det(\bm{I}_d + \bm{X} + \bm{q}_i\bm{q}_i^{\top} -\bm{q}_j\bm{q}_j^{\top}) 	\Longleftrightarrow 
\bm{q}_j^{\top} \bm{\Lambda} \bm{q}_j - \bm{q}_i^{\top} \bm{\Lambda} \bm{q}_j \bm{q}_j^{\top} \bm{\Lambda} \bm{q}_i \ge \bm{q}_i^{\top} \bm{\Lambda} \bm{q}_i (1- \bm{q}_j^{\top} \bm{\Lambda} \bm{q}_j), 
\end{align*}
where we define $\bm X:= -\sum_{\ell\in [n]\setminus \hat{S}} \bm q_{\ell} \bm q_{\ell}^{\top} $ and $\bm \Lambda := (\bm I_d + \bm X)^{-1}$ here.

Then,	by dropping the term $\bm{q}_i^{\top} \bm{\Lambda} \bm{q}_j \bm{q}_j^{\top} \bm{\Lambda} \bm{q}_i$ above, we obtain
\begin{align}\label{ineq2}
- \bm q_{j}^{\top} \bm \Lambda \bm q_j \le - \bm q_i^{\top} \bm \Lambda \bm q_i (1- \bm q_j^{\top} \bm \Lambda \bm q_j) \le - \bm q_i^{\top} \bm \Lambda \bm q_i \frac{1}{1+\sigma_{\max}}, \forall i\in [n]\setminus \hat{S}, \forall j\in \hat{S},
\end{align}
where the second inequality is from by plugging the expressions of $\bm q_j$ and $\bm q_i$, i.e.,
\begin{align*}
&\bm q_j \bm \Lambda \bm q_j = \bm a_j^{\top} (\bm C + \bm A\bm A^{\top})^{-\frac{1}{2}} \left[\bm I_d - (\bm C + \bm A\bm A^{\top})^{-\frac{1}{2}} \bm A_{[n]\setminus \hat{S}} \bm A_{[n]\setminus \hat{S}}^{\top} (\bm C + \bm A\bm A^{\top})^{-\frac{1}{2}} \right]^{-1} \\
&(\bm C + \bm A\bm A^{\top})^{-\frac{1}{2}} \bm a_j= \bm a_j^{\top} \left( \bm C + \bm A\bm A^{\top} - \bm A_{[n]\setminus \hat{S}} \bm A_{[n]\setminus \hat{S}}^{\top} \right)^{-1} \bm a_j \le \bm a_j^{\top} ( \bm C + \bm a_j \bm a_j^{\top})^{-1} \bm a_j \le \frac{ \sigma_{\max}}{1+ \sigma_{\max}}.
\end{align*}

Using inequality \eqref{ineq2}, we can construct a feasible solution to the Lagrangian dual problem \eqref{eq:doptc_ld} as below
\begin{align*}
\hat{\bm \Lambda} = t \bm{\Lambda}, \hat{\nu} = t \frac{1}{1+ \sigma_{\max}} \beta^c_{\min}, 	\hat{\mu}_i =- t \frac{1}{1+ \sigma_{\max}}\bm{q}_i^{\top}\bm{\Lambda} \bm{q}_i -\hat{\nu}, \forall i \in [n]\setminus \hat{S}, \hat{\mu}_i=0, \forall i \in \hat{S},
\end{align*}
where $\beta^c_{\min} := -\max_{i \in[n]\setminus\hat{S}} \bm{q}_i^{\top}\bm{\Lambda} \bm{q}_i$.

Analogous to analyzing the previous approximation bound, we calculate the optimal scalar $t^*= d/[d-\sigma_{\max}/(1+\sigma_{\max})\tr(\bm X \bm \Lambda)]$ and plugging the optimal scalar, the objective value of problem \eqref{eq:doptc_ld} satisfies
\begin{align*}
z^* \le \hat{z}_R &\le \ldet(\bm C+\bm A \bm A^{\top}) + \ldet(\bm{I}_d+\bm{X}) + d \log \left(1- \frac{\sigma_{\max}}{d(1+\sigma_{\max})} \tr(\bm{X}\bm{\Lambda}) \right)\\
&= \ldet(\bm C+\bm A \bm A^{\top}) + \ldet(\bm{I}_d+\bm{X}) + d \log \bigg(1+ \frac{\sigma_{\max}}{d(1+\sigma_{\max})} \sum_{i\in [n]\setminus \hat{S}}\bm{q}_i^{\top}\bm{\Lambda}\bm q_i \bigg)\\
& \le \ldet\bigg(\bm C +\sum_{i \in \hat{S}} \bm a_i \bm a_i^{\top}\bigg) + d \log \left(1+ \frac{(n-s)\sigma^2_{\max}}{d(1+\sigma_{\max})} \right),
\end{align*}
where the last inequality is because for any $i\in [n]\setminus\hat{S}$, we have 
\[\bm q_i \bm \Lambda \bm q_i
= \bm a_i^{\top} \left( \bm C + \bm A\bm A^{\top} - \bm A_{[n]\setminus \hat{S}} \bm A_{[n]\setminus \hat{S}}^{\top} \right)^{-1} \bm a_i \le \bm a_i^{\top} \bm C^{-1} \bm a_i \le \sigma_{\max}.\]

Minimizing the two approximation bounds gives us the symmetric one, $ d \log (1+ \frac{\bar{s}\sigma^2_{\max}}{d(1+\sigma_{\max})} )$.

\item $\bar{s}\log(\bar{s})$-approximation bound.

According to \ref{model_mesp}, the output value of the local-search \Cref{algo:localsearch} becomes
\begin{align*}
z^* \le \hat{z}_M \le & \log f\bigg(\sum_{i\in \hat S} \bm v_i \bm v_i^{\top} \bigg) + \ldet (\bm C) + s\min\left\{\log(s), \log\left(n-s-\frac{n}{s}+2\right)\right\}, 
\end{align*}
where the inequality results from \cite[Theorem 7]*{li2020best}.

Using the complementary M-DDF, we can also show 
\begin{align*}
z^* \le \hat{z}_M^c \le \log f\bigg(\sum_{i\in \hat S} \bm v_i \bm v_i^{\top} \bigg) + \ldet (\bm C) + (n-s)\min\left\{\log(n-s), \log\left(s-\frac{n}{n-s}+2\right)\right\}.
\end{align*}

Taking the minimum of the two bounds above and using the identity $\log f(\sum_{i\in \hat S} \bm v_i \bm v_i^{\top} ) + \ldet (\bm C) =\ldet(\bm C +\sum_{i\in \hat S} \bm a_i \bm a_i^{\top})$ complete the proof. 
		\end{enumerate}
	
\qed
	\end{proof}
	
	\subsection{Proof of \Cref{them:samp} }\label{proof:themsamp}
		\themsamp*
	
To facilitate the analysis of approximation bounds, let us first introduce the elementary symmetric polynomial function and its relation with eigenvalues.
	
	\begin{definition}[elementary symmetric polynomial]
		For a vector $\bm y \in \Re^n$ and an integer $k\in [n]$, let $e_k(\bm y)$ denote the degree $k$ elementary symmetric polynomial, i.e., 
		\[e_k(\bm y) := \sum_{S \in \binom{[n]}{k}} \prod_{i \in S} y_i.\]
	\end{definition}

For any symmetric matrix $\bm X \in \S^n$ with its eigenvalue vector $\bm \lambda\in \Re^n$ and an integer $k\in [n]$, it is well-known that $e_k(\bm \lambda)$ is equal to
\begin{align}\label{eq_poly}
e_k(\bm \lambda) = \sum_{S \in \binom{[n]}{k}} \det(\bm X_{S, S}).
\end{align}

Now we are ready to prove \Cref{them:samp}.
	
	\begin{proof}
		The proof can be split into four parts, corresponding to three approximation bounds and derandomization. 
		
		\noindent\textbf{Part (i)} Let $T:=\supp( \hat{\bm x}_D)$ be the index set of nonzero entries in $\hat{\bm x}_D$. 

		For notational convenience, we define two matrices $\bm X\in \S_+^n$ and $\bm Y\in \S_+^n$ as
		\[ \bm X := \Diag(\hat{\bm x}_D) + \bm Y, \ \ \bm Y:=
		\Diag^{\frac{1}{2}}(\hat{\bm x}_D) \bm{B}^{\top} \bm{B} \Diag^{\frac{1}{2}}(\hat{\bm x}_D). \]
		For each $i\in 	[n]\setminus T$, we can show that the $i$-th column and row vectors of $\bm X$ and $\bm Y$ consist of zeros as $(\hat{x}_D)_i=0$. It follows that both $\bm X$ and $\bm Y$ have rank at most $|T|$. Thus, we let $\lambda_1 \ge \cdots \ge \lambda_{|T|}\geq 0 =\lambda_{|T|+1}=\cdots =\lambda_{n}$ and $\beta_1 \ge \cdots \ge \beta_{|T|}\ge 0 =\beta_{|T|+1}=\cdots =\beta_n$ denote the eigenvalues of $\bm X$ and $\bm Y$, respectively. 
		%
		Furthermore, according to Weyl's inequalities and the fact that $0<x_{\min}\leq 1$, two eigenvalue vectors $\bm \lambda$ and $\bm \beta$ satisfy
		\begin{align}\label{ineq_eign}
		\lambda_{i}\ge x_{\min} + \beta_i \ge x_{\min}(1+\beta_i), \forall i =1,\cdots, |T| .
		\end{align}

		%
		
		
		
		Next, the exponential expectation of the objective value of \ref{model_dopt} is equal to
		\begin{align*}
		\mathbb{E} \bigg[ \det \bigg(\bm{I}_d + \sum_{i \in \tilde{S}_D} \bm{b}_i\bm{b}_i^{\top} \bigg) \bigg] & = \sum_{S \in \binom{[n]}{s} } \mathbb{P}[\tilde{S}_D=S] 
		\det \bigg(\bm{I}_d+ 	\sum_{i \in {S}} \bm{b}_i\bm{b}_i^{\top} \bigg) 
		\\	& = \frac{1}{ \sum_{\bar{S} \in \binom{[n]}{s} } \prod_{i \in \bar{S}} (\hat{x}_D)_i } \sum_{S \in \binom{[n]}{s} } \prod_{i \in S} (\hat{x}_D)_i \det \left(\bm{I}_n+ \bm{B}^{\top} \bm{B} \right)_{S,S}
		\\ & = \frac{1}{	e_s(\hat{\bm x}_D) } \sum_{S \in \binom{[n]}{s} } \det (\bm X_{S,S})
		= \frac{1}{	e_s(\hat{\bm x}_D) } e_{s}(\bm \lambda) = \frac{1}{	e_s(\hat{\bm x}_D) } \lambda_1^s e_{s}\left(\frac{1}{\lambda_1}\bm \lambda\right) 
		\\ & \ge \frac{1}{	e_s(\hat{\bm x}_D) } \lambda_1^s \binom{|T|}{s} e_{|T|}\left(\frac{1}{\lambda_1}\bm \lambda\right) = \frac{1}{	e_s(\hat{\bm x}_D) } \lambda_1^{s-|T|} \binom{|T|}{s} e_{|T|}\left(\bm \lambda\right) 
		\\ & = \frac{1}{	e_s(\hat{\bm x}_D) } \lambda_1^{s-|T|} \binom{|T|}{s} e_{n}\left(\left[{\lambda_1}, \cdots, {\lambda_{|T|}}, 1, \cdots, 1\right]^{\top}\right)
		\\ & \ge \frac{1}{	e_s(\hat{\bm x}_D) } \lambda_1^{s-|T|} \binom{|T|}{s} e_{n}\left(\left[x_{\min}({\beta_1+1}), \cdots, x_{\min}({1+\beta_{|T|}}), 1, \cdots, 1\right]^{\top}\right) 
		\\ & \ge \frac{1}{	\left(\frac{s}{|T|}\right)^s \binom{|T|}{s}} \lambda_1^{s-|T|} \binom{|T|}{s} x_{\min}^{|T|} e_n(\bm 1+\bm \beta) =\left(\frac{|T|}{s}\right)^{s} \lambda_1^{s-|T|} x_{\min}^{|T|} \det(\bm I_n + \bm Y)
		\\ &=\left(\frac{|T|}{s}\right)^{s} \lambda_1^{s-|T|} x_{\min}^{|T|} \det(\bm I_d +	\bm B \Diag(\hat{\bm x}_D) \bm B^{\top})
		\\ & \ge \left(\frac{|T|}{s}\right)^{s} x_{\min}^{n} \lambda_1^{s-|T|} \exp(\hat{z}_R-\ldet \bm C)
		\\ & \ge \left(\frac{|T|}{s}\right)^{s} x_{\min}^{n} \lambda_1^{s-|T|} \exp(z^*-\ldet \bm C),
		\end{align*}
		where the first inequality stems from the fact that each element of vector $\frac{1}{\lambda_1}\bm \lambda$ is no larger than 1, the sixth equality is due to $\lambda_i=0$ for all $i\in[|T|+1, n]$, the second inequality is from \eqref{ineq_eign}, the third inequality is obtained by Maclaurin's inequality, the last equation is from the fact that matrices $\bm Y$ and $\bm B\Diag(\hat{\bm x}_D) \bm B^{\top}$ have the same nonzero eigenvalues, the fourth inequality is because $\hat{\bm x}_D$ is an optimal solution and $x_{\min} \le 1$,
		and the last inequality is due to $\hat{z}_R\ge z^*$.
		
		
		Taking logarithm on both sides of the inequality above, we obtain
		\begin{align*}
		\ldet \bm C+\log	\mathbb{E} \bigg[ \det \bigg(\bm{I}_n+ \sum_{i \in \tilde{S}_D} \bm{b}_i\bm{b}_i^{\top} \bigg) \bigg] & \ge z^* + n\log(x_{\min})+(s-|T|)\log(\lambda_1)-s\log\left(\frac{s}{|T|}\right)\\
		&\ge z^* + n\log(x_{\min}) -(n-s)\log(1+\lambda_{\max}(\bm B^{\top}\bm B )),
		\end{align*}
		where the inequality is because $\lambda_1 \le 1+\beta_1 \le 1+\lambda_{\max}(\bm B^{\top}\bm B)$.

		\noindent\textbf{Part (ii)} The objective value led by the output $\tilde{S}_M$ of \Cref{algo:sampling} with $\hat{\bm x}=\hat{\bm{x}}_M$ is bounded by
		\begin{align*}
		& \log \mathbb{E} \bigg[ f\bigg(\sum_{i \in \tilde{S}_M} \bm{v}_i\bm{v}_i^{\top} \bigg) \bigg] + \ldet (\bm C) \ge 
	\hat{z}_M-s\log\left(\frac{s}{n}\right) - \log \bigg(\binom{n}{s}\bigg),
		\end{align*}
		where the inequality is from \cite[Theorem 5]{li2020best} and using the fact $\hat{z}_M \ge z^*$, we obtain the approximation bound.
		
		\noindent\textbf{Part (iii)} Given the output $\tilde{S}_M^c$ of \Cref{algo:sampling} with $\hat{\bm x}=\hat{\bm{x}}_M^c$, the objective value of Complementary DDF satisfies
		\begin{align*}
		& \log \mathbb{E} \bigg[ f\bigg(\sum_{i \in \tilde{S}_M^c} \bm{v}_i\bm{v}_i^{\top} \bigg) \bigg] + \ldet (\bm C) \ge 
		\hat{z}_M-(n-s)\log\left(\frac{n-s}{n}\right) - \log \bigg(\binom{n}{n-s}\bigg),
		\end{align*}
		where the inequality is from \cite[Theorem 5]{li2020best} and using the fact $\hat{z}_M \ge z^*$, we obtain the approximation bound.
		
		\noindent\textbf{Part (iv)} According the proof of \Cref{them:mesp}, for any cardinality-$s$ subset $S\subseteq [n]$, we must have
		\begin{align*}
		\det \bigg(\bm{C}+\sum_{i \in S} \bm{a}_i\bm{a}_i^{\top}\bigg) = f \bigg(\sum_{i \in [n]} x_i \bm v_i \bm v_i^{\top}\bigg).
		\end{align*}
		The remainder of the proof follows the de-randomization procedure of \cite[Thm.~6]{li2020best}. \qed
	\end{proof}
	
	\section{Supplementary numerical results}\label{app:supnum}
As a supplement of \Cref{table:BC300}, we present in Table \ref{table:suppBC300} detailed numerical results for LP/NLP B\&B with and without submodular cuts. The results expose the quality of the lower bound (LB) and the upper bound (UB) computed by LP/NLP B\&B when the instances are not solved to optimality.

\begin{table}[ht] 
	\centering
	\caption{
 LP/NLP B\&B with and without submodular cuts on IEEE instances} 
	\begin{threeparttable}
		\setlength{\tabcolsep}{4.5pt}\renewcommand{\arraystretch}{1.1}
		\begin{tabular}{ c c|r| r r r r| r r r r}
			\hline 
			\multirow{2}{1.1em}{$n$}	& \multirow{2}{0.5em}{$s$} & \multirow{2}{1.8em}{$z^*$}& \multicolumn{4}{c|}{LP/NLP B\&B} & \multicolumn{4}{c}{LP/NLP B\&B $+$ submod. cuts} 
			\\ \cline{4-11} 
			&	& & \multicolumn{1}{c}{LB} & \multicolumn{1}{c}{UB} & \multicolumn{1}{c}{MIPgap\tnote{1}}& {time\tnote{2} } & \multicolumn{1}{c}{LB} & \multicolumn{1}{c}{UB} & \multicolumn{1}{c}{MIPgap\tnote{1}} & {time\tnote{2} } \\ 
			\hline
			118 & 5	& 80.15& 80.15& 80.15 & 0.00& 11 & 80.15& 80.15 & 0.00&1	\\	
			118 & 10& 156.90& 156.90 & 156.90 & 0.00 & 423 & 156.90& 156.90 & 0.00& 5		\\	
			118 & 15& 231.63& 230.68 & 233.86 & 3.18& - &231.63 & 231.63 & 0.00& 1019 	\\	
			118 & 16& 246.31 &244.65 & 248.92 & 4.27& - &246.31& 246.31 & 0.00& 5368	\\	
			118 & 17& 260.94 &259.77 & 264.96 & 5.19& - & 260.94 & 261.55 & 0.61 & - 	\\	
			118 & 18& 275.56 & 275.01 & 279.82 &4.81 & - & 275.30 & 276.64 & 1.34 & - 	\\	
			118 & 19& 290.15 & 287.34 & 295.82 &8.47 & - &289.53 &	291.69 & 2.16 & - 	\\	
			118 & 20& 304.69 & 300.83 & 311.94 & 11.11 & - & 304.07 &	306.27 & 2.20 & - 	\\	
			300& 35 & 367.49 & 360.53 & 368.29& 7.76& - &367.49& 367.49 & 0.00& 33\\
			300& 40 & 404.02 & 326.62 & 410.14&83.52 & - &404.02& 404.02 & 0.00& 505 \\
			300& 45 & 439.81 & 394.85 & 442.91 & 48.06& - & 439.81 & 439.81 & 0.00& 1334 \\
			300& 50 & 474.49 & 425.78 & 478.45 & 52.67& - & 474.31 & 475.19 & 0.88 & - \\
			300& 51 & 481.24 & 444.43 & 484.84 &40.41 & - & 481.17 & 482.07 & 0.90 & - \\
			300& 52 &487.98 & 364.09 & 504.14 &140.05 & - & 487.64 & 490.08 & 2.44 & - \\
			300& 53 &494.66 & 323.89& 510.02& 186.13 & - & 493.93 & 497.20 &3.27 & - \\
			300& 54 & 501.27 & 345.42 & 514.02 &168.61 & - & 500.36 & 503.91 &3.55 & - \\
			300& 55 & 507.84 &360.22 & 518.89 & 158.67& - & 506.21& 510.81 &4.60 & - \\
			300& 56 &514.37 & 381.03 & 524.84 & 143.81 & - & 512.14& 517.43& 5.29& - \\
			300& 57 & 520.89 &400.04 & 529.63 & 129.59 & - & 519.55 & 524.12 & 4.57& - \\		
			\hline
		\end{tabular}%
		\begin{tablenotes}
 			\item[1] MIPgap = UB $-$ LB
 			\item[2] time in seconds; ``-'': instance not solved within four hours
		\end{tablenotes}
	\end{threeparttable}
	\label{table:suppBC300}
\end{table}

\end{appendices}




\end{document}